\documentclass[11pt,a4paper]{amsart}
\usepackage{a4wide}
\usepackage[latin1]{inputenc}
\usepackage{amsmath, wasysym}
\usepackage{amsfonts}
\usepackage{amssymb}
\usepackage[dvipsnames]{xcolor}

\usepackage{graphicx}
\usepackage{amsthm}
\usepackage{amssymb}
\usepackage{cite}
\usepackage{mathrsfs}
\usepackage[backref=page]{hyperref}
\usepackage{paralist}
\usepackage{enumitem}
\usepackage{tikz}
\usetikzlibrary{matrix,chains}
\usepackage{dynkin-diagrams}

\begin{document}

	\newcommand{\n}{\mathbf{n}}
	\newcommand{\x}{\mathbf{x}}
	\newcommand{\h}{\mathbf{h}}
	\newcommand{\m}{\mathbf{m}}
	\newcommand{\Ph}{\Phi_}
	
	\newcommand{\bN}{\mathbf{N}}
	\newcommand{\bL}{\mathbf{L}}
	\newcommand{\bT}{\mathbf{T}}
	\newcommand{\bG}{\mathbf{G}}
	\newcommand{\bS}{\mathbf{S}}
	\newcommand{\bZ}{\mathbf{Z}}
	\newcommand{\bH}{\mathbf{H}}
	\newcommand{\B}{\mathbf{B}}
	\newcommand{\U}{\mathbf{U}}
	\newcommand{\K}{\mathbf{K}}
	\newcommand{\V}{\mathbf{V}}
	\newcommand{\T}{\mathbf{T}}
	\newcommand{\G}{\mathbf{G}}
	\newcommand{\Para}{\mathbf{P}}
	\newcommand{\Levi}{\mathbf{L}}
	\newcommand{\Y}{\mathbf{Y}}
	\newcommand{\X}{\mathbf{X}}
	\newcommand{\M}{\mathbf{M}}
	\newcommand{\pro}{\mathbf{prod}}
	\renewcommand{\o}{\overline}
	
	\newcommand{\Gtilde}{\mathbf{\tilde{G}}}
	\newcommand{\Ttilde}{\mathbf{\tilde{T}}}
	\newcommand{\Btilde}{\mathbf{\tilde{B}}}
	\newcommand{\Ltilde}{\mathbf{\tilde{L}}}
	\newcommand{\C}{\operatorname{C}}
	
	\newcommand{\N}{\operatorname{N}}
	\newcommand{\bl}{\operatorname{bl}}
	\newcommand{\Z}{\operatorname{Z}}
	\newcommand{\Gal}{\operatorname{Gal}}
	\newcommand{\modulo}{\operatorname{mod}}
	\newcommand{\kernel}{\operatorname{ker}}
	\newcommand{\Irr}{\operatorname{Irr}}
	\newcommand{\D}{\operatorname{D}}
	\newcommand{\I}{\operatorname{I}}
	\newcommand{\GL}{\operatorname{GL}}
	\newcommand{\SL}{\operatorname{SL}}
	\newcommand{\W}{\mathbf{W}}
	\newcommand{\R}{\operatorname{R}}
	\newcommand{\Br}{\operatorname{Br}}
	\newcommand{\Aut}{\operatorname{Aut}}
	\newcommand{\Out}{\operatorname{Out}}
	\newcommand{\End}{\operatorname{End}}
	\newcommand{\Ind}{\operatorname{Ind}}
	\newcommand{\Res}{\operatorname{Res}}
	\newcommand{\br}{\operatorname{br}}
	\newcommand{\Hom}{\operatorname{Hom}}
	\newcommand{\Endo}{\operatorname{End}}
	\newcommand{\Ho}{\operatorname{H}}
	\newcommand{\Tr}{\operatorname{Tr}}
	\newcommand{\opp}{\operatorname{opp}}
	\newcommand{\ssc}{\operatorname{sc}}
	\newcommand{\ad}{\operatorname{ad}}
	\newcommand{\Lin}{\operatorname{Lin}}
	\newcommand{\odd}{\operatorname{odd}}
	\newcommand{\even}{\operatorname{even}}
	
	\newcommand{\cind}{\operatorname{co-ind}}
	\newcommand{\ind}{\operatorname{ind}}
	\newcommand{\CC}{{\mathbb{C}}}	
	\newcommand{\QQ}{{\mathbb{Q}}}	
	\newcommand{\fF}{{\mathfrak F}}
	\newcommand{\la}{{\lambda}}
    \newcommand{\La}{{\Lambda}}
	\newcommand{\al}{{\alpha}}
	\newcommand{\cH}{{\mathcal H}}
	\newcommand{\tw}[1]{{}^#1\!}

    \newcommand{\wt}{\widetilde}
    \newcommand{\galh}{\mathcal{H}}   \newcommand{\gal}{\mathcal{G}}
    \newcommand{\type}{\operatorname}

    \newcommand{\St}{\mathrm{St}}
    \newcommand{\sig}{\langle\sigma\rangle}
    \newcommand{\RTG}{R_{\bT}^\bG}
    \newcommand{\RTGsig}{R_{\bT\sig}^{\bG\sig}}

\newcommand{\tA}{\type{A}}
\newcommand{\tB}{\type{B}}
\newcommand{\tC}{\type{C}}
\newcommand{\tD}{\type{D}}
\newcommand{\tE}{\type{E}}
\newcommand{\tF}{\type{F}}
\newcommand{\tG}{\type{G}}

\newcommand{\subL}{\mathfrak{L}}

    \newcommand{\Syl}{\mathrm{Syl}}

    \theoremstyle{definition}
	\newtheorem{definition}{Definition}[section]
	\newtheorem{notation}[definition]{Notation}
	\newtheorem{construction}[definition]{Construction}
	\newtheorem{remark}[definition]{Remark}
	\newtheorem{example}[definition]{Example}

    \theoremstyle{theorem}
	
	\newtheorem{theorem}[definition]{Theorem}
	\newtheorem{lemma}[definition]{Lemma}
	\newtheorem{question}{Question}
	\newtheorem{corollary}[definition]{Corollary}
	\newtheorem{proposition}[definition]{Proposition}
	\newtheorem{conjecture}[definition]{Conjecture}
	\newtheorem{assumption}[definition]{Assumption}
	\newtheorem{hypothesis}[definition]{Hypothesis}
	\newtheorem{maintheorem}[definition]{Main Theorem}

	\newtheorem{theo}{Theorem}
	\newtheorem{conj}[theo]{Conjecture}
	\newtheorem{cor}[theo]{Corollary}

	\renewcommand{\thetheo}{\Alph{theo}}
	\renewcommand{\theconj}{\Alph{conj}}
	\renewcommand{\thecor}{\Alph{cor}}

	\newcommand{\ov}{\overline }
	
	\def\oo#1{\overline{\overline{#1}}}

    \newcommand{\mandicomment}{\textcolor{blue}}
    \newcommand{\lucascomment}{\textcolor{OliveGreen}}

\title{Galois automorphisms and blocks covering unipotent blocks}

\author{L. Ruhstorfer}
\address[L. Ruhstorfer]{{School of Mathematics and Natural Sciences}, {University of Wuppertal}, {Gau{\ss}str. 20,
		42119 Wuppertal, Germany}}
\email{ruhstorfer@uni-wuppertal.de}
\author{A. A. Schaeffer Fry}
\address[A. A. Schaeffer Fry]{{Department of Mathematics}, {University of Denver}, {Denver, CO 80210, USA}}
\email{mandi.schaefferfry@du.edu}

\thanks{
 The second-named author is partially supported by a CAREER grant through the U.S. National Science Foundation, Award No. DMS-2439897. Some of this research was conducted in the framework of the research
 training group GRK 2240: Algebro-Geometric Methods in Algebra, Arithmetic and Topology,
 funded by the DFG}

\begin{abstract}
In this paper we prove that a recent condition of Lyons--Mart{\'i}nez--Navarro--Tiep, regarding the field of values of extensions of characters in principal blocks, is satisfied for all finite simple groups, which when combined with their results gives a new characterization of finite groups with a normal $\ell$-complement for a prime $\ell$. This leads us to study the distribution of characters in unipotent blocks of disconnected reductive groups and show that this is well-behaved under a generalization of $d$-Harish-Chandra theory. We go on to study the blockwise Galois--McKay (also known as the Alperin--McKay--Navarro) conjecture for the blocks of almost (quasi-)simple groups above unipotent blocks.

\end{abstract}

\maketitle

\begin{center}
 \emph{}
\end{center}

\section{Introduction}

Many interesting problems in the representation theory of finite groups aim to relate information about a group's structure to information that can be obtained from the character table. In particular, several recent problems in this realm require a deep understanding of the fields of values of characters of almost simple groups.

In \cite{nt10}, Navarro and Tiep give a necessary and sufficient condition for a group to have a normal $2$-complement, in terms of the rational-valued irreducible characters of $G$. They also give a corresponding sufficient condition for odd primes. In \cite{lmnt}, Lyons--Mart{\'i}nez--Navarro--Tiep recently considered principal block analogues of those results. In fact, they see that by restricting to the principal block, the version for odd primes should be able to be improved to an if and only if statement. The first main goal of this paper is to complete the proof of that statement.

We write $\QQ_n$ for the cyclotomic field $\QQ(e^{2\pi i/n})$ and, given a prime $\ell$, we write $B_0(G)$ for the principal $\ell$-block of $G$. We complete the proof of the following question considered in \cite{lmnt}:

\begin{theo}\label{thm:main}
Let $\ell$ be any odd prime number and let $G$ be a finite group with order divisible by $\ell$. Then $\ell$ divides $\chi(1)$ for all nonlinear $\QQ_\ell$-valued $\chi\in\Irr(B_0(G))$ if and only if $G$ has a normal $\ell$-complement.
\end{theo}

As mentioned in the last paragraph of the introduction of \cite{lmnt}, this also yields 
 \cite[Conj.~F]{lmnt} for principal blocks. Recall that a block $B$ is called nilpotent if all of its height-zero characters have the same degree, and that for the principal block, this is equivalent to $G$ having a normal $\ell$-complement. We also recall that an almost $\ell$-rational character is one whose values lie in $\QQ_{\ell m}$ for some $m$ relatively prime to $\ell$. Then Theorem \ref{thm:main} gives:

 \begin{cor}\label{thm:maincor}
The principal $\ell$-block $B_0(G)$ is nilpotent if and only if all its height-zero almost $\ell$-rational characters are linear. (That is, \cite[Conj.~F]{lmnt} holds for principal blocks.)
 \end{cor}

We remark that the statement of Theorem \ref{thm:main} also holds for $\ell=2$, and is \cite[Thm.~D]{lmnt}. Further, Theorem \ref{thm:main} is reduced to a problem on simple groups  in \cite[Thm.~E]{lmnt}. Therefore,
our proof of Theorem \ref{thm:main} requires a study of an extendibility condition on simple groups from \cite{lmnt} (see Definition \ref{def:lmnt} below), which is a principal block version of a condition from \cite{nt10}. That is, we consider the ``going up" problem mentioned in the introduction of \cite{lmnt}. In \cite{lmnt}, this condition is proved for nonabelian simple groups $S$ that are not groups of Lie type or that are groups of Lie type defined in characteristic $p=\ell$. For this reason, the current paper focuses on groups of Lie type in nondefining characteristic $p\neq\ell$. 

A key requirement in working with Definition \ref{def:lmnt} below is to be able to determine character extensions that simultaneously preserve desired rationality properties and lie in a prescribed block (here the principal block), which is a deceptively difficult problem. 
While the rationality of characters in the principal block of a group of Lie type $G$ and their distribution into $\ell$-blocks is somewhat well understood, the same cannot be said about the almost (quasi-)simple group $\hat{G}$ having $G$ as a normal subgroup. Here we develop new tools (and adapt established tools) to deal with this question.  
We remark that although (as of the writing of this paper) the Alperin--McKay--Navarro conjecture \cite[Conj.~B]{Na04} has not yet been reduced to conditions on simple groups, this question will also be crucial for working with any future such conditions.

We first study the semisimple characters lying in the principal $\ell$-block of $G$. We show that these characters have an extension that again lie in the principal block. A key step toward this is the following, strengthening the generalized Alvis--Curtis duality introduced by Digne--Michel in \cite{DM94} (we refer the reader to Section \ref{sec:AC} for the relevant notation):

\begin{theo}\label{thm:AC}
    Let $E$ be a group of automorphism of the BN-pair $(G,B,N,S)$. Then the complex $D_{(G)}$ of $G$-$G$-bimodules extends to a complex of $(G \times G^{\mathrm{opp}}) \Delta E$-modules. In particular, if $\ell \nmid |E|$, then the induced complex $\Ind_{(G \times G^{\mathrm{opp}}) \Delta E}^{G E \times (GE)^{\mathrm{opp}}}(D_{(G)})$ induces a splendid derived equivalence between $\Lambda GE$ and itself.
\end{theo}

However, in a small number of cases it can happen that  the trivial character is the unique semisimple character satisfying the conditions of Definition \ref{def:lmnt}.
In these remaining cases, we consider instead unipotent characters in the block.
The fields of values of characters of $\hat{G}$ covering unipotent characters of $G$ were determined recently in \cite{RSST}, but it is not clear in which blocks these covering characters lie. We study this question here, and using this are able to show that in the remaining cases the condition $(\star)$ of Definition \ref{def:lmnt} for $(\QQ_\ell,\ell)$ is satisfied. (It is also worth noting that in some cases the expected unipotent characters cannot be used, see Example \ref{ex:steinberg}, and in those cases the semisimple characters as mentioned above are necessary.)

In fact, this leads us to obtain something stronger and to the second main goal of the paper. From \cite{CE94}, we know that if $\ell$ is an odd prime satisfying some additional reasonable assumptions, then unipotent characters of a finite reductive group $G$ are in the same $\ell$-block $b$ if and only if they are in the same $d$-Harish-Chandra series, where $d$ is the order of $q$ modulo $\ell$. We show that a similar statement can be made in the context of disconnected reductive groups (we refer the reader to Section \ref{sec:disconblocks} for the relevant notation):

\begin{theo}\label{thm:unipotent characters}
Assume that $\ell$ is prime satisfying Definition \ref{def:goodprime} for $\bG$, $[\bG,\bG]$ is simply connected, and $\hat{\G}/\G$ is cyclic of $\{p,\ell \}'$-order. If $(\Levi,\la)$ is a $d$-cuspidal unipotent pair of $(\G,F)$ defining an $\ell$-block $b=b_G(\Levi,\la)$ of $\G^F$ such that $\hat{\G}^F[b]=\hat{\G}^F$, then there exists a Levi subgroup $\hat{\Levi}$ of $\hat{\G}$ such that $\C_{\hat{\G}^F}(\Z(\Levi)_\ell^F)=\hat{\Levi}^F$. Moreover, there exists a bijection
$$\Irr(\hat{\Levi}^F \mid \la) \to \mathrm{Bl}(\hat{\G}^F \mid b)$$ given by $$ \hat{\la} \mapsto b_{\hat{\G}^F}(\hat{\Levi},\hat{\lambda})$$
such that
\begin{enumerate}
    \item $\mathcal{E}(\hat{\G}^F,1) \cap \Irr(\hat{\G}^F,  b_{\hat{\G}^F}(\hat{\Levi},\hat{\lambda}))= \{ \hat{\chi} \in \Irr(\hat{G}) \mid \langle R_{\hat{\Levi} \subset \hat{\Para}}^{\hat{\G}}(\hat{\la}), \hat{\chi} \rangle \neq 0 \}$ and
\item$(1,b_{\hat{\G}^F}(\hat{\Levi},\hat{\lambda})) \leq (\Z(\Levi)_\ell^F, b_{\hat{\Levi}^F}(\hat{\la})))$.
\end{enumerate}
\end{theo}

We anticipate that Theorem \ref{thm:unipotent characters} could be of independent interest. In particular, it may have implications for the inductive Alperin--McKay conditions or eventual similar conditions for the Alperin--McKay--Navarro conjecture. Recall that the Alperin--McKay conjecture posits a bijection between the height-zero characters in an $\ell$-block $b$ of $G$ and those in its Brauer-corresponding block of $\N_G(D)$, where $D$ is a defect group for $b$. The
Alperin--McKay--Navarro conjecture further posits that such a bijection may be chosen to be  $\galh$-equivariant, where $\galh\leq \mathrm{Gal}(\QQ^{\mathrm{ab}}/\QQ)=:\gal$ is the subgroup of those $\sigma\in \gal$ satisfying that there is some $e\geq 0$ such that $\sigma$ acts as the $\ell^e$-power map on all $\ell'$-roots of unity. 

We provide a first look into applications of Theorem \ref{thm:unipotent characters} to the Alperin--McKay--Navarro conjecture in Section \ref{sec:AMN}, where we prove the following:

\begin{theo}\label{thm:AMN}
Let $\G$ be simple, simply connected of classical type but not of type $\tD_4$, and let $\ell$ be an odd prime not dividing $|\Z(G)|$, where $G=\bG^F$ with $F$ a Frobenius map defining $G$ over $\mathbb{F}_q$. Let $(\Levi, \la)$ be a unipotent $d$-cuspidal pair for $G$, where $d$ is the order of $q$ modulo $\ell$. Then there exists a block-preserving, 
$\mathcal{H}$-equivariant bijection
$$\Irr_{0}(GE \mid \mathcal{E}(G,(\Levi,\la))) \to \Irr_{0}(N_{GE}(\Levi) \mid \la),$$
where $E:=E(\G^F)$ is the group of field and graph automorphisms as in \cite[Sec.~2.C]{CS25}.
\end{theo}

Theorem \ref{thm:AMN} is a blockwise analogue of \cite[Thm.~8.8]{RSST}, which we anticipate leading to an analogue of \cite[Thm.~D]{RSST} proving the extension results needed for unipotent blocks in possible future inductive conditions for the Alperin--McKay--Navarro conjecture.

The rest of the paper is structured as follows. In Section \ref{sec:prelim}, we introduce the conditions from \cite{lmnt} and discuss some preliminaries on principal blocks and extensions. In Section \ref{sec:AC}, we prove Theorem \ref{thm:AC} and discuss the generalized Alvis--Curtis duality, which we use in Section \ref{sec:regular} to study the regular and semisimple characters in principal blocks and to introduce the generalized Steinberg character, which we study further in Section \ref{sec:steinberg}. The proof of Theorem \ref{thm:main} is in Section \ref{sec:mainproof}. In Section \ref{sec:disconblocks}, we study the blocks of disconnected reductive groups and prove Theorem \ref{thm:unipotent characters}. Finally, in Section \ref{sec:AMN}, we study implications for the Alperin--McKay--Navarro conjecture and prove Theorem \ref{thm:AMN}.

\section{Preliminaries}\label{sec:prelim}

Most of our notation for blocks and characters is standard. 
We refer the reader to \cite{isaacs, Nav18} for relevant background on ordinary character theory, and to \cite{Nav98} for relevant notation and background on block theory.
For a group $X$ acting on the characters $\Irr(G)$ of a finite group $G$, we will write $X_\theta$ for the stabilizer in $X$ of $\theta\in\Irr(G)$. If $X\leq G$, we will use $\Res_X^G(\theta)$ and $\Ind_X^G(\chi)$, respectively, for the restriction of a character $\theta\in\Irr(G)$ to $X$ and the induced character of $\chi\in\Irr(X)$ to $G$.

\subsection{The conditions of Lyons--Mart{\'i}nez--Navarro--Tiep}\label{sec:conditions}

Here we recall the reduction statement from \cite{lmnt} for Theorem \ref{thm:main}. It requires the following condition on simple groups.

\begin{definition}[Definition 3.1 of \cite{lmnt}]\label{def:lmnt}
Let $\QQ\subseteq\mathbb{K}$ be a field. Assume that $S$ is a nonabelian simple group of order divisible by $\ell$. Then we say $S$ satisfies condition $(\star)$ for $(\mathbb{K},\ell)$ if there there is an $\Aut(S)$-orbit $\mathcal{X}\subseteq\Irr(B_0(S))$ such that
\begin{enumerate}
\item $|\mathcal{X}|$ is not divisible by $\ell$,
\item the characters in $\mathcal{X}$ are nonlinear, of degree not divisible by $\ell$ and $\mathbb{K}$-valued,
\item every $\theta\in\mathcal{X}$ extends to a $\mathbb{K}$-valued character $\hat\theta$ in $B_0(\Aut(S)_\theta)$.
\end{enumerate}
\end{definition}

With this, the following reduces Theorem \ref{thm:main} to a problem on simple groups.

\begin{theorem}[Theorem E of \cite{lmnt}]\label{thm:lmnt}
Let $\ell$ be any odd prime number and assume that every finite nonabelian simple group of order divisible by $\ell$ satisfies condition $(\star)$ of Definition \ref{def:lmnt} for $(\QQ_\ell,\ell)$. Then given a finite group $G$, $\ell$ divides $\chi(1)$ for all nonlinear $\QQ_\ell$-valued $\chi\in\Irr(B_0(G))$ if and only if $G$ has a normal $\ell$-complement.
\end{theorem}

We establish the remaining case of condition $(\star)$ of Definition \ref{def:lmnt} (namely, that of groups of Lie type defined in characteristic $p\neq \ell$) in Theorem \ref{thm:mainlie} below.

\subsection{Initial observations on extensions}\label{sec:preliminaries}

Throughout, we will let $\gal:=\mathrm{Gal}(\QQ^{\mathrm{ab}}/\QQ)$ as in the introduction. We will often use the following statement on gluing extensions. (See also Lemma \ref{lem:gluing blocks} below for a version for blocks.)

\begin{lemma}\label{lem:gluing}
Let $Y$ be a finite group and assume $X_1 \lhd Y$ and $X_2 \leq Y$ are such that  $Y=X_1 X_2$. Assume that $\vartheta_1 \in \Irr(X_1)$ is $Y$-invariant and $\vartheta_2 \in \Irr(X_2)$ is such that $\Res_{X_1 \cap X_2}^{X_2}(\vartheta_2)$ is irreducible and coincides with $\Res_{X_1 \cap X_2}^{X_1}(\vartheta_1)$. Then there exists a unique character $\vartheta \in \Irr(Y)$ which extends both $\vartheta_1$ and $\vartheta_2$. In particular, $\gal_{\vartheta_1}\cap\gal_{\vartheta_2}\leq\gal_\vartheta$.
\end{lemma}
\begin{proof}
For the first statement, see, for example, \cite[Lem.~4.1]{spath10}. (This is also a case of a more general statement by Isaacs---see \cite[Lem.~6.8]{Nav18}.) The last statement follows from the uniqueness of the common extension $\vartheta$.
\end{proof}

\begin{lemma}\label{lem:alpdade}
Let $\ell$ be a prime and let $G$ be a finite group with $N\lhd G$ such that $|G/N|$ is prime to $\ell$. Let $M:=N\C_G(P)$, where $P\in\Syl_\ell(G)$. Assume that $M\lhd G$. 
Then 
\begin{itemize}
    \item Restriction defines a bijection between the pincipal $\ell$-blocks $\Irr(B_0(M))\rightarrow \Irr(B_0(N))$; and
    \item $B_0(G)$ is the unique block of $G$ covering $B_0(M)$.
    \item If further $G\lhd H$ with $H/G$ an $\ell$-group and $M\lhd H$, then $B_0(H)$ is the unique block of $H$ covering $B_0(M)$.
\end{itemize}
\end{lemma}
\begin{proof}
    The first part is the so-called Alperin--Dade correspondence (see \cite{dade77});  the second follows from Brauer's third main and \cite[9.26, 9.20, 9.19]{Nav98}. The third part follows since $H/G$ is an $\ell$-group, so $B_0(H)$ is the unique block covering $B_0(G)$.
\end{proof}

In the situation of Lemma \ref{lem:alpdade}, $M$ is the so-called \emph{Dade Ramification Group} for $B_0(N)$ in $G$ (or $H$). We will in particular apply it when $(N, G, H)=(\bG^F, \bG^FE_{\ell'}, \bG^FE)$ for a finite reductive group $\bG^F$ and its group of graph-field automorphisms $E$ and a Hall $\ell'$-subgroup $E_{\ell'}\leq E$. (In Section \ref{sec:disconblocks}, we will also study the Dade Ramification Group in a more general context.)

\begin{lemma}\label{lem:oddp'}
Keep the situation of Lemma \ref{lem:alpdade} and assume further that $|M/N|$ is odd. Assume $\chi\in\Irr(B_0(N))$ is rational-valued and extends to $G$. Then there is a rational-valued extension of $\chi$ in $\Irr(B_0(M))$.
\end{lemma}
\begin{proof}
    Note that there is some $\wt\chi\in\Irr(B_0(M))$ extending $\chi$ by Lemma \ref{lem:alpdade}, and that this is the unique character in $\Irr(B_0(M))$ above $\chi$. 
    Then since $\Irr(B_0(M))$ is $\gal$-stable and $\chi$ is $\gal$-invariant, it follows that $\wt\chi$ is also $\gal$-invariant, hence rational-valued.
\end{proof}

We will also use the following observation from \cite[Corollary 6.6]{Nav18}:

\begin{lemma}\label{lem:real}
    Let $N \lhd G$ such that $G/N$ is odd. If $\vartheta$ is a $G$-invariant real-valued character, then there exists a unique real-valued character $\chi \in \Irr(G \mid \vartheta)$ extending $\vartheta$. Here $\QQ(\chi)=\QQ(\vartheta)$.
\end{lemma}

Given a block $B$, we will write $\ov B$ for the the block containing the complex conjugates of the characters in $\Irr(B)$. If $\ov B=B$, we say that $B$ is a real block.
In the situation of Lemma \ref{lem:real}, we see that if $G/N$ is abelian, the block $b_G(\chi)$ containing $\chi$ is the unique real block covering $b_N(\vartheta)$.  As is well known, all unipotent characters of groups of classical type are real (even rational). However, as the statement is not necessarily true for the local characters, the following will be useful in Section \ref{sec:AMN} below for dealing with those.

\begin{lemma}\label{lem:real blocks}
\begin{enumerate}
    \item Suppose that $H \leq G$ and $B$ is a block of $H$ such that the induced block $B^G$ is defined. Then $\overline{B}^G$ is defined and $\overline{B}^G=\overline{B^G}$.
    \item Suppose that $N \lhd G$ such that $G/N$ is abelian of odd order. If $B$ is a real block of $N$, then there exists a unique real block of $G$ covering $B$. 
\end{enumerate}
\end{lemma}

\begin{proof}
We use the notation of \cite[Thm.~4.12]{Nav98} and the comments following it. Let $\lambda_B: \Z(FH) \to F$ the $F$-algebra homomorphism associated to $B$. By assumption, the map $\lambda_{B^G}: \Z(FG) \to F$ maping $\sum_{x \in G} c_x x \in \Z(FG)$ to $\sum_{x \in H} c_x \lambda_B(x)$ is an $F$-algebra homomorphism. Now, $\overline{B}$ is the block that contains the characters $\ov{\chi}$ with $\ov{\chi}(g):=\chi(g^{-1})$ and $\chi \in \Irr(B)$. From this it follows that $\lambda_{\ov{B}}(\sum_{y \in H} c_y y)=\lambda_{B}(\sum_{y\in H} c_{y^{-1}} y)$. On the other hand, $$\lambda_{\overline{B^{G}}}(\sum_{x \in G} c_x x)=\lambda_{B^G}(\sum_{x \in G} c_{x^{-1}} x)=\lambda_B(\sum_{x \in H} c_{x^{-1}} x).$$
This shows that $\lambda_{\ov{B}^G}$ coincides with $\lambda_{\overline{B^{G}}}$. Hence, $\overline{B}^G$ is defined and coincides with $\overline{B^G}$.

For the second part, observe that since $G/N$ is abelian, $\Irr(G/N)$ acts transitively on the set of blocks of $G$ covering $B$. Since the action of complex conjugation and $\Irr(G/N)$ are of coprime order, it follows that there is at least one block $b$ covering $B$ fixed by conjugation. If $b'$ is a second block with this property, then we have $b'=b\otimes \la$ with $\la\in\Irr(G/N)$. Note then that $\la \bar{\la}^{-1}=\la^2$ stabilizes $b$. Since $\la$ has odd order, this implies that $\la$ also stabilizes $b$, so $b'=b \otimes \la=b$.
\end{proof}

\section{The Alvis--Curtis duality and automorphisms}\label{sec:AC}

\subsection{Lifting the Cabanes--Rickard equivalence}

Let $p$ be a prime and let $\Lambda$ be any integral domain in which $p$ is invertible.
Suppose that $G$ is a group with a strongly split $BN$-pair of characteristic $p$. Then there is a notion of Alvis-Curtis duality, denoted $D_{(G)}$, inducing a splendid derived equivalence between $\Lambda G$ and itself. (See e.g. \cite[Sec.~4]{CE04} for the construction of $D_{(G)}$ and \cite[Cor.~4.19]{CE04} for the derived equivalence.)

Here we study a generalized version of $D_{(G)}$ and prove Theorem \ref{thm:AC} (restated here), which extends the above-mentioned result. Below we denote by $\Delta E$ the diagonal group in $E\times E^{\mathrm{opp}}$.

\begin{theorem}\label{AC}
    Let $E$ be a group of automorphism of the BN-pair $(G,B,N,S)$. Then the complex $D_{(G)}$ of $G$-$G$-bimodules extends to a complex of $(G \times G^{\mathrm{opp}}) \Delta E$-modules. In particular, if $\ell \nmid |E|$, then the induced complex $\Ind_{(G \times G^{\mathrm{opp}}) \Delta E}^{G E \times (GE)^{\mathrm{opp}}}(D_{(G)})$ induces a splendid derived equivalence between $\Lambda GE$ and itself.  
\end{theorem}

\begin{proof}
Let $ \sigma \in E$.
    We recall the construction of the complex $D_{(G)}$ from \cite[Sec.~4]{CE04}. The complex $D_{(G)}$ has only non-zero terms in degrees $[0,n]$, where $n=|S|$. It's $i$th term is given by 
    $$D_{(G)}^{(i)}:=\displaystyle \bigoplus_{I \subset S : |I|=i} \Lambda G\otimes_{\Lambda P_I} e(U_I) \Lambda G,$$
    where $e(U_I)=\frac{1}{|U_I|} \sum_{ u \in U_I} u$.
Under our assumptions, $\sigma(I)$ is a again a subset of $S$ of the same cardinality. Hence, $\sigma$ has a natural action on 
$$\displaystyle\bigoplus_{ J \in \langle \sigma \rangle I} \Lambda G\otimes_{\Lambda P_J} e(U_J) \Lambda G $$ where $\langle \sigma \rangle I$ denotes the $\sigma$-orbit of $I$. Hence, we can consider each of the terms of $D_{(G)}$ as a $(G \times G^{\mathrm{opp}}) \langle \sigma, \sigma^{-1} \rangle$-module. We must check that the transition maps $\delta^i: D_{(G)}^{(i)} \to D_{(G)}^{(i+1)}$ are also equivariant with respect to this action. For this recall that if $I' \subset I$ then we have a morphism $$f^I_{I'}:\Lambda G\otimes_{\Lambda P_I} e(U_I) \Lambda G \to \Lambda G\otimes_{\Lambda P_{I'}} e(U_{I'}) \Lambda G$$ given by $$ x\otimes_{\Lambda P_I} y \mapsto x \otimes_{\Lambda P_{I'}} y,$$ of $G$-$G$-bimodules.
To define the transition maps we fix a total ordering $``\leq "$ on $S$. In particular, any $I \subset S$ can be written as $I=\{ \alpha_1, \alpha_2,\dots, \alpha_i\}$ such that $\alpha_1 \leq \alpha_2\leq \dots \leq \alpha_i$. For $j \leq i$ we can therefore define $I_j:=I \setminus\{\alpha_j\}$. With these data the transition map is then defined as $$\delta^i(m):=\sum_{j=0}^{i} (-1)^j f_{I_j}^{I}(m)$$
for $m \in \Lambda G\otimes_{\Lambda P_I} e(U_I)\Lambda G$.

In general, the maps $\delta^i$ are not equivariant with respect to the action of $\sigma$ defined above. For this purpose, we define for a given function $\varepsilon: \mathcal{P}(S) \to \{ \pm 1\}$ a new action $``\ast_\varepsilon"$ of $\sigma$ on $D_{(G)}^i$ by setting
$$\sigma \ast_\varepsilon m:=\varepsilon(I) \sigma(m)$$
for $m \in \Lambda G \otimes_{\Lambda P_I} e(U_I) \Lambda G$. 
In order to make the maps $\delta^i$ equivariant with respect to the $``\ast_\varepsilon"$ action of $\sigma$, we must find the signs $\varepsilon(I) \in \{ \pm 1 \}$ such that 
 $$f_{I_j}^I(\varepsilon(I) \sigma(m))=\varepsilon(I \setminus \{\al_j\}) \sigma(f_{I_j}^I(m))$$ 
for all $0 \leq j \leq i$. For this purpose, we set $\delta_I(\al_j):=(-1)^{k-j}$ where $k$ is such that $\sigma(\al_j)=\beta_k$ and $\sigma(I)=\{\beta_1,\dots, \beta_i\}$. In other words, $\delta_I(\al_j)=1$ if and only if the parity between the position of $\al_j$ in $I$ and the position of $\sigma(\al_j)$ in $\sigma(I)$ is even.

With this definition, our claim amounts to showing that $$\varepsilon(I)=\delta_I(\al_j) \varepsilon(I \setminus \{\al_j\}).$$
We set $\varepsilon(\emptyset)=1$ and $\varepsilon(\{a\})=1$ for all $a \in S$. This uniquely determines the function $\varepsilon$ (provided that it exists). In order to show that it exists, we must show that the definition $\varepsilon(I)=\delta_I(\al_j) \varepsilon(I \setminus \{\al_j\})$  is independent of $j$. This amounts to showing that for $k \neq j$ we have $$\delta_I(\al_k) \delta_{I \setminus \{ \al_k\}}(\al_j)= \delta_I(\al_j) \delta_{I \setminus \{ \al_j\}}(\al_k).$$
However, removing $\al_j$ resp. $\sigma(\al_j)$ from $I$ resp. $\sigma(I)$ changes the position mod $2$ of $\al_k$ precisely by one. This implies that $\delta_{I}(\al_k)=\delta_{I \setminus \{\al_k\}}(\al_k)$ (and by symmetry of course $\delta_{I}(\al_j)=\delta_{I \setminus \{\al_j\}}(\al_j)$). Hence, the equality of signs above is always satisfied.

Note that as the function $\varepsilon$ is uniquely determined from $\sigma$, it follows that we can extend the complex $D_{(G)}$ to a complex of $(G \times G^{\mathrm{opp}}) \Delta E$-modules.
(We note that the above action extends to one for $E$. For this, we need to check that $\varepsilon_\tau(\sigma(I)) \varepsilon_\sigma(I)=\varepsilon_{\tau \sigma}(I)$ for each $\tau, \sigma\in E$ for all $I$. By induction on $I$, it suffices to check that the corresponding $\delta$ functions coincide. Write $\delta_{\sigma, I}$, respectively $\delta_{\tau, I}$ for the functions corresponding to $\sigma$, resp. $\tau$. With the notation above, we have 
$\delta_{\sigma,I}(\al_j)=(-1)^{k-j}$ and $\delta_{\tau,\sigma(I)}(\sigma(\al_j))=(-1)^{k-l}$ if $\tau\sigma(\al_j)=\gamma_l$ where $\tau(\sigma(I))=\{\gamma_1, \dots, \gamma_i\}$. However, $\delta_{\tau\sigma,I}(\al_j)=(-1)^{j-k}$, showing the desired equality of the corresponding $\varepsilon$-functions.)
By a result of Marcus \cite{Marcus}, the induced complex $\Ind_{(G \times G^{\mathrm{opp}}) \Delta E}^{G E \times (G E)^{\mathrm{opp}}}(D_{G})$ induces a derived auto-equivalence as desired.
\end{proof}

\begin{remark}\label{different ordering}
Note that if $``\leq_1"$ and $``\leq_2"$ are two total orderings on $S$, then there exists a unique permutation $\pi$ on the set $S$ such that $a \leq_1 b$ if and only if $\pi(a) \leq_2 \pi(b)$. In particular, $\varepsilon_{\leq_1}(I)=\varepsilon_{\leq_2}(\pi(I))$. Hence, if $I$ is a $\sigma$-stable set then we have $\varepsilon_{\leq_1}(I)=\varepsilon_{\leq_2}(I)$ (as the value of $\varepsilon$ on $I$ depends only on the ordering of the elements in $I$).
\end{remark}

In the situation of Theorem \ref{AC}, we write $D_{GE}: \mathbb{Z} \Irr(GE) \to \mathbb{Z} \Irr(GE)$ for the isometry induced by the lifted complex. Note that $\Res_{G}^{GE} \circ D_{GE}=D_{G} \circ \Res_{G}^{GE}$ by construction. Further, if $E'\leq E$, then we have $\Res_{GE'}^{GE} \circ D_{GE}=D_{GE'} \circ \Res_{GE'}^{GE}$.

Recall that the Steinberg character for $G$ is defined as $\St_G:= D_G(1_G) \in \Irr(G)$. We will work often with this character and its generalized version:

\begin{definition}\label{def:genstein}
    With the notation above, the generalized Steinberg character of $G E$ is defined as $\St_{G E}:=D_{GE}(1_{GE}) \in \Irr(GE)$.
\end{definition}

\begin{remark}\label{rem:gensteinrat}
When $E'\leq E$, the observation above regarding $\Res_{GE'}^{GE}$ implies that $\St_{GE'}=\Res_{GE'}^{GE}\St_{GE}$.
Further, note that the extended complex of Theorem \ref{AC} is a complex whose terms are $\Lambda$-modules. Hence, $D_{GE}$ commutes with all Galois automorphisms. In particular, $\St_{GE}$ is a rational character as the dual of a rational character.
\end{remark}

\subsection{Comparison to Digne--Michel's lift}

By a group of Lie type or finite reductive group, we will mean the group $G:=\bG^F$ of fixed points of a connected reductive algebraic group $\bG$ under a Steinberg endomorphism $F\colon \bG\rightarrow \bG$. In this situation, we compare the lift we obtain from Theorem \ref{AC} to the lift of the Alvis--Curtis duality constructed by Digne--Michel \cite{DM94} in the context of finite (disconnected) reductive groups.

\begin{corollary}\label{AC DM}
    If $G=\G^F$ is a group of Lie type and $\sigma$ is a quasi-central automorphism of $\G$ commuting with the action of $F$, then the bijection $D_{G\langle \sigma \rangle}: \mathbb{Z} \Irr(G \langle \sigma \rangle) \to \mathbb{Z} \Irr(G \langle \sigma \rangle)$ defined above coincides with the generalized Alvis--Curtis duality introduced by Digne--Michel in \cite[Prop.~3.13]{DM94}.
\end{corollary}

\begin{proof}
    By the observation following \cite[Cor.~3.12]{DM94}, the Alvis--Curtis duality on the coset $G \sigma $ is given by $$D_{G \sigma}:= \sum_I (-1)^{|I/ \sigma|} {}^\ast R_{L_I \sigma}^{G\sigma} R_{L_I \sigma}^{G \sigma}$$
    where $I$ runs over all $\sigma$-stable subset of $S$. Note that $R_{L_I \sigma}^{G \sigma}: CF(L_I \sigma,K) \to CF(G \sigma,K)$ is by definition the restriction of the functor $R_{L_I \langle\sigma \rangle }^{G \langle \sigma \rangle}$ which is induced by the $G$-$L_I$-bimodule 
    $$\mathbb{C}[G \langle \sigma \rangle/ U_I] \cong \Ind_{(G \times (L_I)^{\mathrm{opp}}) \langle \sigma, \sigma^{-1} \rangle}^{G \langle \sigma \rangle \times L_I \langle \sigma \rangle }(\mathbb{C}[G /U_I]).$$
Now if the subset $I$ is not $\sigma$-stable, we denote by $\langle \sigma \rangle_I$ its stabilizer in $\langle \sigma \rangle$. Then we have $$\Ind^{G \langle \sigma \rangle \times (L_I \langle \sigma \rangle)^{\opp} }_{(G\times (L_I)^{\opp}) \Delta (\langle \sigma \rangle_I \rangle) }(\mathbb{C}[G/U_I]) =\bigoplus_{J \in \langle \sigma \rangle I}(\mathbb{C}[G\langle \sigma \rangle_I/U_J])$$ 
as $G\langle \sigma \rangle \times (L_I \langle \sigma \rangle)^{\opp}$-modules.
Hence, this module induces the functor $R_{L_I \langle \sigma \rangle_I}^{G \langle \sigma \rangle_I} \circ \Res_{L_I \langle \sigma \rangle_I}^{L_I \langle \sigma \rangle}$ on characters. Then for any class function $f$ of $L_I \langle \sigma \rangle$ its image in $G \sigma$ vanishes. This implies that only the $\sigma$-stable orbits contribute to $\Res_{G \sigma}^{G \langle \sigma \rangle} \circ D_{(G)}$.

It remains to check that the signs in \cite{DM94} and in our construction match. For this, we need to show that $\varepsilon(I) (-1)^{|I|}=(-1)^{|I/ \sigma|}$ for all $\sigma$-stable set $I \subset S$. Let us first assume that $I$ is a $\sigma$-orbit. In this case we need to show that $\varepsilon(I)=(-1)^{|I|+1}$. Now by construction if $|I|=r$, then we have for any $x \in I$
$$\varepsilon(I)=\prod_{i=0}^{r-1} \delta_{I \setminus \{x,\sigma(x),\dots,\sigma^{i-1}(x)\}}(\sigma^i(x)).$$
By Remark \ref{different ordering} we can assume that $x \leq \sigma(x) \leq \sigma^2(x) \leq \dots... \leq \sigma^{r-1}(x)$ which implies that $\delta_{I \setminus \{x,\sigma(x),\dots,\sigma^{i-1}(x)\}}(\sigma^i(x))=-1$ for all $i$ so that $\varepsilon(I)=(-1)^{r+1}$ as claimed. 

Let now $I_1,\dots,I_k$ be the orbits of $\sigma$ on $I$. Again by Remark \ref{different ordering}  we can choose an ordering such that $x\leq y$ whenever $x \in I_i$ and $y \in I_j$ with $i < j$. From this we deduce that
$$\varepsilon(I)= \prod_{i=1}^k \varepsilon(I_i)=\prod_{i=1}^k(-1)^{|I_i|+1}$$
so that $\varepsilon(I) (-1)^{|I|}= (-1)^k=(-1)^{|I/\sigma|}$.
\end{proof}

\begin{example}
    Suppose that $G$ is of type $\tA$ and consider the canonical ordering of the simple roots of type $\tA$. Here, the graph automorphism $\sigma$  reverses the ordering of these simple roots. In particular, $\varepsilon(I)=\varepsilon(I \setminus \{\al_j\}) (-1)^{|I|}$ for all $j$. This implies the formula $\varepsilon(I)=1$ if and only if $|I| \equiv 0,1 \mod 4$. Now $|I/\sigma|=1$  if and only if $|I| \equiv 1,2 \mod 4$. This proves the claimed formula $\varepsilon(I) (-1)^{|I|}=(-1)^{|I/ \sigma|}$ in this example.
\end{example}

\begin{remark}
Let $d^1:CF(G \langle \sigma \rangle ,K) \to CF(G \langle \sigma \rangle,K)$ be the decomposition map.
    It follows from \cite[Prop.~2.11]{DM94} that $d^1 \circ D_{G \langle \sigma \rangle}=D_{G \langle \sigma \rangle} \circ d^1$. In particular, from this we see that the Alvis--Curtis duality defined by Digne--Michel is compatible with $\ell$-blocks.
\end{remark}

\section{Regular characters in the principal block}\label{sec:regular}

\subsection{Groups of Lie type}

In this section, we let $\G$ be a simple algebraic group of simply connected type with $F\colon \bG\rightarrow\bG$ a Frobenius endomorphism defining an $\mathbb{F}_q$-structure on $\G$. We denote by $B_0(G)=B_0(\G^F)$ the principal block of $G:=\G^F$. Let $\ell\nmid q$ be a prime dividing $|G|$.

Let $(\bG^\ast, F)$ be dual to $(\bG, F)$ and write $G^\ast:=(\bG^\ast)^F$. Then $\Irr(G)$ is partitioned into rational Lusztig series $\mathcal{E}(G, s)$ indexed by the semisimple elements $s\in G^\ast$, up to $G^\ast$-conjugation. Further, a result of Brou{\'e}--Michel gives that for $s\in G^\ast$ a semisimple element of $\ell'$-order, the union $\mathcal{E}_\ell(G, s):=\cup_{t\in \C_{G^\ast}(s)_\ell} \mathcal{E}(G, st)$ is a union of $\ell$-blocks. For each block $B$ in this union, we have $\Irr(B)\cap \mathcal{E}(G, s)\neq\emptyset$. (See \cite[Thm.~9.12]{CE04}.) 

\subsection{Regular embeddings and regular characters}

Let $\G$ be a simple algebraic group of simply connected type with maximal $F$-stable torus $\T$ contained in an $F$-stable Borel subgroup $\B$ of $\G$. Let $\U$ be the unipotent radical of $\B$ and write $U:=\U^F$.
Consider a regular embedding $\iota:\G \hookrightarrow \Gtilde$ with dual map $\iota^\ast: \Gtilde^\ast \to \G^\ast$ as defined in \cite[Ch.~15]{CE04}. Let $E:=E(\G^F)$ denote the group of field and graph automorphisms as in \cite{CS25}, which act on $\tilde{G}:=\Gtilde^F$. For a power $q_0$ of $p$, we write $F_{q_0}\in E$ for the field automorphism induced by the corresponding standard Frobenius map. {Throughout, we will write $E_\ell$ for a Sylow $\ell$-subgroup in $E$ and $E_{\ell'}$ for a Hall $\ell'$-subgroup in $E$.} Moreover as in \cite{CS25}, we let $E^\ast$ be the group of graph and field automorphisms in $\tilde{G}^\ast:=(\Gtilde^\ast)^F$.

We denote by $\Gamma=\Ind_U^{\tilde{G}}(\phi)$ the unique ordinary Gelfand--Graev character of $\tilde{G}$, which is multiplicity-free by \cite[Rem.~3.4.18]{GM20}. {We recall that a regular character of $\wt{G}$ is one that appears as a constituent of $\Gamma$, and a semisimple character is a constituent of $D_{(\wt G)}(\Gamma)$. In particular, each series $\mathcal{E}(\wt{G}, \wt{s})$ contains exactly one semisimple character and one regular character.}

Now, recall that since $\ell \nmid |U|$ it follows that $\Gamma$ is a projective character of $\tilde{G}$. For $\tilde{s} \in \wt{G}^\ast$ a semisimple element of $\ell'$-order we denote by $\Gamma_{\tilde{s}}$ the projection of $\Gamma$ onto $\mathcal{E}_\ell(\tilde{G},\tilde{s})$ so that $\Gamma$ is by construction the sum of the $\Gamma_{\tilde{s}}$. It is a result by G. Hiss \cite[Thm~3.2]{hiss} that $\Gamma_{\tilde{s}}$ is in fact an indecomposable projective module. Further, we can choose the linear character $\phi$ to be $E_{\ell'}$-stable (see \cite[Def.~3.3]{Spa12}, see also \cite[Sec.~3]{Joh22b} for the case of exceptional graph automorphisms), and we denote by $\hat{\phi}$ its trivial extension to $U E_{\ell'}$. Also denote $\hat{\Gamma}:=\Ind_{UE_{\ell'}}^{\tilde{G} E_{\ell'}}(\hat{\phi})$, which is again a projective $\hat{G}:=\tilde{G} E_{\ell'}$-module and $\hat{\Gamma}_{\tilde{s}}$ the projection of $\hat{\Gamma}$ onto the sum of blocks of $\hat{G}$ covering the blocks in $\mathcal{E}_\ell(\wt{G}, \wt{s})$.

The following generalizes \cite[Cors.~3.3, 3.4]{hiss}.

\begin{lemma}\label{lem:GGRblock}
The character $\hat{\Gamma}_{\tilde{s}}$ is a projective indecomposable character of $\tilde{G} E_{\ell'}$ and thus all its irreducible constituents lie in the same block. In particular, the irreducible constituents of $D_{ \tilde{G} E_{\ell'} }(\hat{\Gamma}_{\tilde{s}})$ all lie in the same $\ell$-block of $\tilde{G} E_{\ell'}$.
\end{lemma}

\begin{proof}
    Note that $\Res_{\tilde{G}}^{\hat{G}}(\hat{\Gamma})=\Gamma$ by Frobenius reciprocity, which implies that $\Res^{\hat{G}}_{\tilde{G}}(\hat{\Gamma}_{\tilde{s}})=\Gamma_{\tilde{s}}$. As $\Gamma_{\tilde{s}}$ is a projective indecomposable module and $\hat{\Gamma}_{\tilde{s}}$ is a projective module (being the projection of a projective module to a sum of blocks), it follows that $\hat{\Gamma}_{\tilde{s}}$ is indecomposable as well. Hence, all its irreducible constituents lie in the same $\ell$-block. By Theorem \ref{AC}, it follows that all constituents of $D_{ \tilde{G} E_{\ell'}}(\hat{\Gamma}_{\tilde{s}})$ also lie in the same $\ell$-block. 
\end{proof}

It is a useful fact that automorphisms of finite reductive groups can be realized as automorphisms of reductive groups in the following sense:

\begin{lemma}\label{lem: model}
    There exists a connected reductive group $\underline{\G}$ with Frobenius endomorphism $\underline{F}$ defining an $\mathbb{F}_p$-structure together with a morphism $\mathrm{pr}: \underline{\G} \to \G$ such that $\underline{\G}^{\underline{F}} \cong \G^F$ via this morphism. Moreover, there is a set $\underline{E}$ of $\underline{F}$-equivariant quasi-central automorphisms of $\underline{\G}$ with $\underline{E} \cong E$ via $\mathrm{pr}$. In addition, if $\mathbf{H}$ is an $F$-stable subgroup of $\G$, then there exists an $\underline{F}$-stable subgroup $\underline{\mathbf{H}}$ of $\underline{\G}$ such that $\mathrm{pr}(\underline{\mathbf{H}})=\mathbf{H}$. 
\end{lemma}

\begin{proof}
This follows from the discussion in \cite[Sec.~2.1]{RSF22}.
\end{proof}

\begin{lemma}\label{lem:ACGGRprincblock}
     All irreducible constituent of $D_{\tilde{G} E_{\ell'}}(\hat{\Gamma}_{1})$ lie in the principal block of $\tilde{G} E_{\ell'}$.
\end{lemma}

\begin{proof}

From Lemma \ref{lem:GGRblock}, all irreducible constituents of $\hat\Gamma_1$ lie in the same $\ell$-block, and
the same is true for the irreducible constituents of $D_{\tilde{G} E_{\ell'}}(\hat{\Gamma}_{1})$. Hence, in order to show that all its irreducible constituents lie in the principal block, it suffices to show that one of them lies in the principal block. Our aim is therefore to show that the trivial character is an irreducible constituent of $D_{\tilde{G} E_{\ell'}}(\hat{\Gamma}_{1})$. For this, it suffices to check that the generalized Steinberg character
$\St_{\wt G E_{\ell'}}:=D_{\wt G{E_{\ell'}}}(1_{\wt{G}{E_{\ell'}}})$ appears in the generalized Gelfand--Graev character $\hat{\Gamma}_{1}$.

{Let $\wt T=\wt{\mathbf{T}}^F$ and $\wt{B}=\wt{\mathbf{B}}^F$ be a maximally split torus and Borel subgroup of $\wt G$ containing $T$ and $B$.}
By Corollary \ref{AC DM} and Lemma \ref{lem: model}, we can use \cite[Proposition 3.17]{DM94}, which shows that {given $\sigma\in E_{\ell'}$}, on the coset
$\wt B^F \sigma$ the Steinberg character is $\Ind_{\wt T \sigma}^{\wt B\sigma} (1_{\wt T \sigma})$. Hence, $\Res^{\wt G{E_{\ell'}}}_{\wt B{E_{\ell'}}}(\St_{\wt G{E_{\ell'}}})=\Ind_{\wt T{E_{\ell'}}}^{\wt B {E_{\ell'}}}(1_{\wt T {E_{\ell'}}})$. Hence, the scalar product 
$$\langle \St_{\wt G{E_{\ell'}}}, \hat{\Gamma}_{1} \rangle= \langle \Ind_{\wt T{E_{\ell'}}}^{\wt B{E_{\ell'}}}(1_{\wt T{E_{\ell'}}}), \Ind_{U{E_{\ell'}}}^{\wt B{E_{\ell'}}}(\hat{\phi}) \rangle.$$ 
Now by the Mackey formula, $\Res_{\wt T{E_{\ell'}}}^{\wt B{E_{\ell'}}} \Ind_{U{E_{\ell'}}}^{\wt B{E_{\ell'}}}(\hat{\phi}) = \Ind_{E_{\ell'}}^{\wt T{E_{\ell'}}} \Res_{{E_{\ell'}}}^{U{E_{\ell'}}}(\hat{\phi})=\Ind_{{E_{\ell'}}}^{\wt T{E_{\ell'}}}(1_{E_{\ell'}})$, so that
$$\langle \Ind_{\wt T{E_{\ell'}}}^{\wt B{E_{\ell'}}}(1_{\wt T{E_{\ell'}}}), \Ind_{U{E_{\ell'}}}^{\wt B{E_{\ell'}}}(\hat{\phi}) \rangle=\langle 1_{\wt T{E_{\ell'}}}, \Ind_{{E_{\ell'}}}^{\wt T{E_{\ell'}}}(1_{E_{\ell'}}) \rangle= 1$$
by applying Frobenius reciprocity twice. Hence the generalized Steinberg character appears as a constituent of $\hat\Gamma_1$, as desired, completing the proof. 
\end{proof}

Throughout, we will use the fact that the Galois group $\gal$ acts on the set of series $\mathcal{E}(\bG^F, s)$ (see e.g. \cite[Prop.~3.3.15]{GM20}). Namely, if $\gamma\in\gal$ maps $e^{2\pi i/|s|}$ to its $r$th power, then $\gamma$ sends $\mathcal{E}(\G^F,s)$ to $\mathcal{E}(\G^F, s^r)$. Given a semisimple element $s\in \G^\ast$, we can then define $\QQ_s$ as in \cite{SFT23}. That is, $\QQ_s=(\QQ^{\mathrm{ab}})^{\gal_s}\subseteq\QQ(e^{2\pi i/|s|})$ is the fixed field of the stabilizer $\gal_s\leq \gal$ of the series $\mathcal{E}(\bG^F, s)$.

\begin{lemma}\label{lem:mult1}
    Let $\chi \in \Irr(B_0(\tilde{G}))$ with $\langle \chi, D_{\tilde{G}}(\Gamma) \rangle=\pm 1$ and $\Z(\tilde{G}) \subset \mathrm{ker}(\chi)$. Denote $E':=E_\chi \cap  E_{\ell'}$.  Then there exists a unique extension $\hat{\chi} \in \Irr(\tilde{G} E')$ of $\chi$ such that $\langle \hat{\chi}, D_{ \tilde{G}E'}(\hat{\Gamma}) \rangle = \pm 1$. Moreover, $\mathbb{Q}(\hat{\chi})=\mathbb{Q}(\chi)=\mathbb{Q}_{\tilde{s}}$.
\end{lemma}

\begin{proof}
First, we note that an extension exists by the proof of \cite[Prop.~3.4]{Spa12}.  The uniqueness of the extension  is clear from Frobenius reciprocity, together with the fact that Alvis--Curtis duality is an isometry on the set of virtual characters. This, together with the fact that $D_{\tilde{G} E_{\ell'}}$ commutes with all Galois automorphisms by construction, implies that if $\sigma\in\gal$ fixes $\chi$ and $\hat{\Gamma}$, then it also fixes $\hat{\chi}$. 

Now, as all non-degenerate characters of $U$ are $\tilde{T}$-conjugate, it follows that for each $\sigma\in\gal$, there is some $\tilde{t} \in \tilde{T}$ such that $\sigma \tilde{t}$ fixes $\phi$. {Hence $\Gamma^\sigma=\Gamma$.} Further,  $[E_{\ell'}, \tilde{t}] \subset \Z(\tilde{G})$, since the character $\phi$ has stabilizer $U \Z(\tilde{G})$ in $\tilde{B}$ and is $E_{\ell'}$-stable, so $[E_{\ell'}, \tilde{t}]  \in \tilde{T} \cap U \Z(\tilde{G})=\Z(\tilde{G})$. Now as $\chi$ is trivial on $\Z(\tilde{G})$ and Alvis-Curtis duality preserves underlying central characters, we deduce that $\mathbb{Q}(\chi)=\mathbb{Q}(\hat{\chi})$. (Indeed, for $\theta \in \Irr(\Z(\tilde{G}))$ we may write $\Gamma_\theta=\Ind_{U \Z(\tilde{G})}^{\tilde{G}}(\hat{\phi} \theta)$ so that $\Gamma=\sum_{\theta \in \Irr(\Z(\tilde{G}))} \Gamma_\theta$ and note that $D_{\wt G}(\chi)$ appears in the summand with $\theta= 1_{\Z(\tilde{G})}$. This summand $\Gamma_1$ has an extension $\hat{\Gamma}_1:=\Ind_{U \Z(\tilde{G}) E_{\ell'}}^{\tilde{G} E_{\ell'}} ( \hat{\phi} 1_{\Z(\tilde{G})})$ which contains $D_{\wt GE_{\ell'}}(\hat{\chi})$. It suffices to check that the character $\hat{\phi}':=\hat{\phi} 1_{\Z(\tilde{G})}$ is $\sigma \tilde{t}$-stable. However, for $e \in E_{\ell'}$ we have $\hat{\phi}'^{\sigma \tilde{t}}(e)=\hat{\phi}(e)\hat{\phi}'([e,\tilde{t}])^\sigma=\hat{\phi}(e)$, as $\hat{\phi}'$ is trivial on $[e,  \tilde{t}] \in \Z(\tilde{G})$.)

The rest follows from the fact that Jordan decomposition in the connected center case is compatible with Galois automorphisms by the main theorem of \cite{SV20}.
\end{proof}

\subsection{Height zero characters in the principal block}

We keep the situation from the previous section and now assume further that $S=G/\Z(G)$ is a non-abelian simple group.

\begin{lemma}\label{lem:bijsimplediags}
    Assume that $\ell \nmid |\Z(G)|$. Then restriction defines a bijection between $$\Irr(B_0(\tilde{G}/\Z(\tilde{G})) \to \Irr(B_0(G/\Z(G))).$$ In this situation, every character $\chi \in \Irr(B_0(S))$ has a unique extension $\tilde{\chi} \in \Irr(B_0( \tilde{G}/\Z(\tilde{G}))$, and this satisfies $(\tilde{G}E)_{\tilde{\chi}}=(\tilde{G} E)_\chi$.
\end{lemma}

\begin{proof}
    This follows from \cite[Thm.~17.7]{CE04}. To see the final part of the statement, we remark that if $\sigma\in E_\chi\setminus{E_{\wt\chi}}$, then $\wt\chi^\sigma=\wt\chi\beta$ for some linear character $\beta$, whose order is prime to $\ell$ by our assumption $\ell\nmid |\Z(G)|$, and hence $\ell\nmid |\wt{G}/G\Z(\wt{G})|$. But this would contradict that $\wt\chi$, and hence $\wt\chi^\sigma$, must lie in a series indexed by a semisimple $\ell$-element by \cite[Thm.~9.12]{CE04}. 
\end{proof}

Recall that if $\mathcal{E}(\tilde{G},\tilde{s})$ is a Lusztig series, then all its characters restrict to a multiple of the same character of $\Z(\tilde{G})$. (See \cite[Lem.~2.2]{Ma07}.) In particular, these characters restrict trivially to $\Z(\tilde{G})$ if and only if $\tilde{s} \in [\tilde{G}^\ast, \tilde{G}^\ast]$. (See \cite[Lem.~4.4, Rem.~4.6]{NT13}.) Now, by the degree properties of Jordan decomposition, the semisimple character in the Lusztig series $\mathcal{E}(\tilde{G},\tilde{s})$  is of $\ell'$-degree if and only if $\C_{\tilde{G}^\ast}(\tilde{s})$ contains a Sylow $\ell$-subgroup of $\tilde{G}^\ast$. 

In particular, if $\wt{s}$ has $\ell$-power order, this means the semisimple character is of $\ell'$-degree if and only if  $\tilde{s} \in \Z(P^\ast)$ for some Sylow $\ell$-subgroup $P^\ast$ of $\tilde{G}^\ast$. Note that $P_0^\ast:=P^\ast \cap [\tilde{G}^\ast,\tilde{G}^\ast]$ is a Sylow $\ell$-subgroup of $[\tilde{G}^\ast,\tilde{G}^\ast]$. 

\begin{lemma}\label{lem:semisimple}
    With the notation above, suppose that $ \tilde{s} \in P_0^\ast \cap \Z(P^\ast)$. Then the semisimple character $\tilde{\chi} \in \mathcal{E}(\tilde{G},\tilde{s})$ lies in $\Irr_0(B_0(\tilde{G}/\Z(\tilde{G})))$. Moreover, if $\tilde{s}$ has order $\ell$, then $\mathbb{Q}(\tilde{\chi}) \subseteq \mathbb{Q}_\ell$ and if $\ell$ does not divide the order of any graph automorphisms in $E$, then $[E:E_{\tilde{\chi}}]$ is not divisible by $\ell$. 
\end{lemma}

\begin{proof}
    That $\wt\chi$ lies in $\Irr_{\ell'}(\wt{G}/\Z(\wt{G}))$ follows from the discussion above. Further, it lies in the principal block by \cite[Cor.~3.4]{hiss}.
    
    It remains to prove the moreover statement. Now since $\tilde{s}$ has order $\ell$ we deduce that $\mathbb{Q}_{\tilde{s}} \subseteq \mathbb{Q}_{o(\tilde{s})}=\mathbb{Q}_\ell$. On the other hand, $\sigma \in E$ stabilizes $\wt\chi$ if and only if $\sigma^\ast$ stabilizes the geometric conjugacy class of $\tilde{s}$, where $\sigma^\ast\in E((\bG^\ast)^F)$ is dual to $\sigma$ (see \cite[Cor.~2.5]{NTT}). 
    Now, $\tilde{s}$ is $\wt \G^\ast$-conjugate to an element $t$ in the maximally split torus $\T^\ast$. Hence, we have $F_p^\ast(t)=t^p$. In particular, $F_p$ acts as an automorphism on the group $\langle t \rangle \cong C_\ell$. Since $\mathrm{Aut}(C_\ell) \cong C_{\ell-1}$ it follows that any automorphism of $\ell$-power order acts trivially on $\langle t \rangle$. Hence, $[\langle F_p \rangle: \langle F_p \rangle_{\wt\chi}]$ is not divisible by $\ell$. Finally, notice that if $\ell$ does not divide the order of any graph automorphisms, then $|E:\langle F_p \rangle|$ is also prime to $\ell$.
\end{proof}

Under our assumptions, note that $|[\tilde{G}^\ast,\tilde{G}^\ast]|_\ell=|G|_\ell$ and $|[\Z(\tilde{G}^\ast,\tilde{G}^\ast])|_\ell=|\Z(G)|_\ell$. In particular, if $\ell \mid |S|$ it follows that the Sylow $\ell$-subgroup $P_0^\ast$ is non-central, so in this case there always exists a nontrivial element in $ \Z(P_0^\ast)$ of order $\ell$. Given Lemma \ref{lem:semisimple}, we are interested in determining convenient elements of $\Z(P^\ast) \cap P_0^\ast$. Note that if $\ell \nmid |\Z(G)|=|\Z([\tilde{G}^\ast,\tilde{G}^\ast])|$, we have $P^\ast=P_0^\ast \Z(\tilde{G}^\ast)_\ell$ which gives $\Z(P^\ast) \cap P_0^\ast=\Z(P_0^\ast)$ in this situation.

\subsection{Some exceptional cases}

Throughout, we will use the common notation $\SL_n(\epsilon q)$, with $\epsilon\in\{\pm1\}$, to denote $\SL_n(q)$ if $\epsilon=1$ and $\operatorname{SU}_n(q)$ if $\epsilon=-1$, and similar for related groups $\GL_n(\epsilon q)$, etc. Similarly, $\tE_6(\epsilon q)$ will denote $\tE_6(q)$ if $\epsilon=1$ and the twisted version $\tw{2}\tE_6(q)$ if $\epsilon=-1$. 

We keep the notation from above. The following is also shown in \cite[Thm.~3.5]{GSV}.
\begin{lemma}\label{lem:typeAZ(P)}
  Suppose that $G=\SL_n(\epsilon q)$ and $2 \neq \ell \mid |\Z(G)|$. Then $\Z(P^\ast) \cap P_0^\ast=\Z(P_0^\ast)$. 
\end{lemma}

\begin{proof}
Note here that $\ell\mid (q-\epsilon)$. As long as $P_0^\ast$ is Cabanes (that is, that $P_0^\ast$ contains a unique maximal abelian normal subgroup $D_0^\ast$)  we know that $\Z(P_0^\ast)$ is contained in the Cabanes subgroup $D_0^\ast$ of $P_0^\ast$, which in this case is the Sylow $\ell$-subgroup $T^\ast_\ell$ of the diagonal torus $T^\ast$ of  $[\wt{G}^\ast, \wt{G}^\ast]\cong G$. Now $P^\ast/D^\ast\cong P_0^\ast/D_0^\ast$, where $D^\ast$ is the Cabanes group of $P^\ast$,  
so an element in the Cabanes subgroup $D_0^\ast$ is central in $P_0^\ast$ if and only if it is central in $P^\ast$. We know that $P_0^\ast$ is Cabanes except when $n=3=\ell \mid \mid (q -\epsilon)$ (see \cite[Lem.~3.12]{KM15} and its proof). The exceptional case is treated immediately.
\end{proof}

Unfortunately, as the following example shows, $\Z(P_0^\ast)$ can sometimes be too small.

\begin{example}
    Suppose that $\ell=3$ and $G=\SL_3(q)$ with $\ell=3 \mid (q-1)$. Then $\Z(P_0^\ast)=\Z([\wt G^\ast, \wt G^\ast])$. In particular, any semisimple character in a Lusztig series associated to a semsimple element in $\Z(P_0^\ast)$ will cover the trivial character of $G$.
\end{example}

However, we see that we can avoid this as long as $n$ is not a power of $\ell$:

\begin{lemma}\label{lem:typeAZP}
    Suppose that $G=\SL_{n}(\epsilon q)$. Assume that $\ell$ is an odd prime dividing $|\Z(G)|$ but 
that $n$ is not a power of $\ell$. Then $\Z(P_0^\ast) \neq \Z([\tilde{G}^\ast,\tilde{G}^\ast])_\ell$. Further, there exists $\tilde{s} \in \Z(P_0^\ast)$ of order $\ell$ such that the semisimple character $\tilde{\chi} \in \mathcal{E}(\tilde{G},\tilde{s})$ restricts irreducibly to a nontrivial character $\chi:=\Res_{G}^{\tilde{G}}(\tilde{\chi})$ with $(\tilde{G} E)_{\wt\chi}=(\tilde{G} E)_\chi$.
\end{lemma}

\begin{proof}
We have here $\wt{G}=\GL_n(\epsilon q)=\wt{G}^\ast$.
    Let $n=\sum_{k} \ell^k a_k$  be the $\ell$-adic representation, and suppose $i$ is minimal such that $a_i \neq 0$. Note that $a_i > 0$ since $\ell \mid n$. Let $\la\in\mathbb{F}_{q^2}^\times$ have order $\ell$. If $n=\ell^i a_i$ (and hence $a_i > 1$ by our assumption that $n$ is not a power of $\ell$) then we consider the matrix $\tilde{s}_0:=\mathrm{diag}(1,\dots,1,\la,\dots,\la)\in \wt{G}^\ast$, where the first $n-\ell^i=\ell^i (a_i-1)$ entries are one and the last $\ell^i$ entries are $\la$. If instead the $\ell$-adic expansion of $n$ has at least two terms, we consider the matrix $\tilde{s}_0:=\mathrm{diag}(1,\dots,1,\la,\dots,\la)$ where  now the first $n-\ell^i a_i$ entries are one and the last $\ell^ia_i$ entries are $\la$. Note that in either case, $\wt{s}_0$ is in $\Z(P^\ast_0)$ but not central in $[\tilde{G}^\ast,\tilde{G}^\ast]$. 

Now observe that the image of $s_0:=\iota^\ast(\tilde{s}_0)$ is again an $\ell$-element. Let us first suppose that $n \neq 2 \ell^i$. In this case the Levi subgroup $\C_{\G^\ast}(s_0)$ is self-normalizing in $\G^\ast$. Note that the centralizer of $s_0$ is connected and therefore the semisimple character $\tilde{\chi} \in \mathcal{E}(\tilde{G},\tilde{s}_0)$ restricts to an irreducible character of $G$. Moreover, $\tilde{\chi}$ is the unique character in its Lusztig series covering $\chi$. Hence, for the last claim is suffices to show that if $\sigma\in E$ such that $s_0$ and $\sigma^\ast(s_0)$ are $G^\ast$-conjugate, then also $\tilde{s}_0$ and $\sigma^\ast(\tilde{s_0})$ are $\tilde{G}^\ast$-conjugate. For this, we can assume that $\sigma^\ast$ stabilizes the diagonal torus $\tilde{T}^\ast$ containing $\tilde{s}_0$ and acts as $\sigma^\ast(t)=t^{-1}$ if $\sigma$ is a graph automorphism and $\sigma^\ast(t)=t^{p_0}$, for some power $p_0$ of $p$, if $\sigma^\ast$ is a field automorphism. It follows that $\sigma^\ast(s_0)={}^n s_0 z$ for some $z \in \Z(\tilde{\G}^\ast)$ and $n \in \Gtilde^\ast$. Hence, $n \in \N_{\Gtilde^\ast}(\C_{\Gtilde^\ast}(\tilde{s}_0))=\C_{\Gtilde^\ast}(\tilde{s}_0)$ which implies $\sigma^{\ast}(\tilde{s}_0)=\tilde{s}_0$ by considering the structure of $\tilde{s}_0$.

Therefore, it only remains to consider the case where $n=2\ell^i$. In this case we consider the element $\tilde{s}:=\mathrm{diag}(\la,\dots,\la,\la^{-1},\dots,\la^{-1})$ with $\ell^i$ entries equal to each of $\la$ and $\la^{-1}$. Note that the $\tilde{\G}^\ast$-conjugacy class of $\tilde{s}$ is stable under the non-trivial graph automorphism $\gamma^\ast$ so that $\tilde{\chi}$ is $\gamma$-stable. Notice that if $\sigma\in E$ stabilizes $\chi$, then $\wt\chi^\sigma$ is stabilized by $\gamma$.  However, as $\gamma$ acts by inversion on $\tilde{G}/G$, it follows that $\tilde{\chi}$ is the unique $\gamma$-stable character covering $\chi$ lying in a Lusztig series corresponding to a semisimple element of order $\ell$. It follows that $\chi$ and $\tilde{\chi}$ again have the same stabilizer in $E$.
\end{proof}

We next treat the case $G=\tD_4(q)$.

\begin{lemma}\label{D4}
    Suppose that $G=\tD_4(q)$ and assume $\ell$ is odd. Then there is some $ \tilde{s} \in P_0^\ast \cap \Z(P^\ast)$ that is noncentral in $\wt{G}^\ast$, such that the semisimple character $\tilde{\chi} \in \mathcal{E}(\tilde{G},\tilde{s})$ lies in $\Irr_0(B_0(\tilde{G}/\Z(\tilde{G})))$,  $\mathbb{Q}(\tilde{\chi}) \subseteq \mathbb{Q}_\ell$, and $\ell \nmid [E:E_{\tilde{\chi}}]$. {Further, we have $1_G\neq\Res_G^{\wt{G}}\wt\chi\in\Irr(B_0(G/\Z(G)))$.}
\end{lemma}

\begin{proof}
Since $\ell\neq 2$, we have $\ell\nmid|\Z(G)|$ and the final statement follows from the first, together with Lemma \ref{lem:bijsimplediags} and that $\wt{s}$ is not central in $\wt{G}^\ast$. Further, by Lemma \ref{lem:semisimple}, we may assume that $\ell$ divides the order of some graph automorphism, 
 so that $\ell=3$. In particular, $\ell \nmid |\Z([\tilde{G}^\ast,\tilde{G}^\ast])|$. 
 
 In this case, let $\sigma$ be a graph automorphism  of order $3$ with dual $\sigma^\ast$. By the arguments in Lemma \ref{lem:semisimple}, it suffices to find some $\tilde{s} \in \Z(P_0^\ast)$ such that $\tilde{s}$ is $\sigma^\ast$-stable.
For this, let $\hat{P}$ be a Sylow $3$-subgroup of $[\tilde{G}^\ast,\tilde{G}^\ast] \langle \sigma^\ast \rangle$. If there is $1 \neq \tilde{s}\in\Z(\hat{P}) \cap [\wt G^\ast,\wt G^\ast]$ then the corresponding Lusztig series is $\sigma^\ast$-stable. Otherwise, $\Z(\hat{P})$ has nontrivial intersection with the coset $[\tilde{G}^\ast,\tilde{G}^\ast] \sigma^\ast$, which implies that $\Z(P_0^\ast) \leq \Z(\hat{P})$. In that case, any $1 \neq \tilde{s} \in \Z(P_0^\ast)$ is $\sigma^\ast$-stable.
\end{proof}

\section{The generalized Steinberg character and the principal block}\label{sec:steinberg}

\subsection{Observations on the Dade group for twisted groups}

Notice that if $G=\bG^F$  is quasisimple and $\ell \mid |Z(G)|$,  
then either $\ell=2$ or $G$ is of type $\tA_{n-1}(\varepsilon q)$ resp. $\tE_6(\varepsilon q)$
with $\ell \mid (n,q-\varepsilon)$ resp. $\ell \mid (3,q-\varepsilon)$. With the previous section and our emphasis on $\ell$ odd, we are interested in considering the latter cases.

Here we consider the Dade ramification group in $GE$ (resp. $\wt{G}E$) for $B_0(G)$ (resp. $B_0(\wt G)$) when $G$ is one of these groups. Note that if $q$ is not square, then $|G\langle F_p\rangle/G|$ is odd. In this case, the Steinberg character $\mathrm{St}_G$ satisfies Lemma \ref{lem:oddp'}. Then 
when combined with the following, we will be particularly interested in the case that $q$ is a square and $\epsilon=1$. We momentarily relax our assumption that $\bG$ is simple, simply connected in the following, in order to work with $\SL_n(\epsilon q)$ and $\GL_n(\epsilon q)$ simultaneously.

\begin{lemma}\label{lem: Dade principal}
Let $G=\bG^F\in\{\SL_n(\varepsilon q), \GL_n(\varepsilon q)\}$ and assume that $2 \neq \ell \mid (q-\varepsilon)$ with $n=\ell^k$. If $\sigma \in E_{\ell'}$ is in the Dade group of $B_0(G)$ in $GE$ (see Lemma \ref{lem:alpdade}), then $\sigma$ is a field automorphism satisfying $\C_{G\sig}(T_\ell)=\C_G(T_\ell) \langle \sigma \rangle$,  where $T_\ell$ is a Sylow $\ell$-subgroup of $\bT^F$ and $\bT\leq\bG$ is the diagonal torus. In particular, if $\sigma$ has even order, then $\varepsilon=1$.
\end{lemma}

\begin{proof}
We let $\bT$ be the diagonal torus of $\bG$ on which $F$ acts as $x \mapsto x^{q}$ (resp. $x \mapsto x^{-q}$ if $F$ is twisted). Let $T:=\bT^F$. Here a field automorphism can be chosen such that it acts as $x \mapsto x^{p_0}$ on $T$ and the graph automorphism acts as $x \mapsto x^{-1}$ on $T$. In particular, $T$ is normalized by $\sigma$.

Recall from the definition of the Dade group that there is some $n \in G$ such that $\sigma n$ centralizes a Sylow $\ell$-subgroup of $G$ containing $T_\ell$, where $T_\ell\in\Syl_\ell(T)$. However, $\C_G(T_\ell)=T$ by \cite[Prop.~22.6]{CE04}. Hence, $n$ normalizes $T$. By looking closely at elements of $T_\ell$ (e.g. the elements $h_{e_i-e_{j}}(\zeta)$ for all $i,j$ and $\zeta$ of order $(q-\varepsilon)_\ell$) and considering the action of $\N_{\bG^F}(\bT)/\bT^F\cong S_{n}$ on $T_\ell$, we see that necessarily $n\in T$ and $\sigma$ must be a field automorphism centralizing $T_\ell$.

The last statement now also follows from our assumption that $\ell\mid (q-\epsilon)$, since if $\sigma$ has even order, then $p_0^f=q$ for some even $f$. But then since $\sigma$ centralizes $T_\ell$, we have $\ell$ divides $p_0-1$, and hence also $q-1$.
\end{proof}

\begin{lemma}\label{lem:2E6}
    If $G=\tE_6(-q)$ with $\ell=3 \mid (q+1)$ and $\sigma\in E$ is in the Dade group of $B_0(G)$ in $GE$, {or in the Dade group of $B_0(\wt{G})$ in $\wt{G}E$}, then $\sigma$ has odd order. 
\end{lemma}

\begin{proof}
Assume the contrary, that there is such a $\sigma$ with even order in the Dade group for $B_0(G)$, and recall that $E$ is cyclic, so this means that in particular the Dade group contains the unique $\tau\in E$ of order two.

Consider the $3$-central element $s$ with centralizer of type $\tA_2(-q)^3.3$ in $G^\ast$. Note that the three semisimple characters in the associated Lusztig series cannot all be $\sigma$-stable as any $\sigma$ of even order acts non-trivially on $\wt{G}/G\Z(\wt{G})$. 
However, note that these characters have degree prime to $3$ and that $B_0(G)$ is the unique unipotent block of maximal defect, so these characters all lie in $B_0(G)$. Hence, $\sigma$ cannot be in the Dade group of $B_0(G)$, as it must stabilize the characters of $B_0(G)$.

Now suppose that the Dade group for $B_0(\wt{G})$ contains some $\sigma\in E$ of even order, so that it contains $\tau$. Let $\wt{P}\in\Syl_\ell(\wt{G})$ be $E$-invariant. Since $\tau$ is in the Dade group, we have that it acts on $\wt{P}$, and hence on $P:=\wt{P}\cap G\in\Syl_\ell(G)$, the same as some $x\in \wt{G}$. But notice that since $\tau$ has order two, then $x^3$ also acts on $P$ as $\tau$. Since $x^3\in G\Z(\wt{G})$, this means that $\tau$ is in the Dade group for $B_0(G)$ as well, a contradiction to the result for $G$ above.
\end{proof}

By the previous two lemmas and the discussion above, we can focus on the case where $q$ is square and $G=\SL_n(q)$ or $G=\tE_6(q)$ with $\ell \mid (q-1)$.

\begin{lemma}\label{lem:regell}
Let $G\in\{\SL_n(q), \GL_n(q)\}$ with $n=\ell^k$ for some integer $k$ and $2\neq\ell \mid (q-1)$.
Let $\sigma\in E$ be such that $G\langle\sigma\rangle$ is the Dade ramification group for  $B_0(G)$ in $GE$.
Then there exists a regular $\ell$-element $x\in G$ that is invariant under $\sigma$. 
\end{lemma}
\begin{proof}
 Let $\zeta$ be an $n=\ell^k$th root of unity and let $x$ be the regular element $\mathrm{diag}(1, \zeta, \dots, \zeta^{n-1})$. By the proof of Lemma \ref{lem: Dade principal}, the element $x$ is $\sigma$-stable. 
\end{proof}

\subsection{The principal block and generalized Harish-Chandra induction}

Here we relax our assumptions on $\bG$, so that $\bG$ can be any connected reductive algebraic group (not necessarily simple, simply connected).

For a finite group $H$, it will be useful to denote by $H_\ell$ the set of elements of $\ell$-power order and by $H_{\ell'}$ the set of elements of order prime to $\ell$ (which we will also call $\ell'$-order). If $K$ is a splitting field for $H$ of characteristic $0$, we write $CF(H, {K})$ for the set of class functions from $H$ to ${K}$. 
For $x\in H_\ell$, we recall from \cite[Def.~5.7]{CE04}  the map $$d^x: CF(H,K) \to CF(\C_H(x),K),$$
defined such that for $\chi \in CF(H,K)$ and $y \in \C_H(x)$ we have $d^x(\chi)(y)=\chi(xy)$ if $y$ has $\ell'$-order and $d^x(\chi)(y)=0$ otherwise. (Note that we use the same notation $d^x$ for any group $H$, and the latter should be clear from context.)

We need the following generalization of \cite[Thm.~21.4]{CE04}.

\begin{lemma}\label{lem:dx}
Let $\hat{\G}$ be any reductive group with $\hat{\G}=\G \langle \sigma \rangle$ where $\G:=\hat{\G}^\circ$.
Let $\Para$ be a parabolic subgroup of $\G$ containing an $F$-stable Levi subgroup $\Levi$, and suppose that $(\Levi,\Para)$ is $\sigma$-stable. Then for any $\sigma$-stable $x \in \Levi^F$ of $\ell$-power order
we have
\[d^x\circ {}^\ast R_{\bL\langle\sigma\rangle \subset \Para \langle \sigma \rangle}^{\bG\langle\sigma\rangle}={}^\ast R_{\C^\circ_{\bL}(x) \langle\sigma\rangle \subset \C^\circ_{\Para}(x) \langle \sigma \rangle}^{\C^\circ_{\bG}(x) \langle\sigma\rangle}\circ d^x.\]
\end{lemma}
\begin{proof}
Write $G:=\bG^F$. Our first observation is that for $\chi$ a class function of $G$, we have $d^1(\chi)=\delta \cdot \chi$ where $\delta$ is the characteristic function of $(G \langle \sigma \rangle)_{\ell'}$. From this, it follows from \cite[Prop.~2.11]{DM94} applied to the function $\delta$ that $d^1 \circ {}^\ast R^{\G \langle \sigma \rangle}_{\Levi \langle \sigma \rangle \subset \Para \langle \sigma \rangle}={}^\ast R^{\G \langle \sigma \rangle}_{\Levi \langle \sigma \rangle \subset \Para \langle \sigma \rangle} \circ d^1$. Using this observation, together with the fact that $d^x=d^1 \circ d^x$, it follows that both sides of the equation vanish on elements of order divisible by $\ell$ for all class functions. Then it suffices to consider elements of $\ell'$-order.

We also note that since $(\Levi,\Para)$ and $x$ are $\sigma$-stable, it follows that $\C^\circ_{\G }(x) \langle \sigma \rangle$ is a reductive group with parabolic subgroup $\C^\circ_{\Para}(x) \langle \sigma \rangle$ and Levi complement $\C^\circ_{\Levi}(x) \langle \sigma \rangle$  in the sense of \cite{DM94}.

Now let $\psi \in \Irr(\G^F)$ and $l \in \C_{\Levi^F \langle \sigma \rangle}(x)_{\ell'}$. Note that $\C_{\Levi^F \langle \sigma \rangle}(x)_{\ell'} \subset \C^\circ_{\Levi}(x)^F \langle \sigma \rangle$.
Let $l=tv$ be the Jordan decomposition of $l$. Since $x$ commutes with $l$, we have $(xl)_{p'}=xt$ and $(xl)_{p}=v$. Hence, by \cite[Cor.~2.9]{DM94} 
$$d^x({}^\ast R_{\bL\langle\sigma\rangle \subset \Para \langle \sigma \rangle}^{\bG\langle\sigma\rangle}(\psi))(l)={}^\ast R_{\Levi \langle \sigma \rangle \subset \Para \langle \sigma \rangle}^{\G \langle \sigma \rangle}(\psi)(xl)= 
{}^\ast R_{\C^\circ_\Levi(xt) \langle xl \rangle }^{\C^\circ_{\G}(xt) \langle xl \rangle}(\Res_{\C_{G}^\circ(xt) \langle xl \rangle}^{G\sig}(\psi))(xl)$$
while \cite[Cor.~2.9]{DM94} also gives
$${}^\ast R_{\C^\circ_{\bL}(x) \langle\sigma\rangle}^{\C^\circ_{\bG}(x) \langle\sigma\rangle}(d^x(\psi))(l)={}^\ast R_{\C^\circ_\Levi(xt) \langle l \rangle }^{\C^\circ_{\G }(xt) \langle l \rangle}(\Res^{\C^\circ_{G}(x) \langle \sigma \rangle }_{\C^\circ_{G}(xt) \langle l \rangle}(d^x(\psi)))(l)$$
noting that $\C_\G(xt) \subset \C_{\G}(x)$ as $l$ is of $\ell'$-order and so $(xt)_\ell=x$. Now $$\Res^{\C_{G}(x) \langle \sigma \rangle }_{\C^\circ_{G}(xt) \langle l \rangle}(d^x(\psi))=d^1(t_x(\Res^{\C_{G}(x) \langle \sigma \rangle}_{\C^\circ_{G}(xt) \langle l \rangle}(\psi))),$$ where $t_x(f)(y):=f(xy)$ for the central element $x \in \Z(\C_\G^\circ(xt) \langle l \rangle)$. Now $t_x$ and $d^1$ commute with $^\ast R_{\C^\circ_\Levi(xt) \langle l \rangle }^{\C^\circ_{\G }(xt) \langle l \rangle}$ by \cite[Prop.~2.6]{DM94} resp. the first observation at the beginning of the proof which implies that
$${}^\ast R_{\C^\circ_\Levi(xt) \langle l \rangle }^{\C^\circ_{\G }(xt) \langle l \rangle}(\Res^{\C^\circ_{G}(x)}_{\C^\circ_{G}(xt)}(d^x(\psi)))(l)={}^\ast R_{\C^\circ_\Levi(xt) \langle l \rangle }^{\C^\circ_{\G }(xt) \langle l \rangle}(\Res^{\C^\circ_{G}(x)}_{\C^\circ_{G}(xt)}(\psi))(xl).$$
Hence the two sides coincide as stated.
\end{proof}

The next two lemmas will establish a situation in which the generalized Steinberg character is guaranteed to lie in the principal block.

\begin{lemma}\label{lem:Steinberg}
Let $\hat{\G}$ be a reductive group with $\hat{\G}=\G \langle \sigma \rangle$ where $\G:=\hat{\G}^\circ$.
Let $\T$ be an $F$-stable maximal torus contained in an $F$-stable Borel subgroup $\B$ of $\G$ such that $(\T,\B)$ is $\sigma$-stable. Let $\hat\chi:=\mathrm{St}_{G\langle \sigma\rangle}$ be the generalized Steinberg character (see Definition \ref{def:genstein}). Then we have
 ${}^\ast R_{\bT\langle\sigma\rangle}^{\bG\langle\sigma\rangle}(\hat\chi)=1_{{\bT}^F\langle\sigma\rangle}$.
\end{lemma}
\begin{proof}
Write $B:=\B^F$ and $T:=\bT^F$. Recall from the proof of Lemma \ref{lem:ACGGRprincblock} that $\hat\chi$ satisfies $\Res_{B \langle \sigma \rangle}^{G \langle \sigma \rangle}(\hat\chi)=\Res_{B \langle \sigma \rangle}^{G \langle \sigma \rangle}(\St_{G\sig})=
\Ind_{T\sig }^{B \sig}(1_{T \sig})$. 
Hence,
$$\langle \RTGsig(1_{T \langle \sigma \rangle}), \St_{G \langle \sigma \rangle} \rangle = \langle \Ind^{G \langle \sigma \rangle}_{B \langle \sigma \rangle}(1_{B \langle \sigma \rangle}), \St_{G \langle \sigma \rangle} \rangle= \langle 1_{B \langle \sigma \rangle}, \Ind_{T \langle \sigma \rangle}^{B \langle \sigma \rangle}(1_{T \langle \sigma \rangle})=\langle 1_{T \langle \sigma \rangle}, 1_{T \langle \sigma \rangle} \rangle=1$$
by appyling Frobenius reciprocity twice.

 Now let $\rho$ be the regular character of $\T^F$, so that $\hat{\rho}:=\Ind_{T}^{T \langle \sigma \rangle}(\rho)$ is the regular character of $T \langle \sigma \rangle$. By classical Harish-Chandra theory (e.g. applying \cite[Prop.~19.6]{CE04}), we have $\langle \RTG(\rho), \St_G \rangle= \langle \RTG(1_T), \St_G \rangle=1$. {We also have 
 \begin{equation}\label{eq:rtgind}
 \RTGsig \circ \Ind_{T}^{T \langle \sigma \rangle}=\Ind_{G}^{G \langle \sigma \rangle} \circ \RTG
 \end{equation} from \cite[Cor.~2.4(ii)]{DM94}
by adjunction.}
 Hence,
  $$ \langle \RTGsig \Ind_{T}^{T \langle \sigma \rangle}(\rho),\hat{\chi} \rangle= \langle \Ind_{G}^{G \langle \sigma \rangle} \RTG(\rho),\hat{\chi} \rangle = \langle \RTG(\rho),\chi \rangle=1,$$ 
which by positivity of Harish-Chandra-induction shows that $\langle \RTGsig \hat{\theta},\hat{\chi} \rangle= 0$ for any nontrivial $\hat{\theta} \in \Irr(T \langle \sigma \rangle)$.
\end{proof}

\begin{lemma}\label{lem:steinbergA}
{Keep the notation of Lemma \ref{lem:Steinberg}.}
    Assume that there exists some $\sigma$-stable element $x \in \T_\ell^F$ that is regular (i.e. satisfies $\C_{\G}^\circ(x)=\T$) where $\T$ is the maximally split torus of $\G$. 
    Then every irreducible constituent of $\RTGsig(1_{T\sig})$ lies in the principal block.
\end{lemma}
\begin{proof}
We follow the proof of \cite[Thm.~21.14]{CE04}, applying Lemma \ref{lem:dx} in place of \cite[Thm.~21.4]{CE04}. 
Let $\hat{\chi} \in \Irr(G \langle \sigma \rangle)$ be an irreducible constituent of $\RTGsig(1_{T\sig})$.
We note that
$${}^\ast \RTGsig(\hat{\chi})=\displaystyle\sum_{\hat{\theta} \in \Irr(T\langle \sigma \rangle \mid 1_T )} n_{\hat{\theta}} \hat{\theta}$$
by positivity of Harish-Chandra induction {and again applying \cite[Cor.~2.4]{DM94}. Note that $d^x(\hat\theta)=d^1(\hat\theta)$ for each $\hat\theta$ in the sum since $x\in T$ and each (linear) $\hat\theta$ in the sum is trivial on $T$}. Hence, 
$$f:=B_0(T \langle \sigma \rangle) d^x({}^\ast \RTGsig(\hat{\chi}))=\displaystyle\sum_{\hat{\theta} \in \Irr(T\langle \sigma \rangle \mid 1_T ) \cap \Irr(B_0(T \langle \sigma \rangle))} n_{\hat{\theta}} d^1(\hat{\theta}),$$ where (by an abuse of notation), the left-hand side denotes multiplication by the idempotent corresponding to $B_0(T\sig)$.

Note that each $n_{\hat{\theta}}$ is a nonnegative integer, each $d^1(\hat{\theta})$ is an irreducible Brauer character, and $n_{1_{T \langle \sigma \rangle}} \neq 0$. This implies that $f \neq 0$ and thus $\langle f,f \rangle \neq 0$. Hence, there is some $\hat{\theta} \in \Irr(T\langle \sigma \rangle \mid 1_T) \cap \Irr(B_0(T \langle \sigma \rangle))$ with $\langle f,d^1(\hat{\theta}) \rangle \neq 0$. 

According to Brauer's first main theorem, it suffices to show that $d^x(\hat{\chi}) \in CF(\C_{G \langle \sigma \rangle}(x),K)$ has a non-zero projection onto the principal block of $\C_{G \langle \sigma \rangle}(x)$. 
Note that since $\C_\bG^\circ(x)=\bT$, we have $d^x (\hat{\chi})={}^{\ast}R_{\C^\circ_{\T}(x) \langle \sigma \rangle}^{\C^\circ_{\G }(x) \langle \sigma \rangle}(d^x (\hat{\chi}))$ and $\hat\theta=R_{\C^\circ_{\T}(x) \langle \sigma \rangle}^{\C^\circ_{\G }(x) \langle \sigma \rangle}(\hat{\theta})$. Then we have

$$ \langle B_0(T \langle \sigma \rangle) d^x (\hat{\chi}), \hat{\theta} \rangle_{T \langle \sigma \rangle}=\langle B_0(T \langle \sigma \rangle) d^x(\hat{\chi}), R_{\C^\circ_{\T}(x) \langle \sigma \rangle}^{\C^\circ_{\G }(x) \langle \sigma \rangle}(\hat{\theta}) \rangle_{\C_\G^\circ(x)^F \langle \sigma \rangle} $$
and by Lemma \ref{lem:dx} this is also the same as
$\langle B_0(T \langle \sigma \rangle) d^x ({}^\ast \RTGsig(\hat{\chi})),\hat{\theta} \rangle_{T \langle \sigma \rangle}$, which is nonzero from above. 
This shows that $d^x \hat{\chi}$ has a non-zero projection onto the principal block of $\C_\G^\circ(x)^F \langle \sigma \rangle$. Following the last seven lines of the proof of \cite[Thm. 21.14]{CE04} then yields the statement.
\end{proof}

\subsection{Consequences for the principal block}

We now establish that in some cases needed for our purposes in which $\St_G$ lies in the principal block,  the extended Steinberg character also lies in the principal block. (However, we refer the reader to Example \ref{ex:steinberg} below to see that this, perhaps unexpectedly, is not always the case!)

\begin{corollary}\label{cor:SLpowerell}
Assume that $G=\SL_n(q)$ with $n=\ell^k$ for some integer $k$ and $2\neq\ell \mid (q-1)$. Then the Steinberg character $\mathrm{St}_{G}$ has a rational-valued extension to $\tilde{G} E$ that lies in $B_0(\wt{G} E)$. 
\end{corollary}

\begin{proof}
We begin by remarking that $\St_{\wt G}$ (which extends $\St_G$) extends to a rational character $\St_{\wt{G} E}$ of $\wt{G} E$, by Remark \ref{rem:gensteinrat}. (See also \cite{schmid}.)
Further, recall that $E$ is abelian, so every character above $\St_{\wt G}$ is an extension.

Let $\sigma\in E$ be such that $\wt G\sig$ is the Dade ramification group for $B_0(\wt G)$ in $\wt G E$ and let $\hat\chi:=\St_{\wt G\sig}$. Recall from Remark \ref{rem:gensteinrat} that $\hat\chi$ is a rational extension of $\St_G$. 
 By Lemma \ref{lem:regell}, there is an $x$ as in Lemma \ref{lem:steinbergA}, and hence combining with Lemma \ref{lem:Steinberg}, we have $\hat\chi$ lies in $B_0(\wt G\sig)$.

 Note that $B_0(\wt{G}E)$  is the unique block lying above $B_0(\wt{G}\sig)$ (see Lemma \ref{lem:alpdade}).
 In particular, note that the rational character $\St_{\wt{G} E}$, which extends $\hat\chi$, must lie in the principal block.
\end{proof}

Note that a regular element $x$ as in Lemma \ref{lem:steinbergA} does not necessarily exist in every case. In some such cases, we may proceed inductively, as in the following.

\begin{lemma}\label{E6}
     Assume that $G=\tE_6(q)$ with $\ell=3$ and $\ell \mid (q- 1)$. Then the Steinberg character $\mathrm{St}_{G}$ has a rational-valued extension to $\tilde{G} E$ that lies in $B_0(\wt{G} E)$.  
\end{lemma}

\begin{proof}
As before, note that $\St_{G}$ has a rational extension to $\wt{G} E$ and that $E$ is abelian.
Let $\sigma\in E$ be such that $ G\langle\sigma\rangle$ is the Dade ramification group for $B_0( G)$ in $ G\langle F_p\rangle$. {We first claim that $\hat\chi:=\St_{ G\langle\sigma\rangle}$ lies in $B_0( G\langle\sigma\rangle)$.}

 We note that there exists an element $x$ in the maximally split torus  $T=\bT^F$ satisfying that $x$ is $\sigma$-stable and has centralizer $\bH:=\C_\G(x)$ of type $\tA_2(q)^3$.  Moreover, there exists a $\sigma$-stable element $y \in H:=\bH^F$ of order $3$ with $\C^\circ_{\bH}(y)^F=T$.  Applying Lemma \ref{lem:steinbergA} to $\sigma$ viewed as an automorphism of $H$, we observe that all irreducible constituents of $R_{\bT \langle \sigma \rangle}^{\bH \langle \sigma \rangle}(1_{T\sig})$ lie in the principal block. 

Consider
$$ \langle d^x \hat{\chi}, R_{\C^\circ_{\T}(x) \langle \sigma \rangle}^{\C^\circ_{\G}(x)  \langle \sigma \rangle}(1) \rangle_{\C^\circ_{\G}(x)^F  \langle \sigma \rangle}. $$
By Lemmas \ref{lem:dx} and \ref{lem:Steinberg}, this is the same as
$$\langle d^x {}^\ast R_{\bT \langle \sigma \rangle}^{\bG \langle \sigma \rangle}(\hat{\chi}), d^1(1_{T\langle \sigma \rangle}) \rangle_{T \langle \sigma \rangle}= \langle d^x(1_{T \langle \sigma \rangle}),d^1(1_{T \langle \sigma \rangle}) \rangle_{T \langle \sigma \rangle}=1/|T_\ell| \neq 0,$$
which shows the claim as in the end of the proof of  Lemma \ref{lem:steinbergA}, as all constituents of $R_{\C^\circ_{\T}(x) \langle \sigma \rangle}^{\C^\circ_{\G}(x) \langle \sigma \rangle}(1)=R_{\bT \langle \sigma \rangle}^{\bH \langle \sigma \rangle}(1)$ lie in the principal block of $H\sig$.
The same proof shows that $\St_{\wt{G}\sig}$ lies in the principal block, taking $\wt{G}\sig$ to be the Dade group for $B_0(\wt{G})$ in $\wt{G}\langle F_p\rangle$.

If $\sigma$ has odd order, then by Lemma \ref{lem:real}, $\St_{\wt{G}\sig}$ is the unique rational extension of $\St_{\wt G}$ to $\wt{G}\sig$, and therefore must lie below one of the rational extensions to $\wt{G} E$. If instead $\sigma$ has even order, then the two rational extensions of $\St_{\wt G}$ to $\wt{G}\langle F_p\rangle$ must lie above the two rational extensions of $\St_{\wt{G}}$ to $\wt{G}\sig$. In particular, in either case there is a rational extension, say $\chi_1$, of $\St_{\wt{G}\sig}$ to $\wt{G}\langle F_p\rangle$ lying in the principal block, as this is the unique block above $\wt{G}\sig$. 

Now, note that the two extensions of $\St_{\wt{G}}$ to $\wt{G}\langle\tau\rangle$, where $\tau$ is a nontrivial graph automorphism, must be rational-valued since there is a rational extension to $\wt{G} E$.
Then applying Lemma \ref{lem:gluing} to $\chi_1$ and these two extensions to $\wt{G}\langle\tau\rangle$ yields that the two characters of $\wt{G}E$ lying above $\chi_1$ must be rational-valued. In particular, at least one of these lies in the principal block, completing the proof.
\end{proof}

We end this section by considering the Suzuki and Ree groups.
\begin{lemma}\label{lem:suzree}
Let $S$ be a simple Suzuki or Ree group $\tw{2}\tF_4(q^2)'$, $\tw{2}\tG_2(q^2)'$, or $\tw{2}\tB_2(q^2)$ and let $\ell$ be any prime dividing $|S|$ but not dividing $q^2$. Then $S$ satisfies condition ($\star$) of Definition \ref{def:lmnt} for $(\QQ, \ell)$. 
\end{lemma}
\begin{proof}
Note that the case of the Tits group $\tw{2}\tF_4(2)'$ or $S=\tw{2}\tG_2(3)'\cong \operatorname{PSL}_2(8)$ can be checked directly in GAP, so we may assume that $S\in\{\tw{2}\tF_4(q^2), \tw{2}\tG_2(q^2), \tw{2}\tB_2(q^2)\}$ with $q^2=2^{2n+1}, 3^{2n+1}, 2^{2n+1}$, respectively, with $n\geq 1$.

   Here the Steinberg character $\mathrm{St}_S$ is rational-valued, lies in $\Irr_{0}(B_0(S))$, and extends to $\Aut(S)$. Further, $\Aut(S)/S$ is cyclic of odd order. Let $SE_{\ell'}\lhd \Aut(S)$ with $SE_{\ell'}/S$ the $\ell$-complement in $\Aut(S)/S$. Then there is a unique rational extension $\wt\chi$ of $\mathrm{St}_S$ to $SE_{\ell'}$ by Lemma \ref{lem:real}.
   Further, by Lemma \ref{lem:oddp'}, we have $\mathrm{St}_S$ extends to a rational-valued character $\hat\chi$ of $B_0(M)$, where $M\leq SE_{\ell'}$ is the Dade ramification group. Note also that $\hat\chi$ is the unique rational extension of $\chi$ to $M$. It follows that $\Res_M^{SE_{\ell'}}(\wt\chi)=\hat\chi$. Then by Lemma \ref{lem:alpdade}, it must be that $\wt\chi\in\Irr_{0}(B_0(SE_{\ell'}))$. Note also that $\ell\nmid o(\wt\chi)$ since $S\leq [SE_{\ell'}, SE_{\ell'}]$, so the linear characters of $SE_{\ell'}$ are of $\ell'$-order.

   Now $\Aut(S)/SE_{\ell'}$ is an $\ell$-group, so $B_0(\Aut(S))$ is the unique block above $B_0(SE_{\ell'})$.  Then since $\ell\nmid o(\wt\chi)\wt\chi(1)$, by \cite[Cor.~6.4]{Nav18}, there is a canonical extension of $\wt\chi$ to $\Aut(S)$, which must be rational and necessarily lies in $B_0(\Aut(S))$. This completes the proof. 
\end{proof}

\section{Proof of Theorem \ref{thm:main}}\label{sec:mainproof}

Finally, here we use the results of the preceding sections to prove the following, which will give  Theorem \ref{thm:main}.

\begin{theorem}\label{thm:mainlie}
    Let $S$ be a simple group of Lie type defined in characteristic $p$, and let $\ell\neq p$ be a prime dividing $|S|$. Then there exists a nontrivial character $\chi \in \Irr_0(B_0(S))$, where $B_0(S)$ is the principal $\ell$-block of $S$, such that $\chi$ extends to a character $\hat{\chi}$ in $\Irr_0(B_0(\Aut(S)_\chi))$ with $\mathbb{Q}(\hat{\chi}) \subseteq \mathbb{Q}_\ell$ and $\ell \nmid |\mathrm{Aut}(S):\mathrm{Aut}(S)_\chi|$. In particular, $S$ satisfies condition ($\star$) of Definition \ref{def:lmnt} for $(\QQ_\ell, \ell)$.
\end{theorem}
\begin{proof}
We may write $S:=G/\Z(G)$ with $G:=\bG^F$, where $F$ is a Steinberg endomorphism and $\bG$ is a simple, simply connected algebraic group. By Lemma \ref{lem:suzree}, we may further assume that $F$ is a Frobenius endomorphism. We have $\Aut(S)=\wt{S}\rtimes E$, where $E=E(\bG^F)$ is as before and $\wt{S}=\wt{G}/\Z(\wt{G})$ with $\wt{G}=\wt\bG^F$ as before.

We may assume $\ell$ is odd, since if $\ell=2$, then every rational-valued character of odd degree is in the principal block by \cite[Cor.~1.3]{gow}. In particular, the Steinberg character $\chi=\mathrm{St}_G$ works in that case, as we know the extended character $\St_{\wt{G} E}$ from Definition \ref{def:genstein} is rational-valued (see Remark \ref{rem:gensteinrat}). (This case is also claimed in \cite{lmnt}, using results in \cite{nt10}.)

 Next, suppose that $G=\SL_n(\epsilon q)$ with $\ell\mid (q-\epsilon)$ and $n=\ell^k$ or that $G=\tE_6(\epsilon q)$ with $\ell=3\mid (q-\epsilon)$. Then 
by Corollary \ref{cor:SLpowerell} and Lemma \ref{E6}, we may assume $\epsilon=-1$. But in this case, by Lemmas \ref{lem: Dade principal} and \ref{lem:2E6}, 
we know the Dade ramification group $M$ of $B_0(\wt{G})$ in $\wt{G} E$ is $\wt G\langle\sigma\rangle$, where $\sigma$ has odd order. By Lemma \ref{lem:oddp'}, the principal block $B_0(M)$ contains the unique rational-valued extension of $\St_{\wt G}$ to $M$. Then applying Lemma \ref{lem:alpdade} (and recalling Remark \ref{rem:gensteinrat}), the extended Steinberg character $\St_{\wt{G} E}$ from Definition \ref{def:genstein} lies in $B_0(\wt{G} E)$ and is rational-valued.

Hence we may assume that either $\ell\nmid |\Z(G)|$ or $G=\SL_n(\epsilon q)$ satisfying that $\ell\mid (n, q-\epsilon)$ but that $n$ is not a power of $\ell$.
By Lemmas \ref{lem:semisimple},  \ref{lem:typeAZ(P)}, \ref{lem:typeAZP}, and \ref{D4},  
there is then some nontrivial semisimple $\wt\chi\in\Irr_0(B_0(\wt{S}))$, where $\wt{S}:=\wt{G}/\Z(\wt{G})$, with $\QQ(\wt\chi)\subseteq\QQ_\ell$ and $\ell\nmid [E:E_{\wt\chi}]$. 

Note that $\Aut(S)_{\wt{\chi}}=\wt{S}E_{\wt\chi}$ contains the Sylow $\ell$-subgroup of $\Aut(S)$.
If $\ell\nmid |Z(G)|$, by Lemma \ref{lem:bijsimplediags}, $\wt\chi$ restricts to some nontrivial $\chi\in\Irr_0(B_0(S))$. If instead $G=\SL_n(\epsilon q)$ with $\ell\mid (n, q-\epsilon)$, the same holds by Lemma \ref{lem:typeAZP}. Further, in these situations the same Lemmas give $E_\chi=E_{\wt\chi}$, so that $\ell\nmid |\Aut(S):\Aut(S)_\chi|$ since $\Aut(S)_\chi=\wt{S}E_\chi=\wt{S}E_{\wt{\chi}}$; and certainly $\QQ(\chi)\subseteq \QQ(\wt\chi)\subseteq\QQ_\ell$.

Note that $\langle\wt\chi, D_{\wt G}(\Gamma)\rangle=\pm1$ by \cite[2.6.10(d), Rem.~3.4.18]{GM20}, where $\Gamma$ is the Gelfand Graev character of $\wt{G}$. 
Write $E'=E_{\wt\chi}\cap E_{\ell'}=E_\chi\cap E_{\ell'}$.  By Lemma \ref{lem:mult1}, there is a unique  $\bar\chi\in\Irr(\wt{G}E')$ with $\langle \bar\chi, D_{\wt{G}E'}(\hat\Gamma)\rangle=\pm1$ extending $\wt{\chi}$, which satisfies $\QQ(\bar\chi)=\QQ(\wt\chi)\subseteq\QQ_\ell$.
 From the results of Hiss \cite{hiss}, we have $\langle \wt\chi, D_{\wt{G}}( \Gamma_1)\rangle \neq 0$, so in fact $\langle\bar\chi, D_{\wt{G}E'}(\hat\Gamma_1)\rangle=\pm1$, and by Lemma \ref{lem:ACGGRprincblock} we have $\bar\chi\in\Irr(B_0(\wt{G}E'))$. (That is, $\bar\chi$ is an extension of $\chi$ to $B_0(\wt{S}(E_\chi\cap E_{\ell'}))$.)

 Now, note that $E_\ell\lhd E$. Assume for the moment that $\ell\nmid|\Z(G)|$. Then $\ell\nmid o(\wt\chi)$ 
since $\ell\nmid |\Z(G)|$, so $\ell\nmid [\wt{S}:S]=|\wt{S}|/[\wt{S}, \wt{S}]$. 
Then we may let $\theta$ be the canonical extension of $\wt\chi$ to $\wt{S}E_\ell$ as in \cite[Cor.~6.4]{Nav18}. Then $\QQ(\theta)=\QQ(\wt\chi)$ and $\theta$ is $E'$-invariant. Hence there is a unique extension $\hat\chi$ of $\wt{\chi}$ to $\wt{S}E_\chi$ extending both $\theta$ and $\bar\chi$, by Lemma \ref{lem:gluing}, and we have $\QQ(\hat\chi)=\QQ(\theta)=\QQ(\bar\chi)=\QQ(\wt\chi)\subseteq\QQ_\ell$. 

If we further assume that $\ell\neq 3$ or $G\neq\tD_4(q)$, then $E'\lhd E_\chi$, and as $E_\chi/E'$ is an $\ell$-group, $\hat\chi$ lies in the principal block (the unique block above $B_0(\wt{S}E')$).  

Now assume $\ell=3$ and $G=\tD_4(q)$. Let $E''$ be the intersection of $E'$ with the group of field automorphisms in $E$, which we recall are central in $E$. Then we have $E''\lhd E_\chi$, $E_\ell\lhd E_\chi$, and $\wt G E_\ell E''\lhd \wt{G}E_\ell E'=\wt{G}E_{\chi}$ with index at most $2$. As above, there is a unique common extension $\hat\theta$ of $\theta$ and $\Res_{\wt G E''}^{\wt G E'}(\bar\chi)$ to $\wt{G}E_\ell E''$, which must have its values in $\QQ(\wt\chi)$, be $E'$-invariant, and lie in the principal block (since $\wt{G}E_\ell E''/\wt{G}E''$ is an $\ell$-group). Then as $\hat\chi$ lies above $\hat\theta$, we see both characters in $\wt{G} E_\chi$ lying above $\hat\theta$ have their values in $\QQ_\ell$, and at least one of these must lie in the principal block.

Now let $\ell\mid|\Z(G)|$. Recall here that then $G=\SL_n(\epsilon q)$ with $\ell\mid (n, q-\epsilon)$ but $n$ is not a power of $\ell$. 
Recall that $\Aut(S)_\chi=\wt{S} E_\chi=\wt{S}E_{\wt \chi}$ and that both $E'$ and $E_\ell$ are normal in $E_\chi$.

Now, since $S$ is simple, we have $o(\chi)=1$, so by \cite[Cor.~6.4]{Nav18} there is a canonical extension $\chi'$ of $\chi$ to $SE_\ell$, which must have  values in $\QQ_\ell$ and be $E'$-invariant. Again this must lie in $B_0(S E_\ell)$ since $E_\ell$ is an $\ell$-group. 
Then we may take the unique common extension $\theta$ of $\wt\chi$ and $\chi'$ to $\wt{S} E_\ell$ (again applying Lemma \ref{lem:gluing} and using that $\wt\chi$ is $E_\ell$-invariant), which must again satisfy $\QQ(\theta)\subseteq \QQ_\ell$. As $\wt{S} E_\ell/\wt{S}$ is again an $\ell$-group, we have $\theta$ must lie in $B_0(\wt{S} E_\ell)$. Since $\chi'$ and $\wt\chi$ are stable under $E'$ and $\theta$ is canonical, we see that  $\theta$ is $E'$-stable.

Finally, we may take the unique common extension $\hat\chi$ of $\theta$ and $\bar\chi$ to $\wt{S} E_{\wt\chi}=\wt S E_\chi$, with another application of Lemma \ref{lem:gluing}, and we see $\QQ(\hat\chi)\subseteq\QQ_\ell$ and $\hat\chi$ is in the principal block since $\wt{S} E_{\wt\chi}/\wt{S} E'$ is an $\ell$-group.
\end{proof}

\begin{proof}[Proof of Theorem \ref{thm:main}]
From the reduction theorem proved in \cite{lmnt} (see Theorem \ref{thm:lmnt}), it suffices to know that every finite nonabelian simple group of order divisible by $\ell$ satisfies condition $(\star)$ of Definition \ref{def:lmnt} for $(\QQ_\ell,\ell)$. This holds for simple groups of Lie type defined in characteristic $p=\ell$ by \cite[Prop.~3.6]{lmnt}. Further, it is satisfied for all primes $\ell$ for $(\QQ, \ell)$ and simple alternating groups by \cite[Prop.~3.7]{lmnt}. In \cite{lmnt}, the condition is also checked in GAP for sporadic groups. Hence, we are left to consider the case that $S$ is a simple group of Lie type defined in characteristic $p\neq \ell$.  This final case was Theorem \ref{thm:mainlie}, completing the proof.
\end{proof}

\section{Block theory for disconnected reductive groups}\label{sec:disconblocks}

The results of this section generalize many ideas of the previous sections and may be of independent interest. More precisely, our aim is to develop a block theory for the blocks of a disconnected reductive group.

For this purpose, let $\hat{\G}$ be a reductive group equipped with a Frobenius endomorphism $F$ and let $\G$ be the connected component of its unit. Recall that a closed subgroup $\Para'$ of $\hat{\G}$ is called parabolic if and only if $\Para:={\Para'}^\circ$ is a parabolic subgroup of $\G$. Then $\Levi':=\N_{\hat{\G}}(\Para',\Levi)$ with $\Levi$ a Levi subgroup of $\Para$ in $\G$ is called a Levi subgroup of $\hat{\Para}$.

In particular, we have that $\hat{\Para}:=\N_{\hat{\G}}(\Para)$ is a parabolic subgroup of $\hat{\G}$ with Levi subgroup $\hat{\Levi}=\N_{\hat{\G}}(\Levi,\Para)$. {(We note that this is what is called  a `parabolic', resp. `Levi', in \cite{DM94})}. Clearly, $\Para \leq \Para' \leq \hat{\Para}$ and $\Levi \leq \Levi' \leq \hat{\Levi}$. Notice however that if $\Levi,\Levi',\hat{\Levi}$ are $F$-stable then we have
\begin{equation}\label{eq:genlusres} \Res_{\Levi'^F}^{\hat{\Levi}^F} \circ {}^\ast R_{\hat{\Levi} \subset \hat{\Para}}^{\hat{\G}}= {}^\ast R_{\Levi' \subset \Para'}^{\hat{\G}} ={}^\ast R_{\Levi' \subset \Para'}^{\G \Levi'} \circ \Res^{\hat{\G}^F}_{(\G\Levi')^F },
\end{equation}
as can be seen from the proof of \cite[Prop.~2.3]{DM94}.

\begin{remark}
Note that for any $x \in \G$, the centralizer $\C_{\hat{\G}}(x)$ is again a reductive group with parabolic subgroup $\C_{\hat{\Para}}(x)$, which has Levi subgroup $\C_{\hat{\Levi}}(x)$. Note that in contrast to the situation above, we may not necessarily have that $\C_{\hat{\Para}}(x)$ is the full normalizer in $\C_{\hat{\G}}(x)$ of $\C_{\hat{\Para}}(x)$, even if this was the case for $(\hat{\Levi},\hat{\Para})$  inside $\hat{\G}$. 
\end{remark}

Throughout, we will continue to let $\ell$ be a prime distinct from the defining characteristic for $\hat\bG$.
We first generalize  Lemma \ref{lem:dx}.
\begin{lemma}\label{lem:dx neu}
    Let $\hat{\G}$ be a reductive group with connected component $\G$. Let $\hat{\Levi}$ be a Levi subgroup of $\hat{\G}$ contained in the parabolic subgroup $\hat{\Para}$ and $x \in \Levi^F=(\hat{\Levi}^\circ)^F$ of $\ell$-power order. Then
    $$d^x \circ {}^\ast R_{\hat{\Levi} \subset \hat{\Para}}^{\hat{\G}}={}^\ast R^{\C_{\hat{\G}}(x)}_{\C_{\hat{\Levi}}(x) \subset \C_{\hat{\Para}}(x)} \circ d^x.$$
\end{lemma}

\begin{proof}
In order to show that the two sides of the equation coincide, note that it suffices to show that their restrictions to $\C_{\Levi}^\circ(x)^F \langle \sigma \rangle$ coincide for every $\sigma \in \C_{\hat{\Levi}}(x)^F$.

Now if $\sigma \in \C_{\hat{\Levi}}(x)^F$ then we have 
$$\Res_{\C_{\Levi}^\circ(x)^F \langle \sigma \rangle}^{\C_{\hat{\Levi}}(x)^F} (d^x \circ {}^\ast R_{\hat{\Levi} \subset \hat{\Para}}^{\hat{\G}})= d^x \circ \Res_{\Levi^F \langle \sigma \rangle}^{\hat{\Levi}^F} \circ {}^\ast R_{\hat{\Levi} \subset \hat{\Para}}^{\hat{\G}}=d^x \circ {}^\ast R_{\Levi \langle \sigma \rangle \subset \Para \langle \sigma \rangle }^{\G \langle \sigma \rangle}\circ \Res_{\G^F\sig}^{\hat\G^F},$$
using the definition of $d^x$ and the formula \eqref{eq:genlusres} for generalized Lusztig restriction above. By Lemma \ref{lem:dx}, we have 
$$d^x \circ {}^\ast R_{\Levi \langle \sigma \rangle \subset \Para }^{\G \langle \sigma \rangle} \circ \Res_{\G^F\sig}^{\hat\G^F}={}^\ast R_{\C^\circ_{\Levi}(x) \langle \sigma \rangle}^{\C^\circ_{\G}(x) \langle \sigma \rangle} \circ d^x \circ \Res_{\G^F\sig}^{\hat\G^F} = \Res_{\C_{\Levi}^\circ(x)^F \langle \sigma \rangle}^{\C_{\hat{\Levi}}(x)^F}({}^\ast R^{\C_{\hat{\G}}(x)}_{\C_{\hat{\Levi}}(x)} \circ d^x),$$
where the second equation again follows from  \eqref{eq:genlusres} and the definition of $d^x$, completing the proof. 
\end{proof}

Denote by $\mathcal{E}(\hat{G},1)$ the set of characters covering a unipotent character of $G$, where $\hat G:=\hat\bG^F$ and $G:=\bG^F$. We need the following lemma (where we also write $\hat L:=\hat\bL^F$):

\begin{lemma}\label{lem:preserveunip}
   Generalized Lusztig induction $R_{\hat{\Levi} \subset \hat{\Para}}^{\hat{\G}}$ restricts to a functor $R_{\hat{\Levi} \subset \hat{\Para}}^{\hat{\G}}: \mathbb{Z}\mathcal{E}(\hat{L},1) \to \mathbb{Z}\mathcal{E}(\hat{G},1)$ and its adjoint restricts to a functor ${}^\ast R_{\hat{\Levi} \subset \hat{\Para}}^{\hat{\G}}: \mathbb{Z}\mathcal{E}(\hat{G},1) \to \mathbb{Z}\mathcal{E}(\hat{L},1)$.
\end{lemma}

\begin{proof}
Let $\hat{\chi} \in \Irr(\hat{G})$ such that $\langle R_{\hat{\Levi} \subset \hat{\Para}}^{\hat{\G}}(\hat{\la}),\hat{\chi} \rangle\neq 0$ for some $\hat{\lambda} \in \mathcal{E}(\hat{L},1)$. Then there is some $i$ such that $\hat\chi$ appears as a direct summand of $H^i_c(Y_{\U}^{\hat{\G}}) \otimes_{\hat{\Levi}^F} \hat{\lambda}$ where $\U$ is the unipotent radical of $\hat{\Para}$. Since $\Res_{L}^{\hat{L}}(\hat{\la})$ is a sum of unipotent characters, it follows from \cite[Example 2.2.26b]{GM20} that there is some $F$-stable torus $\T$ of $\Levi$ such that $\langle R_\T^\Levi(1_T), \Res_{L}^{\hat{L}}(\hat{\la}) \rangle =\langle R_{\T}^{\hat\Levi}(1_T), \hat{\la} \rangle \neq 0$. Hence, there is some $j$ such that $\hat{\lambda}$ is a direct summand of $H^j_c(\Y_{\V}^{\hat{\Levi}})$ where $\V$ is the unipotent radical of a Borel subgroup of $\Levi$ containing the maximal torus $\T$. By the K\"unneth formula, there is some $k$ such that $H^i_c(Y_{\U}^{\hat{\G}})\otimes_{\hat{\Levi}^F} H^j_c(\Y_{\V}^{\hat{\Levi}})$ is a direct summand of $H^k_c(Y_{\U \V}^{\hat{\G}})$. Hence, $\hat{\chi}$ is a direct summand of $H^k_c(Y_{\U \V}^{\hat{\G}})$ and so $\Res_G^{\hat{G}}(\hat{\chi})$ is a direct summand of $H^k_c(Y_{\U \V}^{\G})$. This gives the claim by \cite[Thm.~11.8]{B06}.

The second statement about generalized Luzstig restriction follows in the same way.   
\end{proof}

\subsection{More on the Dade group}
In Section \ref{sec:preliminaries}, we introduced the Dade group for the principal block. We now consider the generalization to arbitrary blocks, see \cite{dade73, Murai}. Given $X\lhd Y$, We denote by $Y[b]\leq Y$ the Dade group in $Y$ for a block $b$ of $X$. We often use the following fact, which is a consequence of \cite{Murai}:

\begin{theorem}\label{thm:dade}
    Let $X \lhd Y$ such that $Y/X$ is an $\ell'$-group. Let $B$ be a block of $Y$ covering a block $b$ of $X$. Then the following are equivalent:
    \begin{enumerate}
        \item $Y[b]=Y$ and some character in $\Irr(B)$ has a multiplicity-free restriction to $X$.
        
        \item The blocks $B \otimes \la$ for $\la \in \Lin(Y/X)$ are all distinct.
    \end{enumerate}
    Moreover, when the above hold, we have $Y=X \C_Y(D)$, where $D$ is a defect group for $b$, and every character of $\Irr(b)$ extends to a character of $\Irr(B)$.
\end{theorem}

\begin{proof}

    In the situation of (1), suppose that $\chi  \in \Irr(B)$ restricts multiplicity-freely to $X$ and let $\vartheta \in \Irr(X \mid \chi)$. Then for $y \in Y$ we have $ \langle X, y\rangle [b]=\langle X,y \rangle$ so \cite[Thm.~3.9]{Murai} implies that $\vartheta$ is $y$-stable. Hence, $\chi$ restricts irreducibly to $X$ and so (1) is equivalent to \cite[Thm.~4.1(i)]{Murai}. In particular, the equivalence of \cite[Thm.~4.1(i)]{Murai} and \cite[Thm.~4.1(v)]{Murai} now implies the equivalence of (1) and (2). The moreover statement is a consequence of \cite[Thms.~3.13 and~4.1]{Murai}, respectively.
\end{proof}

\begin{corollary}\label{cor:linear action}
    Let $X \lhd Y$ such that $Y/X$ is an $\ell'$-group. Let $B$ be a block of $Y$ covering a block $b$ of $X$. Suppose that some character in $\Irr(B)$ restricts irreducibly 
    to $X$. Then for $\la \in \Lin(Y/X)$ we have $B \otimes \la=B$ if and only if $Y[b] \leq \ker(\la)$. 
\end{corollary}

\begin{proof}
Let $\chi \in \Irr(B)$ be a character that restricts irreducibly to $X$ and let $B'$ be the block of $Y[b]$ containing the restriction of $\chi$. Observe that block induction gives a surjective map $\mathrm{Bl}(Y[b] \mid b) \to \mathrm{Bl}(Y \mid b)$ {(by applying the Fong--Reynolds theorem \cite[Thm.~9.14]{Nav98} and \cite[Lem.~2.1(iii) and Thm.~3.5(i)]{Murai})}, which induces a bijection  $\mathrm{Bl}(Y[b] \mid b)/Y \to \mathrm{Bl}(Y \mid b)$, where $\mathrm{Bl}(Y[b] \mid b)/Y$ denotes the orbits under the action of $Y$. 

Under this map, $B'$ maps to $B$ and for each $\la \in \Lin(Y/X)$, the $Y$-stable block $B' \otimes\Res_{Y[b]}^{Y}(\la)$  maps to $B\otimes \la$. Hence, $B=B\otimes \la$ if and only if $\Res_{Y[b]}^{Y}(\la)=1_{Y[b]}$. This gives the claim.
\end{proof}

We also wish to control blocks in the situation of Lemma \ref{lem:gluing}. In the following, given a character $\chi$ of a group $X$, we write $b_X(\chi)$  to denote the block of $X$ containing $\chi$.

\begin{lemma}\label{lem:gluing blocks}
In the situation of Lemma \ref{lem:gluing} assume additionally that {$X_2\lhd X$ and that} the quotient $Y/X$ is abelian of $\ell'$-order, {where $X:=X_1\cap X_2$.} Further assume that 
$Y[b_X(\vartheta_0)]=X_1[b_{X}(\vartheta_0)] X_2[b_{X}(\vartheta_0)]$, where $\vartheta_0:=\Res^{X_1}_X(\vartheta_1)=\Res^{X_2}_X(\vartheta_2)$. Then $b_Y(\vartheta)$ is the unique block covering $b_{X_1}(\vartheta_1)$ and $b_{X_2}(\vartheta_2)$
\end{lemma}

\begin{proof}
Write $b=b_Y(\vartheta)$, $b_0=b_X(\vartheta_0)$ and $b_i=b_{X_i}(\vartheta_i)$ for $i=1,2$. Restriction of characters defines a bijective map $\Irr(Y/X) \to \Irr(X_1/X) \times \Irr(X_2/X)$. Since $Y/X$ is abelian, every block of $X_i$ (resp. $Y$) covering $b_0$ contains an extension of $\vartheta_0$.  Since these differ by multiplication with a character of $\Irr(X_i/X)$ (resp. $\Irr(Y/X)$), this map induces a surjective map  
$\mathrm{Bl}(Y \mid b_0) \to \mathrm{Bl}(X_1 \mid b_0) \times \mathrm{Bl}(X_2 \mid b_0)$. 
Note that this is well-defined, since if $b_Y(\vartheta)=b_Y(\vartheta \la)$ for some $\la \in \Irr(Y/X)$ then $b_{X_i}(\vartheta_i \Res_{X_i}^{Y}(\la))$ and $b_{X_i}(\vartheta_i)$ are $Y$-invariant and lie below $b_{Y}(\vartheta)$ and must hence be equal. To show that it is bijective, it thus suffices to show that the two sets have the same cardinality. By Corollary \ref{cor:linear action}, for $H \in \{X_1,X_2,Y\}$ we have $|\mathrm{Bl}(H \mid b_0)|=|H[b_0]:X|$. As 
$Y[b_0]=X_1[b_0] X_2[b_0]$,
this immediately shows that the two sets have the same cardinality, and the lemma follows. 
\end{proof}

\subsection{Unipotent Blocks and Disconnected Groups}

We now return to the situation of disconnected reductive groups $\hat\G$ with $(\hat \G)^\circ=\G$ and keep the notation from before. We will sometimes need to work with good primes satisfying some stronger hypotheses.

\begin{definition}\label{def:goodprime}
 Let $\ell$ be an odd prime and $\G$ a connected reductive group.   We will say $\ell$ satisfies Definition \ref{def:goodprime} for $\G$ if $\ell$ is a good prime for $(\G,F)$ distinct from the defining characteristic of $\G$,  $\ell \nmid |(\Z(\G)/\Z(\G)^\circ)^F|$, and $\ell\neq 3$ if $\G^F$ has a component of type $\tw{3}\type{D}_4$.
\end{definition}

We observe that if $t \in (\G^\ast)^F_\ell$ then $\C^\circ_{\G^\ast}(t)$ is an $F$-stable Levi subgroup of $\G^\ast$ and we denote by $\G(t)$ an $F$-stable Levi subgroup of $\G^F$ in duality with it. From \cite{CE94} (see also \cite[Secs.~21, 23]{CE04}),  the unipotent blocks of $G$ are of the form $b_G(\bL, \la)$ for a unipotent $d$-cuspidal pair $(\bL, \la)$. (Recall here that a unipotent $d$-cuspidal pair $(\bL, \la)$ consists of a $d$-split Levi subgroup $\bL$ for $(\bG, F)$ and a  cuspidal unipotent character $\la\in\Irr(\bL^F)$.) Further, for an $\ell$-element $t \in (\G^\ast)^F_\ell$ we have 
$$\Irr(b_G(\Levi,\la)) \cap \mathcal{E}(G,t)=\{ R_{\G(t)}^{\G}(\hat{t} \chi_t) \mid \chi_t \in \mathcal{E}(\G(t), (\Levi_t,\lambda_t)) \text{ with } (\Levi,\la) \sim (\Levi_t, \lambda_t) \}.$$

The following will be useful in computing Dade groups in this situation. 

\begin{lemma}\label{lem:DadeCompute}
Assume the prime $\ell$ satisfies Definition \ref{def:goodprime} for $\bG$.
	Let $G \lhd \hat{G}$ such that $\hat{G}/G$ is abelian of $\ell'$-order. Suppose that $b=b_G(\Levi,\lambda)$ is a unipotent $\ell$-block of $G$ with $\hat{L}G=\hat{G}$ where $\hat{L}:=\C_{\hat{G}}(\Z(\Levi)^F_\ell)$. If $\lambda$ is $\hat{L}$-stable and extends to a character of $\N_{\hat{G}}(\Levi)_\lambda$, then $\hat{G}[b]=\hat{G}$. 
\end{lemma}

\begin{proof}
We first note that there exists a defect group $D$ of $b$ such that $\N_G(D) \leq N:=\N_G(\Levi)$. Indeed, by \cite[Lem.~4.16]{CE99} (see also \cite[Sec.~5.2]{CE99} for the case that $\ell=3$), 
we have $Z:=\Z(\Levi)_\ell^F\lhd D$ is the unique maximal abelian subgroup, so $\N_G(D)\leq \N_G(Z)\leq \N_G(\C_{\bG}^\circ(Z))= N$.
	
By Theorem \ref{thm:dade}, it suffices to show that $|\mathrm{Bl}(\hat{G} \mid b)|=|\hat{G}:G|$. Set $\hat{N}:=N \hat{L}$. Recall from \cite[Props.~ 13.16, 13.19]{CE04} that $\Levi=\C_{\G}(\Z(\Levi)_\ell^F)$. In particular, there exists by \cite[ Cor.~9.21]{Nav98} a unique block $B=b_L(\la)^N$ of $N=\N_G(\Levi)$ covering $b_{L}(\lambda)$. By the Harris--Kn\"orr correspondence \cite[Thm.~9.28]{Nav98} applied to both $G\lhd \hat G$ and $N\lhd \hat N$, it then suffices to show that $|\mathrm{Bl}(\hat{N} \mid B)|=|\hat{L}:L|$. ({Indeed, note here that the Brauer correspondents of $B$ and $b$ must coincide since $b_L(\la)^G=b$ and $b_L(\la)^N=B$.})

For $\hat{\lambda} \in \Irr(\hat{L} \mid \lambda)$, we observe that since $\hat{L}=\C_{\hat{G}}(\Z(L)_\ell)$ by construction, there exists a unique block $\hat{B}(\hat{\lambda})$ of $\hat{N}$ covering $b_{\hat{L}}(\hat{\lambda})$ again by \cite[ Cor.~9.21]{Nav98}. Therefore we have a map $\Irr(\hat{L} \mid \lambda) \to \mathrm{Bl}(\hat{N} \mid B)$ defined by $\hat{\lambda} \mapsto \hat{B}(\hat{\lambda})$. It suffices to show that the so-obtained map is injective (and therefore bijective). 
For this, suppose that $\hat{B}(\hat{\lambda}')=\hat{B}(\hat{\lambda})$ for some $\hat\la, \hat{\lambda}' \in \Irr(\hat{L} \mid \lambda)$. Then $b_{\hat{L}}(\hat{\lambda})$ is $\hat{N}$-conjugate to $b_{\hat{L}}(\hat{\lambda}')$. Since $\lambda$ is the canonical character of $b_L(\lambda)$ by \cite[Prop.~22.16]{CE04}, it follows that $\hat{\lambda}$ is the canonical character of $b_{\hat{L}}(\hat{\lambda})$. Hence, $\hat{\lambda}$ and $\hat{\lambda}'$ are conjugate by an element $\hat{n} \in \hat{N}$. Since both of them are extensions of the same character $\lambda$, it follows that $\hat{n} \in \hat{N}_\lambda$. From our assumption that $\la$ extends to $\hat N_\la$, we have $\hat{n} \in \hat{N}_{\hat{\lambda}}$ and so $\hat{\lambda}=\hat{\lambda}'$.
\end{proof}

\begin{lemma}\label{lem:ext char}
Let $\G$ be a connected reductive group and $\G \lhd \hat{\G}$, as before. Suppose that $(\Levi,\la)$ is a unipotent $d$-cuspidal pair. Then $\la$ extends to its inertia group in $\N_{\hat{\G}^F}(\Levi)$.
\end{lemma}

\begin{proof}
 This essentially follows from \cite[Thms.~4.5.11, ~4.5.12]{GM20}, together with arguing as in the proof of \cite[Lem.~7.5]{RSST}. Namely, letting $\Levi_0:=\Levi/\Z(\Levi)$ and $\hat{N}:=\N_{\hat{\G}^F}(\Levi)$, we observe that $\Levi_0^F \lhd \hat{N}_\la/\C_{\hat{N}}(\Levi) \leq \mathrm{Aut}_F(\Levi_0^F)_\la$.
It suffices to show that $\la_0$, the image of $\la$ in $\Levi_0^F$, extends to $\mathrm{Aut}_F(\Levi_0^F)_\la$. 
(Here $\mathrm{Aut}_F(\Levi_0^F)$ is defined as in \cite[Sec.~2.4]{CS13} to be the set containing all restrictions of $F$-equivariant isogenies $\sigma: \G \to \G$ that stabilize a maximal torus of $\Levi_0$.)

Now, $\Levi_0$ is a product of simple algebraic groups, and the group $\mathrm{Aut}_F(\Levi_0)$ is a product of wreath product over the $F$-simple factors. By \cite[Thm.~4.5.12]{GM20}, we know that if $\mathbf{H}$ is $F$-simple then the unipotent characters of $\bH^F$ extend to their inertia group in $\mathrm{Aut}_F(\mathbf{H}^F)$. Choosing the same extension for factors permuted by the wreath product, these will extend to the wreath product as desired.
\end{proof}

\begin{remark}\label{rem:weyl central product}
Keep the situation of Lemma \ref{lem:DadeCompute}. If $\sigma_0 \in GE[b_G(\Levi,\la)]$, then $\sigma_0 g$ centralizes some defect group of $b:=b_G(\Levi, \la)$ for some $g\in G$ thanks to Theorem \ref{thm:dade}. Since by \cite[Lem.~4.13]{Nav98}, $\Z(\bL)_\ell^F$ is contained in such a defect group as it is contained in any defect group of $b_L(\la)$ and $b_L(\la)^G=b$, there is $g \in G$ such that $\sigma:=\sigma_0 g$ centralizes $\Z(\Levi)_\ell^F$. As $\Levi=\C_{\G}(\Z(\Levi)_\ell^F)$ we therefore get $\sigma \in \N_{GE}(\Levi)$. For $\hat{G}:=G \langle \sigma \rangle$, $\hat{L}:=\C_{\hat{G}}(\Z(\Levi)_\ell^F)$ and $W_{\hat{G}}(\Levi):=\N_{\hat{G}}(\Levi)/L$ we then get 
$$W_{\hat{G}}(\Levi)=W_G(\Levi) \times \hat{L}/L.$$ Indeed for $z\in \Z(\Levi)_\ell^F$ and $n \in \N_G(\Levi)$, we have ${}^n z \in \Z(\Levi)_\ell^F$ and so ${}^{\sigma n} z={}^n z={}^{n \sigma} z$. Hence, the commutator $[\sigma,n]$ lies in $ \hat{L} \cap G=\C_G(\Z(\Levi)_\ell^F)=L$.
\end{remark}

The following can be compared to \cite[Thm.~21.7]{CE04}.

\begin{proposition}\label{prop:decomp RLG}
Assume the prime $\ell$ satisfies Definition \ref{def:goodprime} for $\G$,
$[\G,\G]$ is simply connected, and $\hat{\G}/\G$ is abelian of $\ell'$-order.
Let $\hat{\Para}$ be a parabolic subgroup of $\hat{\G}$ with Levi subgroup $\hat{\Levi}$. Suppose that $\Levi=\C_\G(\Z^\circ(\Levi)^F_\ell)$ and let $\lambda \in \mathcal{E}(\Levi^F,1)$ with $\hat{\Levi}^F[b_{\Levi^F}(\la)] \G^F= \hat{\G}^F$. Let $\hat{\la} \in \Irr(\hat{\Levi}^F \mid \la)$ and $b_{\hat L}:=b_{\hat{\Levi}^F}(\hat\la)$ its block. Then all constituents of $R_{\hat{\Levi} \subset \hat{\Para}}^{\hat{\G}}(\hat{\la}')$ for all $\hat{\la}' \in \Irr(b_{\hat{L}}) \cap \mathcal{E}(\hat{\Levi}^F,1)$ lie in the same $\ell$-block $\hat{b}$ of $\hat\G^F$. Moreover, $(1,\hat{b}) \leq (\Z(\Levi)_\ell^F, b_{\hat{L}})$.
\end{proposition}

\begin{proof}
Let $L:=\Levi^F$, $\hat L:=\hat\Levi^F$, $b_L:=b_L(\la)$, $b_{\hat{L}}:=b_{\hat{L}}(\hat{\la})$, $G:=\bG^F$, and $\hat{G}:=\hat\G^F$. Observe that our assumption $\hat L[b_L] G=\hat{G}$ implies that $\hat{L}= \hat L[b_L] $. Moreover, every defect group of $b_L$ contains the central subgroup $\Z(\Levi)_\ell^F$. Since $\hat{\Levi}^F/\Levi^F$ is abelian of $\ell'$-order and every unipotent character of $\Levi^F$ extends to its inertia group in $\hat{\Levi}^F$ (see \cite[Thm.~4.5.12]{GM20}), it follows from Clifford correspondence that there is a character in $b_{\hat L}$ with multiplicity-free restriction to $L$. Then from Theorem \ref{thm:dade} and Corollary \ref{cor:linear action}, we have $|\mathrm{Bl}(\hat L \mid b_L)|=|\hat L[b_L]/L|=|\hat L/L|$. 
In particular, the block $b_{\hat{L}}$ is isomorphic to $b_L$ via restriction. From this, combined with Theorem \ref{thm:dade} and the fact that $\hat L G=\hat G$, we deduce that $\hat{L}=\C_{\hat{\G}^F}(\Z(\Levi)_\ell^F)$.

The claim when $\G=\Levi$ is therefore clear.

We can thus assume that $\Levi \neq \G$. Let $\hat{\chi} \in \Irr(\hat{G})$ with $\langle \hat{\chi}, R_{\hat{\Levi} \subset \hat{\Para}}^{\hat{\G}}(\hat{\lambda}') \rangle \neq 0$ for some $\hat{\lambda}' \in \Irr(b_{\hat{L}}) \cap \mathcal{E}(\hat{\Levi}^F,1)$. By assumption, there exists $x \in \Z(\Levi)_\ell^F \setminus \Z(\G)^F_\ell$. Note that our assumption also implies that $\C_{\hat{\Levi}^F}(x)=\hat{\Levi}^F$. In particular, $\Levi=\C_\Levi(x)$ is a $d$-split Levi subgroup of $\mathbf{C}:=\C_{\G}(x)$ in the parabolic subgroup $\C_{\Para}(x)$.
    Consider $f:=b_{\hat{L}} d^1({}^\ast R_{\hat{\Levi} \subset \hat{\Para}}^{\hat{\G}}(\hat{\chi}))$. By Lemma \ref{lem:preserveunip}, ${}^\ast R_{\hat{\Levi}\subset \hat{\Para}}^{\hat{\G}}(\hat{\chi}) \in \mathbb{Z} \mathcal{E}(\hat{L},1)$ is a $\mathbb{Z}$-linear combination of characters of $\hat{L}$ covering unipotent characters of $L$. We therefore have $f=\sum_{\mu \in \mathcal{E}(\hat L,1) \cap \Irr(b_{\hat{L}})} m_\mu d^1(\mu)$ 
    with $m_{\hat{\la}'} \neq 0$.
    Note that $\ell \nmid |(\Z(\Levi)/\Z^\circ(\Levi))^F|$ by \cite[Prop.~13.12]{CE04} so $\mathcal{E}(L,1)$ is a basic set of characters for the unipotent blocks of $L$ by \cite[Thm.~14.4]{CE04}. As argued above, the block $b_{\hat{L}}$ is isomorphic to $b_L$ via restriction. 
    Then the set $\mathcal{E}(\hat L,1) \cap \Irr(c)$ defines a basic set for each block $c$ covering $b_L$.
    Hence, the set $d^1(\mathcal{E}(\hat L,1) \cap \Irr(b_{\hat{L}}))$ is linearly independent, which implies that $f \neq 0$. In particular, there exists some $\mu \in \mathcal{E}(\hat L,1) \cap \Irr(b_{\hat{L}})$ such that $\langle d^1(\mu),f \rangle \neq 0$ since otherwise $\langle f,f \rangle =0$.

    We argue by induction on the rank of $\G$ minus the rank of $\Levi$. By the induction hypothesis, there is a unique block $b_{\hat{C}}$ with Brauer pair $(\Z(\Levi)_\ell^F,b_{\hat{L}})$ of $\hat{C}=\C_{\hat{G}}(x)$ which contains all irreducible constituents of $R_{\hat{\Levi} \subset \hat{\Para} \cap \hat{\mathbf{C}}}^{\hat{\mathbf{C}}}(\hat{\la}')$ for $\hat{\la}' \in \Irr(b_{\hat{L}})$.

    Then 
    $\langle d^x \hat{\chi}, R_{\hat{\Levi} \subset \hat{\Para} \cap \mathbf{\hat{C}}}^{\hat{\mathbf{C}}}(\mu) \rangle
   = \langle d^1 {}^\ast R_{\hat{\Levi} \subset \hat{\Para}}^{\hat{\G}}(\hat{\chi}), \mu \rangle$ using Lemma \ref{lem:dx neu} and the fact that all characters in $R_{\hat{\Levi} \subset \hat{\Para} \cap \mathbf{\hat{C}}}^{\hat{\mathbf{C}}}(\mu)$ cover unipotent characters of ${C}$ (again by Lemma \ref{lem:preserveunip}) which have $\Z(\mathbf{C}^\circ)^F$ (hence $x$) in their kernel.
From before we have $0 \neq \langle f, d^1(\mu) \rangle= \langle f, \mu \rangle$.
 Hence, $\langle d^x \hat{\chi}, R_{\hat{\Levi} \subset \hat{\Para} \cap \mathbf{\hat{C}}}^{\hat{\mathbf{C}}}(\mu) \rangle \neq 0$. By assumption, all constituents of $R_{\hat{\Levi} \subset \hat{\Para} \cap \mathbf{\hat{C}}}^{\hat{\mathbf{C}}}(\mu)$ lie in $b_{\hat{C}}$, which implies that $b_{\hat{C}} d^x(\hat{\chi}) \neq 0$. By Brauer's second main theorem, it follows that $(1,b_{\hat{G}}(\hat{\chi})) \leq ( \langle x \rangle, b_{\hat{C}} )$. The induction hypothesis now gives $(\Z(\Levi)_\ell^F,b_{\hat{L}})$ is a $b_{\hat{C}}$-Brauer pair. Hence, by transitivity of Brauer pairs we have that $(\Z(\Levi)_\ell^F,b_{\hat{L}})$ is a $b_{\hat G}(\hat \chi)$-Brauer pair.
\end{proof}

As in the case for connected reductive groups, we define $$\mathcal{E}(\hat{G}, (\hat{\Levi},\hat{\la}))=\{ \hat{\chi} \in \Irr(\hat{G}) \mid \langle R_{\hat{\Levi} \subset \hat{\Para}}^{\hat{\G}}(\hat{\la}), \hat{\chi} \rangle \neq 0 \}.$$
We remark some consequences of our theorem. 

\begin{corollary}\label{cor:disj HC}
    In the situation of Proposition \ref{prop:decomp RLG}, we have the following:
\begin{enumerate}
    \item The block $\hat{b}$ does not depend on the choice of the parabolic subgroup $\hat{\Para}$ containing $\hat{\Levi}$.
    \item 
    We have a disjoint union
    $$\Irr(\hat{G} \mid \mathcal{E}(G,(\Levi,\la)) = \displaystyle\bigcup_{\hat{\la} \in \Irr(\hat{L} \mid \la)} \mathcal{E}(\hat{G}, (\hat{\Levi},\hat{\la})).$$
\end{enumerate}
\end{corollary}

\begin{proof}
    The first statement is obtained as in \cite[Rem.~21.12]{CE04}. For the second statement, we first observe that the right-hand side is a disjoint union as the characters in  $\mathcal{E}(\hat{G}, (\hat{\Levi},\hat{\la}))$ lie in different $\ell$-blocks by Proposition \ref{prop:decomp RLG}.

    Let $\chi\in\mathcal{E}(G, (\bL, \la))$. Recall that for each $\hat{\la} \in \Irr(\hat{\Levi} \mid \la)$ we have $\Res_L^{\hat{L}}(\hat{\la})=\la$. Then by Frobenius reciprocity and the fact that restriction is compatible with generalized Lusztig induction by \cite[Cor.~2.4]{DM94}, we obtain 
    $\langle \chi, R_\Levi^{\G}(\lambda) \rangle =\langle \Ind_{G}^{\hat{G}}(\chi),  R_\Levi^{\G}(\hat{\lambda}) \rangle \neq 0$. By Lemmas \ref{lem:DadeCompute} and \ref{lem:ext char}, we have $\hat{G}[b_G(\Levi, \la)]=\hat{G}$, so by Theorem \ref{thm:dade} all characters of $\hat{G}$ covering $\chi$ lie in different $\ell$-blocks.
    But since each constituent of  $R_{\hat{\Levi}}^{\hat{\G}}(\hat{\lambda})$ lies in the same $\ell$-block by Proposition \ref{prop:decomp RLG}, we must have $\langle \Ind_{G}^{\hat{G}}(\chi),  R_{\hat{\Levi}}^{\hat{\G}}(\hat{\lambda}) \rangle=\langle \hat{\chi},R_{\hat{\Levi}}^{\hat{\G}}(\hat{\lambda}) \rangle$ for a unique character $\hat{\chi}$ extending $\chi$. That is, $\hat\la$ determines $\hat\chi$. Then given $\hat{\chi}' \in \Irr(\hat{G} \mid \chi)$ with $\langle \chi, R_\Levi^{\G}(\lambda) \rangle \neq 0$, we have $\hat\chi'=\hat\chi\beta$ for some $\beta\in\Irr(\hat G/G)$ and $\langle\hat\chi', R_{\hat{\Levi}}^{\hat{\G}}(\hat{\lambda}\beta) \rangle\neq0$.
    
Conversely, for $\hat{\chi} \in \mathcal{E}(\hat{G}, (\hat{\Levi},\hat{\la}))$ covering $\chi\in\Irr(G)$, we similarly have $\langle \Ind_{G}^{\hat{G}}(\chi),  R_{\hat{\Levi}}^{\hat{\G}}(\hat{\lambda}) \rangle=\langle \hat{\chi},R_{\hat{\Levi}}^{\hat{\G}}(\hat{\lambda}) \rangle \neq 0$ so that $\langle \chi, R_\Levi^{\G}(\la) \rangle \neq 0$. Hence $\hat\chi\in\Irr(\hat{G}\mid\mathcal{E}(G, (\Levi, \la))$.
\end{proof}

One of the obstructions to applying Proposition \ref{prop:decomp RLG} is the assumption that $\C_{\hat{\G}^F}(\Z(\Levi)_\ell^F) \leq \hat{\Levi}^F$.
    While this condition seems necessary to control the relevant block decomposition, it is on the other hand unclear when it is actually satisfied. Lemma \ref{lem:Levi existence} below will give a partial solution to this question. First, we need the following:
    
\begin{lemma}\label{lem:sigma fixed}
    Let $\sigma$ be a quasi-semisimple automorphism of $\G$. In addition, let $\Levi$ be a Levi subgroup of $\G$ such that $\Levi=\C^\circ_{\G}(S)$ 
    for some subgroup $S$ contained in a torus of $(\bG^\sigma)^\circ$ satisfying that every prime dividing $|S/S^\circ|$ is good for $\G^x$ and is coprime to $|\Z(\G^x)/\Z^\circ(\G^x)|$ for any quasi-central $x \in \G \sigma $.
    (Here, as usual, $\bG^\sigma$ denotes the set of fixed points under $\sigma$.) Then $\Levi$ is contained in a $\sigma$-stable parabolic subgroup of $\G$. 
\end{lemma}

\begin{proof}
Fix a maximal torus $\T$ contained in a Borel subgroup $\B$ of $\G$ such that $(\T,\B)$ is $\sigma$-stable.
By \cite[Thm 1.8]{DM94} it follows that $(\T^\sigma)^\circ$ is a maximal torus of $(\G^\sigma)^\circ$.

As $S$ is contained in a maximal torus of $(\G^\sigma)^\circ$ it follows that there is some $g \in (\G^\sigma)^\circ$ such that ${}^g S \leq (\T^\sigma)^\circ$.

Hence, $\Levi_0:={}^g \Levi=\C_{\bG}({}^gS)$ is $\sigma$-stable as $\Levi$ is $\sigma$-stable. Note here that $\bT\leq \Levi_0$ since $\tw{g}S\leq \bT$. By \cite[Prop.~1.16]{DM94}, there is some element $t \in \T$ such that $\sigma':=t \sigma$ is quasi-central for $\G$. It follows that $\B_{\Levi_0}=\B \cap \Levi_0$ is a $\sigma'$-stable Borel subgroup of $\Levi_0$. 
Moreover, since $\sigma'$ is quasi-central for $\G$, for every $n \in \N_{\Levi_0}(\T) \subset \N_{\G}(\T)$ such that $\sigma'(n)n^{-1} \in \T$ there is some $t' \in \T$ such that $\sigma'(nt')=nt'$ by \cite[Thm-Def.~1.15]{DM94}. Then again by \cite[Thm-Def.~1.15]{DM94} the automorphism $\sigma'$ is also quasi-central for $\Levi_0$. Hence, $\sigma_0:={}^{g^{-1}} \sigma'$ is quasi-central for $\Levi$ and $\G$ (as conjugates of quasi-central elements are again quasi-central).

It follows that $(\G^{\sigma_0})^\circ$ is a connected reductive group, see \cite[Thm.~1.8]{DM94}. We claim that $(\Levi^{\sigma_0})^\circ$ is a Levi subgroup of $(\G^{\sigma_0})^\circ$. Indeed, recall $\Levi=\C^\circ_{\G}(S)$ so that $(\Levi^{\sigma_0})^\circ=\C^\circ_{(\G^{\sigma_0})^\circ}(S)$.
    This implies that $(\Levi^{\sigma_0})^\circ$ is a Levi subgroup  of $(\G^{\sigma_0})^\circ$ by \cite[Prop.~13.16]{CE04}.

    Let $\mathbf{M}:=\C_{\G}(\Z^\circ(\Levi^{\sigma_0}))$ 
    and note that $\mathbf{M} \leq \Levi$ since $S \leq \Z(\Levi)^{\sigma_0} \leq \Z(\Levi^{\sigma_0})$ which gives $S \leq \Z^\circ(\Levi^{\sigma_0})$ by our assumption on $S$. Then $\mathbf{M}$ is a $\sigma_0$-stable Levi subgroup of $\G$ contained in a $\sigma_0$-stable parabolic $\Para$ by \cite[Cor.~1.25]{DM94}.

    Note that $\Levi$ is $\sigma_0$-stable and $\mathbf{M}$ is a Levi subgroup of $\Levi$ in the $\sigma_0$-stable parabolic $\Para \cap \Levi$. However, $\mathbf{M}^{\sigma_0}=\Levi^{\sigma_0}$ which implies $\Levi=\mathbf{M}$ by the lemma after \cite[Prop.~1.23]{DM94}. In particular, $\Levi$ is contained in a $\sigma_0$-stable parabolic $\Para$. Recall that $g \in \G^\sigma$ so that $\sigma_0= {}^{g^{-1}} t\cdot\sigma$. Further,  ${}^{g^{-1}} t \in \Levi$, so it follows that $(\Levi,\Para)$ is also $\sigma$-stable.
\end{proof}

Recall that for $d$ a positive integer prime to $p$, a $d$-torus of $(\bG, F)$ is an $F$-stable torus in $\bG$ whose order polynomial is a power of the $d$th cyclotomic polynomial $\Phi_d$; a Sylow $d$-torus is a $d$-torus of maximal possible rank; and a $d$-split Levi
subgroup is a centralizer of a d-torus. (See \cite[Sec.~3.5]{GM20}.) 

In the case where a Sylow $d$-torus of $\G$ is contained in $(\G^\sigma)^\circ$, Lemma \ref{lem:sigma fixed} gives the following:

\begin{lemma}\label{lem:sylow fixed}
    Let $\sigma$ be a quasi-central automorphism of a connected reductive group $\G$ of $p'$-order commuting with the action of $F$. If $(\G^\sigma)^\circ$ contains a Sylow $d$-torus of $(\bG, F)$, then there is a bijection between the $\G^F$-conjugacy classes of $d$-split Levi subgroups of $\G$ and the  $((\G^\sigma)^\circ)^F$-conjugacy classes of $d$-split Levi subgroups of $(\G^\sigma)^\circ$ as follows: Every $d$-split Levi subgroup of $\G$ is $\G^F$-conjugate to a Levi subgroup $\Levi_1$ contained in a $\sigma$-stable parabolic, and the map sends the conjugacy class of $\Levi_1$ to the conjugacy class of $(\Levi_1^\sigma)^\circ$. 
\end{lemma}

\begin{proof}
 Let $\bS$ be a Sylow $d$-torus of $(\bG, F)$ contained in $(\bG^\sigma)^\circ$.
Let $\Levi$ be a $d$-split Levi subgroup of $\bG$. Then $\Levi$ is $\G^F$-conjugate to a Levi subgroup $\Levi_1$ with $\Z^\circ(\Levi_1)_{\Phi_d} \leq \mathbf{S}$. Since $\Levi_1$ is $d$-split, we have $\Levi_1=\C_{\G}(\Z^\circ(\Levi_1)_{\Phi_d})$ (see \cite[Prop.~3.5.5]{GM20}) and so $\Levi_1$ is contained in a $\sigma$-stable parabolic subgroup by Lemma \ref{lem:sigma fixed}. 
We also obtain that $(\Levi_1^\sigma)^\circ$ is a $d$-split Levi subgroup of $(\G^\sigma)^\circ$. 

If $\Levi_1'$ is $\G^F$-conjugate to $\Levi_1$ and is contained in a $\sigma$-stable parabolic, then one can show as in the proof of \cite[Prop.~1.40]{DM94} that $\Levi_1'^\sigma$ and $\Levi_1^\sigma$ are $((\G^\sigma)^\circ)^F$-conjugate. Hence, the map $\Levi_1 \mapsto (\Levi_1^\sigma)^\circ$ is well-defined on conjugacy classes. On the other hand, any $d$-split Levi subgroup of $(\G^\sigma)^\circ$ is, up to $((\G^\sigma)^\circ)^F$-conjugation, a centralizer of an $F$-stable subtorus of $\mathbf{S}$. Hence, the map is surjective and it is injective as $\Levi_1=\C_{\G}(\Z^\circ((\Levi_1^\sigma)^\circ))$ by \cite[Prop.~1.23]{DM94}.
\end{proof}

We now obtain our partial answer to when the conditions of Proposition \ref{prop:decomp RLG} hold.

\begin{lemma}\label{lem:Levi existence}

Assume the prime $\ell$ satisfies Definition \ref{def:goodprime} for $\G$, $[\G,\G]$ is simply connected, and $\hat{\G}/{\G}$ is cyclic of $\{p,\ell\}'$-order.
Let $\Levi$ be a Levi subgroup of $\G$ such that $\Levi=\C_{\G}(\Z^\circ(\Levi)_\ell^F)$ and let $(\Z(\Levi)_\ell^F,b_L)$ be a Brauer pair of a block $b_G$ of $G:=\bG^F$ with $\hat{G}[b_G]=\hat{G}$, where $\hat G=\hat\G^F$. Then there exists a Levi subgroup $\hat{\Levi}$ with connected component $\Levi$ such that $\hat{\Levi}^F[b_L] \G^F=\hat{\G}^F$.
    \end{lemma}

    \begin{proof}
    By assumption together with Theorem \ref{thm:dade}, we find some $\sigma \in \hat{\G}^F$ of $\{\ell,p\}'$-order such that $\sigma$ generates the quotient $\hat{\G}^F/ \G^F$ and centralizes a defect group of $b_G$. In particular, by \cite[Lem.~4.13]{Nav98}, $\Z(\bL)_\ell^F$ is contained in such a defect group as it is contained in any defect group of $b_L$ and $(\Z(\Levi)_\ell^F,b_L)$ is a $b_G$-Brauer pair. Hence we can (by possibly replacing $\sigma$ by a $\G^F$-conjugate) assume that $\sigma$ centralizes $\Z(\Levi)_\ell^F$. Moreover, since $\sigma$ is of $p'$-order and in particular semisimple, it follows that $\sigma$ is quasi-semisimple by \cite[Prop.~1.3]{DM94}.
    Recall $\Levi=\C_{\G}(Z)$ for $Z:=\Z^\circ(\Levi)_\ell^F$ an $\ell$-subgroup of $\G^{\sigma}$. On the other hand, $Z \leq \mathbf{T}:=\Z^\circ(\Levi)$, the latter is $\sigma$-stable, and $\T^\sigma/(\T^\sigma)^\circ$ is an $\ell'$-group by \cite[Prop.~1.28]{DM94}. Hence $Z\leq (\bT^\sigma)^\circ\leq (\bG^\sigma)^\circ$. By Lemma \ref{lem:sigma fixed} we deduce that $\Levi$ is contained in a $\sigma$-stable parabolic $\Para$.

Hence, we can set $\hat{\Para}:=\Para \langle \sigma \rangle$ and $\hat{\Levi}=\Levi \langle \sigma \rangle$. Note that $\Z(\Levi)_\ell^F$ is centralized by $\sigma$ so that $\hat{\Levi}^F=\C_{\hat{\G}^F}(\Z(\Levi)_\ell^F)$.  This implies that $\N_{\hat{\G}}(\Levi^F)[b_L] \leq \hat{\Levi}^F$ {using \cite[Thm.~3.13]{Murai}},
as any defect group for $b_L$ contains $\Z(\Levi)_\ell^F$. On the other hand $\hat{G}[b_G]=\hat G$ and $(\Z(\Levi)_\ell^F,b_L)$ is a $b_G$-subpair. From this we have $|\mathrm{Bl}(\hat G\mid b_G)|=|\hat G/G|$ by Theorem \ref{thm:dade}. Then applying the Brauer homomorphism to the central idempotents and using \cite[Thm.~4.14]{Nav98}, we have the same holds for the blocks above $b_L$ in $\hat{L}=\C_{\hat{G}}(\Z(\Levi)_\ell^F)$. Then again applying Theorem \ref{thm:dade}, we have $\hat{L}[b_L]=\hat L$. Since $\sigma \in \hat L$, we also know $\hat L G=\hat G$, and this completes the proof.
\end{proof}

We summarize our findings in the following theorem, which is a restatement of Theorem \ref{thm:unipotent characters} and aims to give an analogue to part of \cite{CE94} for disconnected groups.  Recall here the notation $d_\ell(q)$ for the order of $q$ modulo $\ell$ when $\ell$ is odd.

\begin{theorem}\label{unipotent characters}

Assume the prime $\ell$ satisfies Definition \ref{def:goodprime} for $\G$, $[\G,\G]$ is simply connected, and $\hat{\G}/{\G}$ is cyclic of $\{p,\ell\}'$-order. 
Let $d=d_\ell(q)$. If $(\Levi,\la)$ is a unipotent $d$-cuspidal  pair of $(\G,F)$ defining an $\ell$-block $b=b_G(\Levi,\la)$ of $\G^F$ such that $\hat{\G}^F[b]=\hat{\G}^F$, then there exists an $F$-stable Levi subgroup $\hat{\Levi}$ contained in a parabolic subgroup $\hat{\Para}$ of $\hat{\G}$ such that $\C_{\hat{\G}^F}(\Z(\Levi)_\ell^F)=\hat{\Levi}^F$. Moreover, there exists a bijection
$$\Irr(\hat{\Levi}^F \mid \la) \to \mathrm{Bl}(\hat{\G}^F \mid b), \hat{\la} \mapsto b_{\hat{\G}^F}(\hat{\Levi},\hat{\lambda})$$
such that
\begin{enumerate}
    \item $\mathcal{E}(\hat{\G}^F,1) \cap \Irr( b_{\hat{\G}^F}(\hat{\Levi},\hat{\lambda}))= \{ \hat{\chi} \in \Irr(\hat{\G}^F) \mid \langle R_{\hat{\Levi} \subset \hat{\Para}}^{\hat{\G}}(\hat{\la}), \hat{\chi} \rangle \neq 0 \}$
\item$(1,b_{\hat{\G}^F}(\hat{\Levi},\hat{\lambda})) \leq (\Z(\Levi)_\ell^F, b_{\hat{\Levi}^F}(\hat{\la})))$.
\end{enumerate}
\end{theorem}
\begin{proof}
By \cite[Prop.~3.2]{CE99}, the assumption that $\Levi$ is $d$-split yields that $\C_{\bG}(\Z^\circ(\bL)_\ell^F)=\Levi$. Then our assumptions and the results of \cite[Thm.~2.5]{CE99} imply that Lemma \ref{lem:Levi existence} is satisfied, so there is $\hat\bL$ with $\hat\bL^\circ=\bL$ with $\hat\bL^F[b_L(\la)]\bG^F=\hat\bG^F$. Hence, Proposition \ref{prop:decomp RLG} applies. In particular, the proof of Proposition \ref{prop:decomp RLG} gives $\C_{\hat{\G}^F}(\Z(\Levi)_\ell^F)=\hat{\Levi}^F$. 

The rest eventually follows from Proposition \ref{prop:decomp RLG} and Corollary \ref{cor:disj HC}. We note that the map $\hat\la\mapsto b_{\hat{G}}(\hat\Levi, \hat\la)$ is well-defined by Proposition \ref{prop:decomp RLG}, surjective by Lemma \ref{lem:preserveunip}, and injective by Corollary \ref{cor:disj HC} together with (1). Hence it now suffices to see (1). 

To see (1), we observe that for $\hat\la\in\Irr(\hat\Levi^F\mid\la)$, any $\hat{\chi } \in \Irr(\hat{\G}^F)$ with $\langle R_{\hat{\Levi} \subset \hat{\Para}}^{\hat{\G}}(\hat{\la}), \hat{\chi} \rangle \neq 0$ lies in $\mathcal{E}(\hat{\G}^F,1)$ by Lemma \ref{lem:preserveunip} and in $ \Irr(b_{\hat{\G}^F}(\hat{\Levi},\hat\la))$ by Proposition \ref{prop:decomp RLG}. On the other hand, any character $\hat{\chi} \in \mathcal{E}(\hat{\G}^F,1) \cap \Irr(b_{\hat{\G}^F}(\hat{\Levi},\hat\la))$ must cover a character $\chi \in \mathcal{E}(\G^F,1) \cap \mathcal{E}(\G^F,(\Levi,\la))$, since the analogous statement to (1) for $b_G(\bL, \la)$ holds by \cite[Thm.]{CE94}. Hence, $\langle \chi ,R_\Levi^{\G}(\la) \rangle \neq 0$ and $\langle \Ind_{ G}^{\hat{G}}(\chi), R_{\hat{\Levi} \subset \hat{\Para}}^{\hat\G}(\hat{\la}) \rangle \neq 0$ by \cite[Cor.~2.4]{DM94} and Frobenius reciprocity. By Theorem \ref{thm:dade}, the irreducible constituents of $\Ind_{G}^{\hat{G}}(\chi)$ are in distinct blocks. But by Proposition \ref{prop:decomp RLG}, any $\psi\in\Irr(\hat G)$ with $\langle \psi, R_{\hat{\Levi} \subset \hat{\Para}}^{\hat\G}(\hat{\la}) \rangle \neq 0$ must lie in $b_{\hat{G}}(\hat\Levi, \hat\la)=b_{\hat{G}}(\hat\chi)$.
In particular, it must be that $\langle R_{\hat{\Levi} \subset \hat{\Para}}^{\hat{\G}}(\hat{\la}), \hat{\chi} \rangle \neq 0$ as desired. 
\end{proof}

\subsection{Unipotent blocks and Alvis--Curtis duality}\label{sec:ACblockswap}

Recall that if $\chi \in \Irr(G)$, there exists a unique sign $\varepsilon_\chi \in \{\pm 1\}$ and $\chi^\ast \in \Irr(G)$ such that $\D_\G(\chi)=\varepsilon_\chi \chi^\ast$. In particular, if $\hat{\chi} \in \Irr(\hat{G})$ is an extension of $\chi$, then $\D_{\hat{\G}}(\hat{\chi})=\varepsilon_\chi \hat{\chi}^\ast$ for some $\hat{\chi}^\ast \in \Irr(\hat{G})$ extending $\chi^\ast$. (Recall the construction of $\D_{\hat\G}$ in Section \ref{sec:AC}.)
As the Mackey formula holds for unipotent characters, we obtain by \cite[Thm.~3.44]{GM20} 
$$\varepsilon_\G \varepsilon_\Levi \D_\G R_\Levi^{\G}(\la)=R_\Levi^{\G} \D_\Levi(\la).$$ In particular, $(\Levi,\la^\ast)$ is again a $d$-cuspidal pair and labels the block $\D_\G(b_\G(\Levi,\la))$. Let $\hat{\G}=\G \langle \sigma \rangle$ for some $\sigma$ such that every power of $\sigma$ is a rational quasi-central automorphism. If $\Levi$ is $\sigma$-stable and $\hat{\Levi}=\Levi \langle \sigma \rangle$ is a Levi subgroup of $\hat{\G}$, then under suitable assumptions \cite[Thm.~3.11]{DM94} shows that this equality generalizes to 
$$\varepsilon \D_{\hat{\G}} R_{\hat{\Levi}}^{\hat{\G}}(\hat \la)= R_{\hat{\Levi}}^{\hat{\G}} \D_{\hat{\Levi}}(\hat\la),$$
where $\varepsilon$ is the class function defined by $\varepsilon(G \sigma^i):=\varepsilon_{\G^{\sigma^i}} \varepsilon_{\Levi^{\sigma^i}}.$
We want to compute the image of $b_{\hat{\G}}(\hat{\Levi},\hat{\la})$ in the situation of Theorem \ref{unipotent characters} under $\D_{\hat{\G}}$. Under some reasonable assumptions we compute this image and we will describe two interesting cases below where these assumptions are satisfied and illustrate what can happen in these cases.

The formula above shows that $\langle \hat{\chi},R_{\hat{\Levi}}^{\hat{\G}}(\hat{\la}) \rangle \neq 0$ if and only if $\langle \varepsilon \D_{\hat{\G}}(\hat{\chi}),R_{\hat{\Levi}}^{\hat{\G}}(\D_{\hat{\Levi}}(\hat{\la})) \rangle \neq 0.$ 
On the other hand, since $\D_\G$ fixes $b_\G(\Levi,\la)$, there exists a linear character $\varepsilon_0$ such that $\D_{\hat{\G}}$ maps the block $b_{\hat{G}}(\hat{\Levi},\hat{\la})$ to its tensor product with $\varepsilon_0$. Using the fact that $\langle \chi,R_\Levi^{\G}(\la) \rangle=\langle \hat{\chi},R_{\hat{\Levi}}^{\hat{\G}}(\hat{\la}) \rangle$ for the unique character $\hat{\chi} \in \Irr(\hat{G} \mid \chi)$ lying in the block $b_{\hat{\G}}(\hat{\Levi},\hat{\la})$, we deduce that
$$\varepsilon_\G \varepsilon_\Levi \varepsilon_0 \D_{\hat{\G}} R_{\hat{\Levi}}^{\hat{\G}}(\hat{\la})=R_{\hat{\Levi}}^{\hat{\G}} \D_{\hat{\Levi}}(\hat{\la}).$$
If we assume additionally that $R_{\hat{\Levi}}^{\hat{\G}} \D_{\hat{\Levi}}(\hat{\la})$ does not vanish on any coset of $G$ in $\hat{G}$ then we have
$$\varepsilon=\varepsilon_\G \varepsilon_\Levi \varepsilon_0$$
which in particular shows that $\varepsilon \in \varepsilon_\G \varepsilon_\Levi \Lin(\hat{G})$. Since $\varepsilon \in \pm \Lin(\hat{G})$ we thus have $$\D_{\hat{\G}}(b_{\hat{G}}(\hat{\Levi},\hat{\la}))=\varepsilon b_{\hat{G}}(\hat{\Levi},\hat{\la}^\ast)$$
in the situation of Theorem \ref{unipotent characters}.

Assume that $\Levi$ is a torus so that the formula in \cite[Thm.~3.11]{DM94} holds. Then we have $\la^\ast=\la=1_{\Levi^F}$. Moreover, by \cite[Thm.~4.13]{DM94} we have $R_{\hat{\Levi}}^{\hat{\G}}(\D_{\hat{\Levi}}(\hat{\la}))(\sigma^i)=R_{\Levi^{\sigma^i}}^{\G^{\sigma^i}}(1) \D_{\hat{\Levi}}(\hat{\la})(\sigma^i) \neq 0$.
In particular, if $\varepsilon \neq \pm 1_{\hat{G}}$ then $\D_{\G}$ fixes the block $b_G(\Levi,\la)$ while $\D_{\hat{\G}}$ maps $b_{\hat{G}}(\hat{\Levi},\hat{\la})$ to $\varepsilon b_{\hat{G}}(\hat{\Levi},\hat{\la})$. See Example \ref{ex:steinberg} for a specific situation in which this occurs. 

Assume now that $(\Levi,\la)$ is a cuspidal pair and that $\hat{\Levi}$ is contained in an $F$-stable parabolic of $\hat{\G}$. In this situation the conclusion of \cite[Thm.~3.11]{DM94} holds again. Then $\la$ is a cuspidal character and the formula for (generalized) Alvis--Curtis duality shows that $\D_\Levi(\la)=(-1)^{\mathrm{rk}(\Levi)} \la$ and $\D_{\hat{\Levi}}(\hat{\la})=(-1)^{\mathrm{rk}(\Levi)} \hat{\la}$. Moreover, $\varepsilon=1$ which shows that $\D_{\G}(b_G(\Levi,\la))=b_G(\Levi,\la)$ and also $\D_{\hat{\G}}(b_{\hat{G}}(\hat{\Levi},\hat{\la}))=b_{\hat{G}}(\hat{\Levi},\hat{\la})$.

\subsection{Height zero characters in blocks covering unipotent blocks.}

We end this section by describing the non-unipotent height zero characters in the unipotent blocks.

\begin{lemma}\label{lem:hz}
In the situation of Theorem \ref{unipotent characters}, 
$$\Irr_0(b_{\hat{G}}(\hat{\Levi},\hat{\la}))= \{\hat{\chi} \in \Irr(\hat{G} \mid \Irr_0(b_G(\Levi,\la)) ) \mid \langle \hat{\chi},R_{\hat{\Levi} \subset \hat{\Para}}^{\hat{\G}}(\hat{\zeta}) \rangle \neq 0 \text{ for some } \hat{\zeta} \in \Irr(b_{\hat{L}}(\hat{\la})) \}.$$
\end{lemma}

\begin{proof}
Set $b_G:=b_G(\Levi,\la)$ and $b_{\hat G}:=b_{\hat{G}}(\hat{\Levi},\hat{\la})$. {Note that restriction gives a bijection between $\Irr_0(b_{\hat G})$ and $\Irr_0(b_G)$, by applying Theorems \ref{unipotent characters} and \ref{thm:dade}, with the assumption that $\hat{G}/G$ is $\ell'$.}

Let $t\in G^\ast$ have order a power of $\ell$.
Let $\chi \in \mathcal{E}(G,t) \cap \Irr_0(b_G)$ and let $\bH:=G(t)$ be a Levi subgroup of $\G$ in duality with  $\C_{\G^\ast}^\circ(t)$. Write $\hat{t}$ for the linear character of $H=\bH^F$ dual to $t \in \Z(H^\ast)$. Hence, there exists a unique character $\chi_t \in \mathcal{E}(H,1)$ such that $\chi=R_{\bH}^{\bG}(\chi_t \hat{t})$. By \cite[Thm.~21.13]{CE04}, all constituents of $R_{\bH}^{\bG}(\chi_t)$ lie in $b_G(\Levi,\la)$. Hence, by \cite[Thm.~21.7]{CE04}, $(\Z(H)_\ell,b_H(\chi_t))$ is a $b_G$-Brauer pair.
By \cite[Thm.~23.2]{CE04}, the character $\chi_t$ lies in the $d$-Harish-Chandra series of $(\Levi_H,\lambda_H)$, where $(\Levi_H, \lambda_H) \sim (\Levi,\la)$ in the sense of  \cite[Def.~23.1]{CE04}.

We see that $(\Z(\Levi_H)_\ell^F,b_{L_H})$ is a $b_{H}$-Brauer pair, where we write $b_{L_H}:=b_{\Levi_H^F}(\la_H)$ and $b_H:=b_H(\chi_t)$,
and 
$(\Z(\Levi)_\ell^F,b_{L}(\la))$ is a $b_G$-Brauer pair. Write $L_H:=\Levi_H^F$ and $L=\Levi^F$. Let $D$ be a Sylow $\ell$-subgroup of $\C_{\G}^\circ([\Levi,\Levi])^F$ and 
$D_H\leq D$ a Sylow $\ell$-subgroup of $\C_{\bH}^\circ([\Levi_{\bH},\Levi_{\bH}])^F$ (this is possible as $[\Levi,\Levi]=[\Levi_{\bH},\Levi_{\bH}]$) 
such that $(\Z(\Levi_H)_\ell^F,b_{L_H}) \lhd (D_H, b_{\C_H(D_H)}(\Res_{\C_H(D_H)}^{L_H}(\la_H))$ 
and $(\Z(\Levi)_\ell^F,b_L(\la)) \lhd (D,b_{\C_G(D)}(\Res_{\C_G(D)}^{L}(\la))$ (see \cite[Lem.~22.18]{CE04}; these are defect groups for $b_G$ and $b_H$, respectively, by \cite[Thm.~22.9]{CE04}). 

By assumption, $\chi$ is a height zero character and $\chi(1)_\ell=|G:H|_\ell \chi_t(1)_\ell$ with $\chi_t(1)_\ell\geq|H:D_H|_\ell$ so we must necessarily have $D=D_H$. As $\Z(\Levi_H)_\ell^F$ and $\Z(\Levi)_\ell^F$ are the unique maximal abelian normal subgroups of $D_H$ and $D$ respectively (see \cite[Lem.~4.16]{CE99}), it follows that $\Z(\Levi_H)_\ell^F=\Z(\Levi)_\ell^F$ so that $\Levi \cap \mathbf{H}=\Levi_{\mathbf{H}}$, using that $\bL=\C_{\bG}(\Z(\Levi)_\ell^F)$ and similar for $\Levi_H$. As $[\Levi,\Levi]=[\Levi_{\bH},\Levi_{\bH}]$ it therefore follows that $\Levi=\Levi_{\bH}$ {using e.g. \cite[Prop.~22.8]{CE04}}. 
Now note that $\lambda$ and $\lambda_H$ are unipotent characters that restrict to the same unipotent character of $[\Levi,\Levi]^F$. Hence, they must coincide.

Note that the block $b_H$ of the Levi subgroup $\mathbf{H}^F$ satisfies the assumptions of Lemma \ref{lem:Levi existence}. Hence, there exists a parabolic subgroup of $\hat{\G}$ with Levi subgroup $\hat{\mathbf{{H}}}$ with $\hat{\mathbf{{H}}}^\circ=\mathbf{H}$ and such that $\hat{H}[b_H]G= \hat{G}$, where $\hat H=\hat\bH^F$. In particular, we know that every character in $\Irr(b_H)$ is $\hat{H}$-stable. Let $\sigma$ be an automorphism whose image in $\hat{H}/H$ generates the quotient. Then its dual automorphism $\sigma^\ast$ of $\mathbf{H}^\ast$ satisfies $\sigma(\mathcal{E}(H,t))=\mathcal{E}(H,\sigma^\ast(t))$. Hence, $\sigma^\ast(t)$ and $t$ must be $H^\ast$-conjugate. 

As $\bH^\ast=\C_{\G^\ast}(t)$ it follows that $t= \sigma^\ast(t)$. In other words, the linear character $\hat{t}$ is $\sigma$-stable. 
Since $\hat{t}$ has $\ell$-power order and $\hat{H}/H$ is an $\ell'$-group, this character has a unique extension $\hat{t}'$ to $\hat{H}$ of $\ell$-power order.
We apply Proposition \ref{prop:decomp RLG}, which shows that for any $\hat{\zeta} \in \mathcal{E}(\hat{H},1) \cap \Irr(b_{\hat{H}})$, where $b_{\hat{H}}:=b_{\hat{H}}(\hat{\Levi},\hat{\lambda})$, all constituents of $R_{\hat{\bH}}^{\hat{\G}}(\hat{\zeta})$ lie in $b_{\hat{G}}$. On the other hand, all constituents of $R^{\hat{\bH}}_{\hat{\Levi}}(\hat{\lambda})$ lie in $b_{\hat{H}}$ again by Proposition \ref{prop:decomp RLG}.  We have $R_{\hat{\bH}}^{\hat{\G}}(\hat{t}' R^{\hat{\bH}}_{\hat{\Levi}}(\hat{\la}))=R_{\hat{\Levi}}^{\hat{\G}}(\hat{\la} \hat{t}')$. Since $d^1 R_{\hat{\Levi}}^{\hat{\G}}(\hat{\la} \hat{t}')=d^1(R_{\hat{\Levi}}^{\hat{\G}}(\hat{\la}))$, these characters lie in $b_{\hat{G}}$. Since $\hat{t}' R_{\hat{\bH}}^{\hat{\G}}: \mathcal{E}(\hat{H},1) \to \mathcal{E}(\hat{G},t)$ is a signed bijection and the characters in $\mathcal{E}(\hat{H},1) \cap \Irr(b_{\hat{H}})$ are precisely the constituents of $R^{\hat{\bH}}_{\hat{\Levi}}(\hat{\lambda})$ {by Theorem \ref{unipotent characters}}, it follows that every height zero character of $b_{\hat{G}}$ occurs in this way.
(Indeed, given $\hat\chi\in\Irr_0(b_{\hat G})$, we take $\chi:=\Res^{\hat G}_G(\hat \chi)\in\Irr_0(b_G)\cap \mathcal{E}(G, t)$ to be as above. Recall that $\chi=R_\bH^\bG(\hat t\chi_t)$ is in $R_\bH^\bG(\hat t R_\Levi^\bH(\la))$ since $\mathcal{E}(H, 1)\cap \Irr(b_H)$ is exactly the set of constituents of $R_\bL^\bH(\la)$. Since the constituents of $R_{\hat{\bH}}^{\hat{\G}}(\hat{t}' R^{\hat{\bH}}_{\hat{\Levi}}(\hat{\la}))$ lie in $\Irr(b_{\hat{G}})$, the same argument as the last several lines of the proof of Theorem \ref{unipotent characters} gives that $\langle \hat\chi, R_{\hat{\bH}}^{\hat{\G}}(\hat{t}' R^{\hat{\bH}}_{\hat{\Levi}}(\hat{\la}))\rangle\neq 0$. Hence $\hat\chi$ is obtained in the desired way, noting that $\hat t'\hat\la$ lies in $b_{\hat L}(\hat\la)$ since $d^1(\hat{t}')=1$.)

For the converse, it suffices to see that every $\hat\chi\in\Irr(\hat G\mid \Irr_0(b_G))$ with $\langle\hat\chi, R_{\hat \Levi}^{\hat\G}(\hat\zeta)\rangle\neq0$ for some $\hat\zeta\in\Irr(b_{\hat L}(\hat \la))$ lies in $b_{\hat{G}}$, as any $\hat\chi\in\Irr(b_{\hat{G}} \mid \Irr_0(b_G))$ necessarily has height zero. We claim that every such $\hat\zeta$ is of the form $\hat t' \hat\la$ as above, and hence the constituents of $R_{\hat \Levi}^{\hat\G}(\hat\zeta)$ lie in $b_{\hat{G}}$ from the above discussion. Indeed, first recall from the proof of Proposition \ref{prop:decomp RLG} that the block $b_{\hat{L}}(\hat\la)$ is isomorphic to $b_L(\la)$ via restriction. Now, by \cite[Prop.~22.16]{CE04} the block $b_L(\la)$ is of central defect. Then the map $\Z(\Levi^\ast)_\ell^F \to \Irr(b_L(\la)), t \mapsto \hat{t} \la,$ is injective because the images are in different Lusztig series for different $t \in \Z(\Levi^\ast)^F_\ell$ and surjective since $|\Irr(b_L(\la))|=|\Z(\Levi)^F_\ell|$ by the properties of blocks of central defect and the fact that $|\Z(\Levi^\ast)_\ell^F |=|\Z(\Levi)_\ell^F|$ by our assumptions on $\ell$. 
\end{proof}

\section{Application to the Galois version of the Alperin--McKay conjecture}\label{sec:AMN}

Throughout this section we assume that $\G$ is a simple, simply connected group whose root system is of classical type but not of type $\tD_4$, {write $G:=\G^F$ for $F$ a Frobenius morphism endowing $\G$ with an $\mathbb{F}_q$-rational structure for some power $q$ of a prime $p$, and we let $\ell\neq p$ be an odd prime not dividing $|\Z(\G^F)|$. (Note that then $\ell$ satisfies Definition \ref{def:goodprime} for $\G$.) We further let $d:=d_\ell(q)$ and let} $(\Levi,\la)$ be a unipotent $d$-cuspidal pair of $(\G,F)$, and write $E:=E(\bG^F)$ and $N:=\N_G(\bL)$.  

In \cite[Thm.~8.8]{RSST}, we have constructed an $\mathcal{H}$-equivariant bijection
$$\Irr_{0}(GE \mid \mathcal{E}(G,(\Levi,\la))) \to \Irr_{0}(\N_{GE}(\Levi) \mid \la).$$
(Note that in \cite{RSST}, heights are not mentioned. However, the preservation of height-zero characters follows from the construction of the bijection, together with Clifford theory and that we know the corresponding bijection for $G$ and $N$ is height-preserving, using the character formula for example in  \cite[Prop.~3.1]{Ros24}.)

In this section, we want to study the compatibility of this bijection with $\ell$-blocks. 
For this, note that Lemmas \ref{lem:DadeCompute} and \ref{lem:ext char} will be helpful to compute the Dade group.

Recall that if $(\G,F)$ is split and $\Levi$ is a $d$-split Levi subgroup, then the polynomial order of $\Z^\circ(\Levi)$ is divisible only by cyclotomic polynomials $\Phi_i$ with $i \mid 2d$ (see e.g. \cite[Ex.~3.5.15]{GM20}). In particular, $\Z(\Levi)_\ell^F=(\Z(\Levi)_{\Phi_d})_\ell^F$.

Throughout this section, we fix a maximally split $E$-stable torus $\T_0$ of $\G$ contained in an $E$-stable Borel subgroup $\B_0$ of $\G$. It follows that the $d$-split Levi subgroup $\Levi$ is $\G$-conjugate to a standard Levi subgroup $\Levi_I$. We let $n:=g^{-1} F(g) \in \N_{\G}(\T_0)$, where $g \in \G$ is such that ${}^g \Levi_I=\Levi$. Since $\Levi$ is $d$-split, we can assume that the image $w$ of $n$ in $W:=\N_{\G}(\T_0)/\T_0$ is a Sylow $d$-twist of $(\G,F)$ (as defined in \cite[Def.~3.1]{Spa10a}) and we assume that $n \in V$, where $V$ is the extended Tits' group, and the order of $n$ is divisible only by the same prime divisors as the order of $w$. 

    \begin{lemma}\label{lem:linear}
Keep the assumptions on $\G$, $\ell$ above, and further assume $(\G,F)$ is split and let $F_{p_0}$ be a field automorphism with $p_0^{f_0}=q$ and $\ell \nmid f_0$. If $\ell$ is odd and $d_\ell(p_0)=d_\ell(q)$, then all unipotent blocks have $F_{p_0}$ in their Dade ramification group. 
    \end{lemma}

    \begin{proof}

We keep the situation of the preceeding paragraph and write $F_0:=F_{p_0}$. Our assumption implies that $(d,f_0)=1$ where $d:=d_\ell(p_0)=d_\ell(q)$. In particular, as $w$ has order dividing $d$ and $n$ is divisible by the same prime divisors as $w$, it follows that the order of $n$ is coprime to $f_0$. Hence, there is some $n_0 \in \langle n \rangle$ such that $n_0^{f_0}=n$.  Note that $W$ is a rational group, so the images of $n_0$ and $n$ in $W$ are $W$-conjugate. This implies that $n_0 $ is a Sylow $d$-twist of $(\G,F_0)$, using that $d_\ell(p_0)=d_\ell(q)$ and the polynomial order depends only on the $W$-conjugacy class of the element in the Weyl group. 

We claim that $\Z(\Levi_I)_\ell^{n_0 F_0}=\Z(\Levi_I)_\ell^{nF}$. Since $V^{F_0}=V$ we have $(n_0 F_0)^{f_0}=nF$, so it suffices to show that these two groups have the same cardinality. For this observe that $(\Z(\Levi_I)_{\Phi_d})_\ell^{nF}=\Z(\Levi_I)_\ell^{nF}$ since $\ell$ is odd.  
We then need to show that $\Phi_d(p_0)_\ell=\Phi_d(q)_\ell$. For this, it suffices to note that $\frac{p_0^{df_0}-1}{p_0^d-1}$ is not divisible by $\ell$, as this is the same as  $1+p_0^d+\dots p_0^{d(f_0-1)} \equiv df_0 \not \equiv 0 \mod \ell$. Hence, $\Z(\Levi_I)_\ell^{nF}$ is centralized by $n_0 F_0$, which gives the claim.  

Now, we see from the structure of $[\Levi, \Levi]$ from the discussion in \cite[Ex.~3.5.29]{GM20}, that $n_0 F_0$ acts as a field automorphism on $[\Levi,\Levi]^F$ and as such an automorphism it fixes every unipotent character of $[\Levi,\Levi]^F$. In particular, $n_0 F_0$ stabilizes the unipotent character $\lambda$. From these considerations, we deduce using Lemma \ref{lem:DadeCompute} together with Lemma \ref{lem:ext char} that $F_0$ is in the Dade group of $b_G(\Levi,\la)$. 
\end{proof}

Before considering the analogue of Lemma \ref{lem:linear} for twisted groups, we make the following observation:

\begin{lemma}\label{lem:graph cyclo}
    For a graph automorphism $\gamma$ of $\G$ order $2$, we have that $|\G^{\gamma F_q}: (\G^\gamma)^{F_q}|_{p'}$ is a product of cyclotomic polynomials $\Phi_e(q)$, where each $e$ is a positive even integer.
\end{lemma}

\begin{proof}
We consider each case and the remarks after \cite[Prop.~1.22]{DM94}.
    If $\G$ is of type $\tD_n$, then $\G^\gamma$ is of type $\tB_{n-1}$ and $| \G^{\gamma F_q}: (\G^\gamma)^{F_q}|_{p'}=(q^n+1)$.
    If $\G$ is of type $\tA_n$ then $\G^\gamma$ is of type $\type{X}_{\lfloor \frac{n+1}{2} \rfloor}$ with $\type{X}\in\{\tC, \tB\}$. The $p'$-part of the order of ${}^2 \tA_n(q)$ is $\prod_{i=1}^{n} (q^{i+1}-(-1)^{i+1})$. If we count only when $i+1$ is even and set $j:=(i+1)/2$ we get $\prod_{j=1}^{\lfloor (n+1)/2 \rfloor} (q^{j}-1)$ which is precisely the contribution of odd cyclotomic polynomials to the order of $(\G^\gamma)^{F_q}=\type{X}_{\lfloor \frac{n+1}{2} \rfloor}(q)$.
\end{proof}
\begin{remark}
Although we have assumed $\bG$ is classical throughout this section, we remark that the case that $\G=\tE_6$ and $\G^\gamma=\tF_4$ of Lemma \ref{lem:graph cyclo} also holds and is dealt with in the same way.
We also note that the assumption that $\gamma$ has order two in Lemma \ref{lem:graph cyclo} is necessary, since if $G=\tD_4$ and $\gamma$ is a triality automorphism, then $\G^\gamma=\tG_2$ so $|\G^{\gamma F_q}: (\G^\gamma)^{F_q}|$ is divisible by $\Phi_3(q)$.
\end{remark}

We next prove an analogue of Lemma \ref{lem:linear} for twisted groups.

\begin{lemma}\label{lem:linear twisted}
    Suppose that $F=F_q\gamma$ for some graph automorphism $\gamma$ of order $2$. Let $F_0=F_{p_0}$ for some integer $p_0$ with $p_0^{f_0}=q$ and $\ell \nmid 2 f_0$.  If $d=d_\ell(q)$ is odd and $d_\ell(p_0)=d_\ell(q)$, then all unipotent blocks have $F_0$ in their Dade ramification group. 
\end{lemma}

\begin{proof}
 Let $b:=b_G(\bL, \la)$ be a unipotent block of $G$. Recall that there exists a standard Levi subgroup $\Levi_I$ (relative to $\T_0 \subset \B_0$) such that $\Levi_I$ is $\G$-conjugate to $\Levi$. Moreover, the classification of unipotent $d$-cuspidal pairs shows that the subset $I$ can always be chosen to be $\gamma$-stable (compare to Lemma \ref{lem:sylow fixed}).

    It follows that $\Levi_I^{\gamma}$ is a Levi subgroup of the connected reductive group $\G^{\gamma}$ (see \cite[Prop.~1.11]{DM94} and \cite[Rem.~1.30]{DM94}). Consider a Sylow $d$-twist $n \in \N_{\G^\gamma}(\T_0^\gamma)$ of $(\G^\gamma,F_q)$. This means by construction that $(\T,n F_q)$ contains a Sylow $d$-torus of $(\G^\gamma,n F_q)$. By Lemma \ref{lem:graph cyclo}, $|\G^{\gamma F_q}: (\mathbf{\G}^{\gamma})^{F_q}|_{p'}$ is only divisible by cyclotomic polynomials $\Phi_e(q)$ where $e$ is even. As $(\T^\gamma)^{n F_q} \leq \T^{\gamma n F_q}$ it therefore follows that $(\T,n \gamma F_q)$ contains a Sylow $d$-torus of $(\G, \gamma n F_q)$ and so $n$ is also a Sylow $d$-twist of $(\G,\gamma F_q)$.

Now observe that $(\G^\gamma,F_q)$ is a split simple, simply connected group of classical type. Therefore, we can apply the proof of Lemma \ref{lem:linear} find some $n_0 \in \langle n \rangle \leq V^\gamma$ such that $n_0^{f_0}=n$ and $\Z(\Levi_I)_\ell^{nF}$  is centralized by $n_0 F_0$. We can therefore conclude again with Lemma \ref{lem:DadeCompute} together with Lemma \ref{lem:ext char}.
\end{proof}

On the other hand, if $d_\ell(p_0) \neq d_{\ell}(q)$, we have a restriction on the structure of the element in the Dade group:

\begin{lemma}\label{lem:square is twist}
Suppose that $F_0:=F_{p_0}$ is a field automorphism and $q=p_0^{f_0}$.  Continue to assume that $\G$ is of classical type and suppose that $f_0$ is a power of $2$ and that $d:=d_\ell(q)$ is odd. Suppose that $b=b_G(\Levi,\la)$ is a unipotent $\ell$-block
such that $F_{0}$ is in the Dade ramification group of $b$. Then
    $$W_{G \langle F_0 \rangle}(\Levi) \cong W_G(\Levi) \times \C_{G \langle F_0 \rangle}(\Z(\Levi)_\ell^F)/\Levi^F.$$
  Moreover, if $\C_{G \langle F_0 \rangle}(\Z(\Levi)_\ell^F)/\Levi^F= \langle \ov{F}_0 \rangle$ then there is an isomorphism $$W_{G \langle F_0 \rangle}(\Levi) \cong (W_{G}(\Levi_I)^{w F} \times \langle F_0 \rangle)/ \langle w F \rangle,$$
    under which $\ov{F}_0$ corresponds to $vF_{p_0}$ such that $v^2 \in \langle w \rangle$. If $v \notin \langle w \rangle$ then $d_\ell(p_0)=2 d_\ell(q)$. In this case $F_{0}$ is necessarily a generator of the Sylow $2$-subgroup of $G \langle F_p \rangle[b]$.
\end{lemma}

\begin{proof}
As argued before we can find $x \in G$ such that $x F_0$ fixes a defect group of $D$, which we can assume to contain $\Z(\Levi)_\ell^F$. Hence, $xF_0$ normalizes $\Levi=\C_{\G}(\Z(\Levi)_\ell^F)$ and we can denote by  $\overline{F_{p_0}}$ its class in $W_{G \langle F_0 \rangle}(\Levi)$.

Write $F=F_q \gamma_0$ for $\gamma_0$ a possibly trivial graph automorphism.
Let $g \in \G^{\gamma_0}$ with Lang image $n \in \N_G(\T_0)^{\gamma_0}$ under $F$ and $w=n \T_0$ as in the proof of Lemma \ref{lem:linear} and Lemma \ref{lem:linear twisted} respectively.
By conjugating with the Lang preimage $g$ of $n$, we obtain an isomorphism 
$$W_{G \langle F_0 \rangle}(\Levi) \cong (W_G(\Levi_I)^{w F} \times \langle F_0 \rangle)/ \langle w F \rangle.$$
Now if $v F_{p_0}$ corresponds to $\overline{F_{p_0}}$ via this isomorphism, then it centralizes $W_{G \langle F_0 \rangle}(\Levi)$ by Remark \ref{rem:weyl central product}. Since $\G$ is of classical type, we know that $W_G(\Levi)$ is isomorphic to a subgroup of index $1$ or $2$ of $S_{d'} \wr S_a$ for some $d' \in \{d,2d\}$ and some positive integer $a$. (See, for example \cite[Ex.~3.5.29]{GM20}.) 
    
    However, $\Z(W_{G \langle F_0 \rangle}(\Levi_I)^{w F})=\Z(W_{G}(\Levi_I)^{wF}) \langle F_0 \rangle$  and so $\Z(W_{G}(\Levi_I)^{wF})=\langle x_0, w \rangle$, where $x_0$ is the longest element of $W$, if the longest element is in $\Z(W^F)$ and $x_0=1$ otherwise. It follows that $v^2 \in \langle w \rangle$ and so $v^2$ is a Sylow $d$-twist of $(\bG,F)$, since $d$ is odd and $f_0$ is a power of $2$. 
  
    On the other hand, if $d_\ell(p_0)=d_\ell(q)$ then $v \in \langle w \rangle$ by the proofs of Lemmas \ref{lem:linear} and \ref{lem:linear twisted}. 
    Hence, if $v \not\in \langle w \rangle$ it follows that $d_\ell(p_0)=2 d_\ell(q)$ as $f_0$ is a power of $2$. 
    If $F_{p_0'} \in G \langle F_p \rangle [b]$ then the image of the square of $\overline{F}_{p_0'}$ must be a Sylow $d$-twist, forcing $(F_{p_0'})_2 \in \langle F_{p_0} \rangle$.
\end{proof}

Recall $(\underline{\G},\underline{F})$ from Lemma \ref{lem: model} such that that for any $F$-stable subgroup $\mathbf{H}$ of $\G$ we have a corresponding $\underline{F}$-stable subgroup $\underline{\mathbf{H}}$ whose image under the map $\mathrm{pr}: \underline{\G} \to \G$ is $\mathbf{H}$. 
We fix the normal subgroup $\underline{\G}\lhd\hat\G\lhd \underline{\G} \, \underline{E}$ such that the quotient $\hat\G/\underline{\G}$
corresponds to the Sylow $2$-subgroup of $\langle F_p \rangle \subset E$  under the isomorphism $\underline{\G}^{\underline{F}} \underline{E} \cong \G^F E$. Moreover, the group $\underline{\Levi}$ is an $F$-stable Levi subgroup of $\underline{\G}$ which is contained in the parabolic $\underline{\Para}$ such that $\Levi$ (resp. $\Para$) is the image of $\underline{\Levi}$ (resp. $\underline{\Para}$) under the map $\mathrm{pr}: \underline{\G} \to \G$. In addition, we have $$\mathrm{pr}|_{\underline{\G}^{\underline{F}}} \circ R_{\underline{\Levi} \subset \underline{\Para}}^{\underline{\G}} \circ \mathrm{pr}^{-1}|_{\Levi^F}= R_{\Levi \subset \Para}^{\G}.$$

Note that $\underline{F}$ defines an $\mathbb{F}_p$-structure on $\underline{\G}$ so for an odd prime $\ell_0$ the unipotent $\ell_0$-blocks of $\underline{\G}^{\underline{F}}$ are parametrized by the unipotent $d_{\ell_0}(p)$-cuspidal pairs. One can now check that the Levi subgroup $\underline{\Levi}$ constructed above is $d_{\ell_0}(p)$-split for any prime $\ell_0$. This follows from the fact that since $\Levi=\C_{\G}(\mathbf{S})$ for $\mathbf{S}=\Z^\circ(\Levi)_{\Phi_d}$ then $\underline{\Levi}=\C_{\underline{\mathbf{G}}}(\underline{\mathbf{S}})$ 
and if $\mathbf{S}_0=\underline{\mathbf{S}}_{\Phi_{d_0}}$ where $d_0=d_{\ell_0}(p)$ then $\underline{\mathrm{pr}(\mathbf{S}_0)}=\underline{\mathbf{S}}$  as $\underline{F}$ acts transitively on the simple components of $\underline{\G}$ (or by comparing polynomial orders noting that if $q=p^f$ then $\Phi_d(q)=\prod_{e \mid f, (d,f/e)=1} \Phi_{de}(p)$ and $d_0=d (f,d_0) $). This gives $\underline{\Levi}=\C_{\underline{\G}}(\mathbf{S}_0)$.

Therefore $(\underline{\Levi},\underline{\la})$, where $\underline{\la}$ is the character corresponding to $\la$ under $\mathrm{pr}|_{{\underline{\Levi}^F}}$, is a unipotent $d_{\ell_0}(p)$-cuspidal pair defining the same $\ell$-block as $(\Levi,\la)$ under the isomorphism between $\underline{\G}^{\underline{F}}$ and $\G^F$.

We wish to apply Theorem \ref{unipotent characters} in our situation which will require the prime $p$ to be odd. We first discuss how we can circumvent this problem when $p=2$ and start with the following consequence of Lang's theorem.

\begin{lemma}\label{lem:Lang}
Let $F=F_q \gamma_0$ for some (possibly trivial) graph automorphism $\gamma_0$ and $q=p_0^{f_0}$. Suppose that $g \in \G^{\gamma_0} F_p$ satisfies $(g F_{p_0})^{f_0}=F_q$. Then $g F_{p_0}$ and $F_{p_0}$ are $(\G^{\gamma_0})^{F_q}$-conjugate.
\end{lemma}

\begin{proof}
    As $\G^{\gamma_0}$ is connected by \cite[Rem.~1.30]{DM94} we can apply Lang's theorem to find some $t \in \G^{\gamma_0}$ such that $t F_{p_0}(t^{-1})=g$. Hence, ${}^t F_{p_0}=g F_{p_0}$ and so ${}^t F_q=(gF_{p_0})^{f_0}=F_q$ which shows that $t \in (\G^{\gamma_0})^{F_q}$.
\end{proof}

\begin{remark}\label{rem:char2}
Suppose that we are in the situation of Lemma \ref{lem:square is twist} and that $p=2$.
    In this case, we have $V \cong W$. Let $\gamma_0$ again denote some (possibly trivial) graph automorphism such that $F=F_q\gamma_0$. Using the construction in Lemma \ref{lem:square is twist} shows that $\C_{G \langle F_0 \rangle}(\Z(\Levi)^F_\ell)=\Levi^F \langle y \rangle$ for some $\gamma_0$-stable element $y$ with $y^{f_0}=F_{q}$. Hence, $y$ is in particular $G$-conjugate to $F_{p_0}$ by Lemma \ref{lem:Lang}. Then the corresponding element in $\underline{\G}^{\underline{F}} \underline{E}$ is quasi-semisimple (as it is $\underline{\G}^{\underline{F}}$-conjugate to a quasi-semisimple element in $\underline{E}$) and the proof of Theorem \ref{unipotent characters} therefore goes through.
\end{remark}

Note that the following result does not mention the prime $\ell_0$ anymore and depends only on the $d$-Harish-Chandra series (that is only on the integer $d$).

\begin{corollary}\label{cor:disjoint HC}
Assume that $d$ is an odd integer. Then there exists a Levi subgroup $\hat{\Levi}$ contained in a parabolic subgroup $\hat{\Para}$ of $\hat{\G}$ with connected component $\underline{\Levi}$ such that the generalized $d$-Harish-Chandra series $R_{\hat{\Levi} \subset \hat{\Para}}^{\hat{\G}}(\hat{\la})$ are disjoint for distinct extensions of $\hat{\la} \in \Irr(\hat{\Levi}^{\underline{F}} \mid \lambda)$.
\end{corollary}

\begin{proof}
 We can assume that $1 \neq d$ since otherwise the statement follows from \cite[{Sec.~2}]{RSF22} together with Lemma \ref{lem:ext char}. 
We want to apply Corollary \ref{cor:disj HC}. For this it suffices to check the assumptions of Theorem \ref{unipotent characters} (see Remark \ref{rem:char2} in case $p=2$) for an odd prime $\ell_0$ with $d_{\ell_0}(q)=d$. In other words, it suffices to find a prime $\ell_0$ such that the Dade group of the unipotent $\ell_0$-block associated to $(\Levi,\la)$ contains a Sylow $2$-subgroup of $\hat{G}/G$, {where $\hat{G}=\hat\bG^{\underline{F}}$}. We will show that we can choose a prime $2 \neq \ell_0$ with $d_{\ell_0}(q)=d$.

Let $p_0$ be a power of $p$ and $f_0$ be such that $q=p_0^{f_0}$.  As $1 \neq d$ is odd, it follows that $\Phi_d(p_0)$ is odd so it has at least one odd prime divisor and we can choose $\ell_0$ to be a Zsigmondy prime divisor of $p_0^d-1$. Note that since $\bG$ is classical and $\ell_0$ is odd, $\ell_0$ is necessarily a good prime. 
If $f_0$ is coprime to $d$ (which happens in particular if $f_0$ is a power of $2$), then the order of $q$ mod $\ell_0$ is also $d$. Then choosing $p_0$ such that $f_0$ is a power of $2$, so that $F_{p_0}$ is in a Sylow $2$-subgroup of $\hat{G}/G$, we have $d=d_{\ell_0}(p_0)=d_{\ell_0}(q)$. 
As $d_{\ell_0}(q) \notin \{1,2\}$ it also follows that $\ell_0 \nmid (q \pm 1)$ so $\ell_0 \nmid |\Z(G)|$.

 By Lemma \ref{lem:linear} and Lemma \ref{lem:linear twisted} it follows that $\hat{G}[b_{\hat{G}}]=\hat{G}$. Then the result is a consequence of Theorem \ref{unipotent characters} for odd $p$. If $p=2$ we use Remark \ref{rem:char2} to conclude the result.
\end{proof}

\begin{proposition}\label{prop:extensionblock}

    Fix an extension of $\hat{\Lambda}(\lambda) \in \Irr(\N_{\hat{\G}^F}(\underline{\Levi})_\la)$ of $\lambda$. Denote by $\hat{\la}$ the restriction of $\hat{\Lambda}(\la)$ to $\hat{\Levi}^F$. Let $\ell$ be an odd prime such that $d_\ell(q)=d$ is odd and denote by $\hat{\la}'$ the restriction of $\hat{\Lambda}(\lambda)$ to $\hat L':=\C_{\hat{\G}^F}(\Z(\underline{\Levi})_\ell^{\underline{F}})$. Then for every $\chi \in \Irr(\G^F)$ there exists an extension $\hat{\chi} \in \Irr(\hat{\G}^F \mid \chi)$ such that $\hat \chi$ lies in the unique block with Brauer pair $(\Z(\underline{\bL})_\ell^{\underline{F}}, b_{\hat L'}(\hat\la'))$ and $\hat{\chi}$ and $\hat{\la}$ have the same field of values over the $\ell$-adic numbers.
\end{proposition}

\begin{proof}
 
 First, recall that $\hat\Lambda(\la)$ exists by Lemma \ref{lem:ext char}.
Let $f_0$ be the maximal power of $2$ such that $q$ can be written $q=p_0^{f_0}$ for some power $p_0$ of $p$. Let $p'$ be the minimal power of $p_0$ such that the centralizing element corresponds to a Sylow $d$-twist as in Lemma \ref{lem:square is twist}. The Dade ramification group of the block $b_G(\Levi,\la)$ is then generated by $F_{p_0'}$ where $p_0'^{\delta}=p'$ for $\delta \in \{1,2 \}$. If $\delta=2$ let $f'$ be such that $p_0^{f'}=p_0'$. It follows from this that the $\ell$-adic integers contain a primitive $2f'$th root of unity (as the order of $p_0'$ modulo $\ell$ is necessarily even). Let $\hat{\G}'$ be the full preimage of $\G \langle F_{p'} \rangle$ under the map $\mathrm{pr}: \underline{\G} \underline{E} \to \G E$. 
Now by Theorem \ref{thm:unipotent characters} combined with Corollary \ref{cor:disjoint HC} (and taking into account Remark \ref{rem:char2} in case $p=2$), we find some extension $\hat{\chi}' \in \Irr(G \langle F_{p'} \rangle)$ of $\chi$ in the block with Brauer pair $(\Z(\underline{\bL})_\ell^{\underline{F}}, b_{\hat L'}(\hat\la'))$  such that $\mathbb{Q}(\hat{\chi}')=\mathbb{Q}(\hat{\lambda}')$ and such that $\hat{\chi}'$ is a constituent of $R_{\hat{\Levi} \cap \hat{\G}' \subset \hat{\Para} \cap \hat{\G}'}^{\hat{\G}'}(\hat{\la}')$. (Note that we have $\mathbb{Q}(\hat{\chi}')=\mathbb{Q}(\hat{\lambda}')$ because of the disjointness properties guaranteed in Theorem \ref{thm:unipotent characters} and Corollary \ref{cor:disjoint HC}, together with the fact that $R_{\hat{\Levi} \cap \hat{\G}' \subset \hat{\Para} \cap \hat{\G}'}^{\hat{\G}'}$ is compatible with Galois automorphisms thanks to \cite[Cor.~8.1.6]{DM20}.)
Here, $\hat{\Levi} \cap \hat{\G}'=\C_{\hat{\G}'}(\Z(\Levi)_\ell^F)$. It follows that there is a character $\hat{\chi}$ covering $\hat\chi'$ in the block of $G \langle F_{p_0} \rangle$ with Brauer pair $(\Z(\underline{\bL})_\ell^{\underline{F}}, b_{\hat L'}(\hat\la'))$, and then $\hat{\chi}$ is a constituent of $R_{\hat{\Levi} \subset \hat{\Para}}^{\hat{\G}}(\hat{\la} \varepsilon)$ for some character ${\varepsilon}$. If $\delta=1$ we can choose $\varepsilon=1$ and the claim follows. 

Assume therefore now that $\delta=2$. Since $\hat{\chi}$ covers $\hat{\chi}'$ it follows that $\varepsilon$ restricts trivially to $\hat{\G}'^{\underline{F}}$ which forces $\varepsilon^{2f'}=1$. On the other hand, as argued above, the $\ell$-adic numbers contain a primitive $2f'$th root of unity. Thanks to \cite[{Thm.~7.6, Sec.~8.A}]{RSST}, the characters $\hat{\chi}$ and $\hat{\la}$ have the same fields of values over the $\ell$-adic numbers. Similarly, $\hat{\la}$ and $\hat{\la} \varepsilon$ differ by a linear character of order dividing $2f'$. In particular, both have the same field of values over the $\ell$-adic numbers.
\end{proof}

We now prove Theorem \ref{thm:AMN}, which we restate here and which solves part of the block distribution problem for the Alperin--McKay--Navarro conjecture. 

\begin{theorem}
Keep the assumptions of the first paragraph of this section. In particular,  $\G$ is simple of classical type, but not of type $\tD_4$, and $\ell$ is an odd prime not dividing $|\Z(G)|$. Further, $(\Levi, \la)$ is a unipotent $d$-cuspidal pair, where $d:=d_\ell(q)$.
Then there exists a block-preserving 
$\mathcal{H}$-equivariant bijection
$$\Irr_{0}(GE \mid \mathcal{E}(G,(\Levi,\la))) \to \Irr_{0}(\N_{GE}(\Levi) \mid \la).$$
\end{theorem}

\begin{proof}
Note that $E$ is abelian by our assumptions on $G$. Further,  
by Lemma \ref{lem:ext char} we know that for any $G \leq H \leq GE$ we can fix an extension $\hat{\Lambda}(\la) \in \Irr(\N_{H}(\Levi)_\la)$ of $\la \in \Irr(L)$. 

We therefore obtain a bijection
$$\Irr_0(H) \to \Irr_0(\N_H(\Levi))$$
by sending the set of characters $\Irr( H \mid R_{\Levi}^{\G}(\la)_\eta)$ (where $R_{\Levi}^{\G}(\la)_\eta$ is the character in $\mathcal{E}(G, (\bL, \la))$ corresponding to $\eta$ under the bijection in \cite[Thm.~3.2]{BMM}) bijectively to the set of characters $\{\Ind_{\N_{H}(\Levi)_\la}^{\N_{H}(\Levi)}(\hat{\La}(\la) \hat{\eta}) \mid \hat{\eta} \in \Irr(W_{H}(\Levi,\la) \mid \eta) \}$. 

Set $\hat{G}:=G \langle F_p \rangle$ and let $b:=b_G(\bL, \la)$. Then either $\hat{G}[b]=(GE)[b]$ or there exists some involution $e\in E$ such that $GE[b]=\hat{G}[b] G \langle e \rangle[b]$.  Now for $\chi \in \Irr(b)$, the two extensions to $G \langle e \rangle$ (and similarly in the local case) have the same field of values. Hence, by applying Lemma \ref{lem:gluing blocks}, it suffices to show the statement for $\hat{G}$ instead of $GE$. Note that the unipotent characters are $\hat{G}$-invariant by \cite[Thm.~4.5.11]{GM20}. 

Set $\hat{N}:=\N_{\hat{G}}(\Levi)$ and $\hat{L}':=\C_{\hat{G}}(\Z(L)_\ell)$, and denote by $\hat{\la}$ the restriction of $\hat{\La}(\la)$ to $\hat{L}' \cap \hat{N}=\hat{L}'$. Note that the character $\psi:=\Ind_{\hat N_\la}^{\hat N}(\hat{\La}(\la) \hat{\eta})$ lies in the unique block of $\N_{\hat{G}}(\Levi)$ that covers the block $b_{\hat{L}'}(\Res_{\hat{L}'}^{\hat{N}_\la}(\hat{\La}(\la) \hat{\eta}))$. 
(Recall that the uniqueness follows by \cite[Cor.~9.21]{Nav98}, and that this unique block is given by block induction.) 
Now, recall that the unique block $b'$ of $N$ covering $b_L(\la)$ induces to the block $b=b_G(\Levi, \la)$ of $G$ by \cite[Thm.~21.7]{CE04}. Then using a generalized Harris--Kn{\"o}rr theorem as in \cite[Thm.~B]{KS15}, the block of $\psi$ induces to $\hat{G}$ and this construction gives a surjection from the blocks of $\hat{N}$ above $b'$ and those of $\hat{G}$ above $b$. Since $\hat{G}/G$ is cyclic, we may then use this and the fact that the other blocks above $b$ are tensors with $b$  to obtain that there is some block-preserving bijection between $\Irr_0(\hat{G} \mid \mathcal{E}(G,(\Levi,\la)) \to \Irr_0(\hat{N} \mid \la)$. It remains to show that such a bijection can be chosen to be $\galh$-equivariant.

Decompose $\hat{G}=\hat{G}_2 \hat{G}_{2'}$ such that $\hat{G}_{\pi}/G$ is a Hall $\pi$-subgroup of $\hat{G}/G$ for $\pi=2$ and $2'$, respectively. (Here by abuse of notation, in this context we let $2$ denote the set $\{2\}$ and $2'$ denote the set of odd primes dividing $|\hat G|$.) Since the decomposition here is coprime, note that Lemma \ref{lem:gluing blocks} applies, and hence it suffices to show the statement for $\hat{G}_2$ and $\hat{G}_{2'}$ individually. That is, it suffices to show for $\pi \in \{2,2'\}$ that there exists block-preserving $\mathcal{H}$-equivariant bijections
$$\Irr_0(\hat{G}_\pi \mid \mathcal{E}(G,(\Levi,\la)) \to \Irr_0(\hat{N}_\pi \mid \la),$$
where $\hat{N}_\pi:=\hat{G}_\pi \cap NE$.

We first consider the $2'$-part. For $\hat G_{2'}$, we observe that every unipotent character is real (even rational) so that there exists a unique real character of $\hat{G}_{2'}$ extending it and this character lies in the unique real block covering $b$, see Lemmas \ref{lem:real} and \ref{lem:real blocks}.  On the local side, note that block induction is compatible with complex conjugation (see Lemma \ref{lem:real blocks}). In particular, the unique block of $\hat{N}_{2'}$ that covers the block $b_{\hat{L}_{2'}}(\Res_{\hat{L}_{2'}}^{\hat{L}'}(\hat{\la}))$ of $\hat L_{2'}:=\hat L'\cap \hat N_{2'}$ (see \cite[Cor.~9.21]{Nav98}) is the unique real block covering $b_L(\la)^N$ (see Lemma \ref{lem:real blocks}). Hence, this bijection can always be chosen such that it preserves the distribution of $\ell$-blocks, {since the quotient is abelian, and hence the remaining blocks are tensors of these real blocks with linear characters.}

We now consider the extensions to $\hat{G}_2$. We have shown in \cite[{Thm.~7.6, Sec.~8.A}]{RSST} that $\mathbb{Q}(\hat{\chi}) \subset \mathbb{Q}_{f_2}$ for every $\hat{\chi} \in \Irr(\hat{G}_2 \mid \mathcal{E}(G,(\Levi,\la))) \cup \Irr(\hat{N}_2 \mid \la)$, where $f$ is the integer such that $q=p^f$ and $f_2$ is the largest power of $2$ dividing $f$. 

Assume first that $d=d_\ell(q)$ is even. In this case, the $\ell$-adic numbers contains all $f_2$th roots of unity {since then $f_2$ divides $d_\ell( p)$, and hence $\ell-1$}. In particular, all extensions of unipotent characters {to $\hat G_2$} are $\mathcal{H}$-invariant. Then here any block-preserving bijection (which exists as discussed for $\hat{G}$ above) works.

We can therefore assume that $d=d_\ell(q)$ is odd.
In this case,  
Proposition \ref{prop:extensionblock} shows that for every $\chi \in \mathcal{E}(G,(\Levi,\la))$ labeled by the character $\eta \in \Irr(W_G(\Levi,\la))$, there exists a $\mathcal{H}$-equivariant bijection $$R_\eta:\Irr(\hat{L}_2 \mid \la) \to \Irr(\hat{G}_2 \mid \chi)$$ such that $(1,b_{\hat{G}_2}(R_\eta(\hat{\la}))) \lhd (\Z(\Levi)_\ell^F,b_{\hat{L}_2})$ where $b_{\hat{L}_2}$ is the unique block of $\hat L_2:=\C_{\hat{G}_2}(\Z(L)_\ell)$ which contains the irreducible constituents of the restriction of $\hat{\la}$.  
By the proof from \cite[Lem.~7.5]{RSST}, we obtain that for $\hat{\la} \in \Irr(\hat{L}_2 \mid \la)$ there exists an extension $\hat{\La}(\hat{\la}) \in \Irr((\hat{N}_2)_\la \mid \hat{\la})$  with $\mathcal{H}_{\hat{\La}(\hat\la)}=\mathcal{H}_{\hat{\la}}$. 
Moreover, as $\hat{N}_2=\N_{\hat{G}_2}(\Z(L)_\ell)$, the characters $\Ind_{(\hat{N}_2)_\la}^{\hat{N}_2}(\hat{\La}(\hat \la) \hat{\eta})$ all lie in the induced block $(b_{\hat{L}_2})^{\hat{N}_2}$ by \cite[Cor.~9.21]{Nav98}. Recall from Remark \ref{rem:weyl central product} that $W_{\hat{G}_2}(\Levi,\la) =W_G(\Levi,\la) \times \hat{L}_2/L$ so that every character $\eta \in \Irr(W_G(\Levi,\la))$ has a unique extension $\eta^\circ \in \Irr(W_{\hat{G}_2}(\Levi,\la))$ with $\hat{L}_2/L$ in its kernel. We therefore obtain that if we extend the map $\Ind_{(\hat{N}_2)_\la}^{\hat{N}_2}(\hat{\La}(\hat\la) \eta^\circ) \mapsto R_{\eta}(\hat\la)$ in an $\Irr(\hat{N}/N)$-equivariant way, we get a bijection
$$\Irr_0(\hat{N}_2 \mid \la ) \to \Irr_0(\hat{G}_2 \mid \mathcal{E}(G,(\Levi,\la)))$$
which is $\mathcal{H}$-equivariant and preserves $\ell$-blocks, as desired.
\end{proof}

\subsection{An Example}

We end the paper by explicitly treating the case where $G=\SL_2(q)$ and $\ell \mid (q-1)$, which in particular illustrates a situation where the Steinberg character lies in the principal block of $G$ but the extended Steinberg character from Definition \ref{def:genstein} does \emph{not} lie in the principal block.

\begin{example}\label{ex:steinberg}
    Suppose that $G=\SL_2(q)$ and $\ell$ is an odd prime with $\ell \mid (q-1)$. Note that the Steinberg character $\St_G$ lies in the principal block $b:=B_0(G)$. Let $\bT_0\leq \B$ be the diagonal torus and corresponding $F$-stable Borel, and write $B:=\B^F$ and $T_0:=\bT_0^F$. Let $p_0$ be minimal such that there is some positive integer $f_0$ with $p_0^{f_0}=q$ an $d_\ell(p_0)=1$. In this case, $\Ind_{B\langle F_{p_0} \rangle}^{G \langle F_{p_0} \rangle}(1_{B \langle F_{p_0} \rangle})$ coincides with generalized Harish-Chandra induction and contains both the trivial character and the generalized Steinberg character $\St_{G\langle F_{p_0}\rangle}$. These are precisely the unipotent characters that lie in the principal block of $G \langle F_{p_0} \rangle$. 
    
    Assume now that there is $p_1$ such that $p_1^2=p_0$, so that $d_\ell(p_1)=2$. In particular, if $n\in \N_G(\T_0) \setminus T_0$ then $n F_{p_1}$ centralizes $(T_0)_\ell$, as $(p_0-1)_\ell=(p_1+1)_\ell (p_1-1)_\ell$. In this situation, the generalized Steinberg character of $G \langle F_{p_1} \rangle$ will not be in the principal block $B_1:=B_0(G \langle F_{p_1}\rangle)$ 
    but in the block $B_1\otimes  \varepsilon$, where $\varepsilon \in \Irr(G \langle F_{p_1} \rangle/G \langle F_{p_0} \rangle)$ is the unique character of order two. This can be computed as follows: Using Theorem \ref{unipotent characters} together with \cite[Theorem 3.11]{DM94} (whose assumptions are satisfied in our case as observed in the remarks following the theorem) shows that generalized Alvis--Curtis duality swaps the blocks $B_1$ and $B_1  \otimes \varepsilon$. ({Indeed, note that $|\T^{nF}|=p+1$ so Alvis--Curtis duality and Lusztig induction do not commute on the coset $G F_{p_0}$ but rather a $-1$ sign is introduced.} See also Section \ref{sec:ACblockswap}). Since generalized Alvis--Curtis duality maps the trivial character to the generalized Steinberg character, it follows that the Steinberg character is not in the principal block of $G \langle F_{p_1} \rangle$. In particular, if $p_2^2=p_1$ then $G \langle F_{p_2}\rangle[b]=G \langle F_{p_1} \rangle$ and therefore the principal block $B_2$ of $G \langle F_{p_2} \rangle$ gets swapped via Alvis--Curtis duality with $B_2 \otimes \varepsilon_2$ where $\varepsilon_2$ is some linear character (necessarily of order $4$) extending $\varepsilon$. Consequently, the characters covering the trivial character are the unique unipotent characters in $B_2$ with real character field. Note that on the other hand $d_{\ell}(p_2)=4$ so that all characters in $B_2$ are $\mathcal{H}$-invariant.

    Compare this to the situation where $\tilde{G}=\GL_2(q)$ in which case $\tilde{G}[B_0(\tilde{G})]=\tilde{G} \langle F_{p_0} \rangle$. In particular, the image of the principal block under generalized Alvis--Curtis duality is always again the principal block.
\end{example}


\begin{thebibliography}{SpSt70}

	\bibitem[Bon06]{B06}
	{\sc C.~Bonnaf\'{e}.}
	\newblock Sur les caract\`eres des groupes r\'{e}ductifs finis \`a centre non
	connexe: applications aux groupes sp\'{e}ciaux lin\'{e}aires et unitaires.
	\newblock {\em Ast\'{e}risque}, (306), 2006.

\bibitem[BMM93]{BMM}
{\sc M. Brou\'e, G. Malle and J.~C.~M. Michel}, Generic blocks of finite reductive groups, {\em Ast\'erisque \bf{212}} (1993), 7--92.



\bibitem[CE94]{CE94}
 {\sc M.~Cabanes and M.~Enguehard}, On unipotent blocks and their ordinary characters.
{\em Invent. Math. \bf{117}} (1994), 149--164.

 \bibitem[CE99]{CE99} {\sc M.~Cabanes and M.~Enguehard}, On Blocks of Finite Reductive Groups and Twisted Induction. {\em Adv. Math. \bf{145}} Issue 2 (1999), 189--229.

 \bibitem[CE04]{CE04}  {\sc M.~Cabanes and M.~Enguehard}, \emph{Representation Theory of Finite  Reductive Groups}.  Cambridge University Press, Cambridge, 2004.

 \bibitem[CS13]{CS13}
 		{\sc M.~Cabanes and B.~Sp\"ath}, Equivariance and extendibility in finite reductive groups with connected center. \emph{Math. Z. \bf275} (2013), 689--713.

	
\bibitem[CS25]{CS25}
		{\sc M.~Cabanes and B.~Sp\"ath},
		The McKay conjecture on character degrees, \emph{Ann. of Math.}, To Appear,
		\newblock Preprint  arXiv:2410.20392 (2024).

\bibitem[DM94]{DM94}
{\sc F.~Digne and J.~Michel},
\newblock Groupes r\'{e}ductifs non connexes.
\newblock {\em Ann. Sci. \'{E}cole Norm. Sup. (4)}, 27(3):345--406, 1994.

\bibitem[DM20]{DM20}
{\sc F.~Digne and J.~Michel},
 {\it Representations of finite groups of Lie type}, second edition, 
London Mathematical Society Student Texts, 95, Cambridge Univ. Press, Cambridge, 2020.

\bibitem[Da73]{dade73}
{\sc E.~C. Dade}, Block extensions, {\em Illinois J. Math. {\bf 17}} (1973), 198--272.

\bibitem[Da77]{dade77}
{\sc E.~C. Dade}, Remarks on isomorphic blocks, {\em J. Algebra {\bf 45}} (1977), no.~1, 254--258.

\bibitem[GM20]{GM20}
{\sc M. Geck and G. Malle},
\newblock {\em The Character Theory of Finite Groups of {L}ie Type}, volume 187
of {\em Cambridge Studies in Advanced Mathematics}.
\newblock Cambridge University Press, Cambridge, 2020.

\bibitem[GSV18]{GSV}
{\sc E. Giannelli, A. A. Schaeffer Fry, and C. Vallejo Rodr{\'i}guez}.
{ \em Proc. Amer. Math. Soc. \bf{147}} (2019), 4697--4712.

\bibitem[G]{gow}
{\sc R. Gow}, Real $2$-blocks of characters of finite groups, {\em Osaka J. Math. {\bf 25}} (1988), no.~1, 135--147.


\bibitem[Hi90]{hiss}
{\sc G.~Hiss},
\newblock Regular and semisimple blocks of finite reductive groups.
\newblock{\em J. Lond. Math. Soc. \bf{s2-41}}:1  (1990), 63--68.



\bibitem[Isa76]{isaacs}
		{\sc I.M.~Isaacs}, \emph{Character Theory of Finite Groups}. Academic Press, New York, 1976.
     

\bibitem[Joh22b]{Joh22b}
{\sc B. Johansson}, On the inductive McKay--Navarro condition for finite groups of Lie type in their
defining characteristic. \emph{J. Algebra \bf610}  (2022), 223--240.


	\bibitem[KM15]{KM15}
	{\sc R.~Kessar and G.~Malle}.
	\newblock Lusztig induction and {$\ell$}-blocks of finite reductive groups.
	\newblock {\em Pacific J. Math. \bf{279}} (2015), no.~1-2, 269--298.

\bibitem[KS15]{KS15}
{\sc S. Koshitani and B. Sp\"ath}, Clifford theory of characters in induced blocks, {\em Proc. Amer. Math. Soc. {\bf 143}} (2015), no.~9, 3687--3702.
    


\bibitem[LMNT]{lmnt}
{\sc R.  Lyons, J. M. Mart{\'i}nez, G. Navarro, and P. H. Tiep},
\newblock
Principal blocks, irreducible restriction, fields and degrees. (2025)
\newblock arXiv:2507.22503

\bibitem[Ma07]{Ma07}	
{\sc G.~Malle}, Height 0 characters of finite groups of Lie type. \emph{Represent. Theory \bf11} (2007), 192--220.



	\bibitem[Mar96]{Marcus}
	{\sc A. Marcus}.
	\newblock On equivalences between blocks of group algebras: reduction to the
	simple components.
	\newblock {\em J. Algebra}, 184(2):372--396, 1996.


    \bibitem[Mur13]{Murai}
	{\sc M. Murai}.
	\newblock On blocks of normal subgroups of finite groups.
	\newblock {\em Osaka J. Math.}, 50(4):1007--1020, 2013.
	
    \bibitem[Na98]{Nav98}
	{\sc G. Navarro}.
	\newblock {\em Characters and blocks of finite groups}, volume 250 of {\em
		London Mathematical Society Lecture Note Series}.
	\newblock Cambridge University Press, Cambridge, 1998.
        
        \bibitem[Na04]{Na04}
		{\sc 		G. Navarro}, The McKay conjecture and Galois automorphisms. \emph{Ann. of Math. (2) \bf{160}(3)} (2004), 1129--1140.

        \bibitem[Na18]{Nav18}{\sc G.~Navarro}, \emph{Character Theory and the McKay Conjecture}. Cambridge University Press, Cambridge 2018. 
        


\bibitem[NT10]{nt10}
{\sc G. Navarro and P.~H.~Tiep}, Degrees of rational characters of finite groups, {\em Adv. Math. {\bf 224}} (2010), no.~3, 1121--1142.

\bibitem[NT13]{NT13}
{\sc G. Navarro and P.~H.~Tiep},
Characters of relative $p'$-degree over normal subgroups, {\em Ann. of Math. (2) {\bf 178}} (2013), no.~3, 1135--1171.



\bibitem[NTT08]{NTT} 
{\sc G. Navarro, P.~H.~Tiep and A. Turull}, Brauer characters with cyclotomic field of values, {\em J. Pure Appl. Algebra {\bf 212}} (2008), no.~3, 628--635.

\bibitem[Ros24]{Ros24}
{\sc D. Rossi},
A local-global principle for unipotent characters.
{\em Forum Math. Sigma \bf{12}} (2024), e125, 1--29.


\bibitem[RSF22]{RSF22}		
{\sc L. Ruhstorfer and A.~A.~ Schaeffer Fry}, The inductive McKay--Navarro conditions for the prime $2$ and some groups of Lie type. \emph{ Proc. Amer. Math. Soc. Ser. B, \bf{9}} (2022), 204--220. (Minor corrections at arXiv:2106.14745) 



\bibitem[RSST25]{RSST}		
{\sc L. Ruhstorfer, A.~A.~Schaeffer Fry,  B. Sp{\"a}th, and J. Taylor},  Towards the inductive Galois-McKay Condition for groups of Lie type. arXiv:2506.17123.



\bibitem[SFT23]{SFT23}
{\sc A.~A.~Schaeffer Fry and J.~Taylor},
Galois automorphisms and classical groups.  {\em Transform. Groups \bf{28}} (2023), 439--486.



\bibitem[S92]{schmid}
{\sc P.~P. Schmid}, Extending the Steinberg representation, {\em J. Algebra {\bf 150}} (1992), no.~1, 254--256.



\bibitem[Sp10a]{Spa10a}
{\sc B.~Sp{\"a}th},
\newblock Sylow {$d$}-tori of classical groups and the {M}c{K}ay conjecture. {I}.
\newblock {\em J. Algebra} {\bf 323} (2010), 2469--2493.
        
\bibitem[Sp10b]{spath10}
{\sc B. Sp{\"a}th}, 
Sylow $d$-tori of classical groups and the McKay conjecture, II.
\emph{J. Algebra \bf{323}} (2010), 2494--2509.

\bibitem[Sp12]{Spa12}
{\sc B. Sp{\"a}th},
Inductive McKay condition in defining characteristic.
{\em Bull. London Math. Soc. \bf{44}} (2012), 426--438.



\bibitem[SV20]{SV20}		
{\sc B.~Srinivasan and C.~R.~Vinroot}, Galois group action and Jordan decomposition of characters of finite reductive groups with connected center. \emph{J. Algebra} \textbf{558}, 708--727 (2020). 




\end{thebibliography}
\end{document}